\documentclass[12pt]{article}

\usepackage{amsmath}

\usepackage{geometry}
\numberwithin{equation}{section}
\usepackage{amsfonts}
\usepackage{amsthm}

\newtheoremstyle{annals}
  {10pt plus 2pt minus 2pt} 
  {10pt plus 2pt minus 2pt} 
  {\itshape}                
  {}                        
  {\bfseries}               
  {.}                       
  {0.5em}                   
  {}                        
\theoremstyle{annals}

\usepackage{amssymb}
\usepackage{bbm}
\usepackage{centernot}
\usepackage[x11names]{xcolor}
\usepackage{chemfig}
\usepackage{amsmath}
\usepackage{cite}
\usepackage[nodayofweek]{datetime} 
\usepackage{dsfont}
\usepackage{enumitem}
\usepackage{euscript}
\usepackage{faktor}
\usepackage{float}
\usepackage{geometry}
\usepackage{graphicx}
\usepackage{mathrsfs}
\usepackage{mathtools}
\usepackage{polynom}
\usepackage{stmaryrd}
\usepackage{tikz}
\usepackage{tikz-cd}
\usepackage{wasysym}
\usepackage{xfrac}
\usepackage[x11names]{xcolor}
\usepackage{mathtools}
\usepackage{setspace}
\usepackage[all, cmtip]{xy}
\usepackage{comment}
\usepackage{pgfplots}
\pgfplotsset{compat=1.15}
\usepackage{mathrsfs}
\usetikzlibrary{arrows}
\usepackage{thmtools}
\usepackage{xcolor}
\usepackage[hidelinks]{hyperref}

\usepackage{microtype}

\usepackage{amssymb}
\usepackage{nomencl}
\makenomenclature

\usepackage[T1]{fontenc}
\usepackage{esint}

\newlist{legal}{enumerate}{10}
\setlist[legal]{label*=\arabic*.}

\newif\ifproofread

\DeclarePairedDelimiter\abs{\lvert}{\rvert}

\newcommand\norm[1]{\left\lVert#1\right\rVert}

\makeatletter
\newcommand*\bigcdot{\mathpalette\bigcdot@{.5}}
\newcommand*\bigcdot@[2]{\mathbin{\vcenter{\hbox{\scalebox{#2}{$\m@th#1\bullet$}}}}}
\makeatother

\usepackage{mathtools}
\DeclarePairedDelimiter\ceil{\lceil}{\rceil}
\DeclarePairedDelimiter\floor{\lfloor}{\rfloor}

\newtheorem*{Lemma*}{Lemma}
\newtheorem*{Corollary*}{corollary}
\newtheorem*{theorem*}{Theorem}
\newtheorem*{definition*}{Definition}

\newtheorem{theorem}{Theorem}[section]
\newtheorem{nota}[theorem]{Notation}
\newtheorem{definition}[theorem]{Definition}

\newtheorem{remark}[theorem]{Remark}

\newtheorem{lemma}[theorem]{Lemma}

\newtheorem{dichotomy}[theorem]{Dichotomy}
\newtheorem{proposition}[theorem]{Proposition}
\newtheorem{corollary}[theorem]{Corollary}
\newtheorem{claim}[theorem]{Claim}

\theoremstyle{remark}

\declaretheoremstyle[
  headfont=\normalfont\bfseries,
  numbered=unless unique,
  bodyfont=\normalfont,
  spaceabove=1em plus 0.75em minus 0.25em,
  spacebelow=1em plus 0.75em minus 0.25em,
]{hartending}

\newcommand{\thistheoremname}{}
\newtheorem*{genericthm}{\thistheoremname}

\newcommand{\boundaryFunc}{\delta c^{-m}}

\newcommand{\NN}{\mathbb{N}}

\newcommand{\RR}{\mathbb{R}}

\newcommand{\eps}{\epsilon}

\newcommand{\fg}{\mathfrak{g}}
\newcommand{\fv}{\mathfrak{v}}

\newcommand{\parpar}[1]{\partial_{#1}^{\mathbin{\|}}}

\newcommand{\Aa}{\mathcal{A}}

\newcommand{\fa}{\mathfrak{a}}

\newcommand{\Dd}{\mathcal{D}}

\newcommand{\Ee}{\mathcal{E}}

\newcommand{\Ff}{\mathcal{F}}

\newcommand{\Mm}{\mathcal{M}}

\newcommand{\Pp}{\mathcal{P}}
\newcommand{\fp}{\mathfrak{p}}
\newcommand{\Qq}{\mathcal{Q}}

\newcommand{\vV}{\mathfrak{v}}

\newcommand{\stab}{\operatorname{stab}}

\newcommand{\bd}{{\rm d}}

\newcommand{\on}[1]{\operatorname{#1}}

\newcommand{\rdreg}[3]{{(#1)}^{#3}{#2}^{-#3}}
\newcommand{\rilip}{e^{Ri}}

\newcommand{\g}[1]{\mathcal G\left(#1\right)}
\newcommand{\inn}[1]{\left<#1\right>}

\newcommand{\set}[1]{\left\{#1\right\}}

\newcommand{\pid}{\mathrel{\ooalign{$\lneq$\cr\raise.22ex\hbox{$\lhd$}\cr}}}

\newlength{\leftstackrelawd}
\newlength{\leftstackrelbwd}
\def\leftstackrel#1#2{\settowidth{\leftstackrelawd}%
{${{}^{#1}}$}\settowidth{\leftstackrelbwd}{$#2$}%
\addtolength{\leftstackrelawd}{-\leftstackrelbwd}%
\leavevmode\ifthenelse{\lengthtest{\leftstackrelawd>0pt}}%
{\kern-.5\leftstackrelawd}{}\mathrel{\mathop{#2}\limits^{#1}}}

\AtBeginDocument{}

\DeclareMathOperator{\diam}{diam}

\DeclareMathOperator{\GL}{GL}

\DeclareMathOperator{\Ad}{Ad} 
\DeclareMathOperator{\dist}{dist}

\DeclareMathOperator{\Stab}{Stab}

\newcommand{\TV}{\mathrm{TV}}

\newcommand{\yynote}[1]{\marginpar{\color{cyan}\tiny [YY] #1}}

\newcommand{\Wb}{W^{\mathrm{b}}_{1}}
\newcommand{\Wf}[1]{W_{1}^{#1}}

\usepackage{setspace}
\title{The Structure of Almost Stationary Measures}
\date{}
\author{
  Ilya Gekhtman \footnote{\texttt{ilyagekh@gmail.com}}, Simon Machado \footnote{\texttt{smachado@ethz.ch}}, Omri Solan \footnote{\texttt{os4293@princeton.edu}} and Yuval Yifrach \footnote{\texttt{yuval.yifrach@math.uzh.ch}}
}
\Huge
\begin{document}
\maketitle

\begin{abstract}
	Let $G$ be a higher-rank simple Lie group acting on a space $X$. A
	theorem of Nevo and Zimmer asserts that every ergodic stationary
	probability measure on $X$ is either $G$-invariant or admits a projective
	factor $G/Q$ for a proper parabolic subgroup $Q$.

	We develop a quantitative theory of stationary measures and prove an
	effective form of this dichotomy. We introduce notions of $\eps$-almost
	stationarity, $\delta$-almost invariance and $\delta$-almost projective
	factor, and show that every $\eps$-almost stationary measure is either
	$\delta$-almost invariant or carries a $\delta'$-almost projective factor,
	with $\delta,\delta'$ explicit in $\eps$ and depending only on $G$. No
	ergodicity, arithmeticity or Diophantine hypothesis is imposed, and the
	bounds are uniform over all $G$-spaces.
    Additionally, we prove a decomposition version of this theorem where each measure is decomposed into an almost invariant part and an almost projective part. This comes at the cost of worse constants. 

	The proof introduces several tools: the \emph{entropigeonhole method}, an
	entropy-based pigeonhole principle yielding a quantitative Mautner
	phenomenon; \emph{factor functions}, quantitative analogues of functions on
	homogeneous factor spaces; and a \emph{fast generation} dichotomy in the
	spirit of growth in groups. In a companion paper these are used to show,
	among other things, that a discrete subgroup of infinite covolume has
	injectivity radius at least $c\log^{(4)}r$ somewhere in the ball of radius
	$r$, which is an effective form of a theorem of Fr\k{a}czyk and Gelander. We also prove rates for Benjamini--Schramm convergence of quotients of higher-rank lattices.
\end{abstract}

\clearpage
{\small\singlespacing
\setcounter{tocdepth}{2}
\tableofcontents
}

\clearpage
\section{Introduction}

Invariant measures are the basic object of ergodic theory and a remarkable source of structure, but for actions of
non-amenable groups they are hard to construct and need not exist. A natural
substitute is provided by \emph{stationary measures}, that is, probability
measures invariant under averaging along a random walk. Unlike invariant
measures, stationary measures come for free: if $\mu$ is a probability measure
on a locally compact group $G$ acting on a compact space $X$, then every
weak-$^*$ limit of the Ces\`aro averages
\begin{equation}\label{eq: cezaro average}
	\nu_n=\frac1n\sum_{i=0}^{n-1}\mu^{*i}*\delta_x
\end{equation}
is $\mu$-stationary.

Stationarity is therefore cheap, while invariance is the source of rigidity.
All the content of the theory has to be recovered afterwards, from a structure
theorem describing how far a stationary measure can be from an invariant one.
For higher-rank semisimple Lie groups such a theorem is available, and it is
remarkably strong.

\begin{theorem}[Nevo--Zimmer \cite{NZ}]\label{thm: NZ}
	Let $(X,\nu)$ be a probability measure space equipped with an action of $G$,
	a higher-rank semisimple Lie group without rank-one factors. If $\nu$ is
	$\mu$-stationary and ergodic, then one of the following holds:
	\begin{itemize}
		\item \emph{(Invariance)} the measure $\nu$ is $G$-invariant;
		\item \emph{(Projective factor)} there is a parabolic subgroup
		$Q\leq G$ and a $G$-equivariant map $\pi:X\rightarrow G/Q$ pushing
		$\nu$ to the standard stationary measure on $G/Q$.
	\end{itemize}
\end{theorem}

The theorem makes no assumption whatsoever on the space acted upon, and
identifies a finite list of algebraically prescribed models accounting for every
obstruction to invariance. The list originates in Furstenberg's identification
of the Poisson boundary of $G$ with a homogeneous space $G/Q$
\cite{FurstenbergPoisson,Furstenberg73}; the content of Theorem~\ref{thm: NZ} is that
these boundaries are not merely a source of stationary measures but the only
one.

Theorem~\ref{thm: NZ}, together with the classification of measure-preserving
actions of $G$ due to Stuck and Zimmer \cite{Stuck-Zimmer}, reaches geometry
through a recurring pattern. One builds a sequence of measures on the Chabauty
space of closed subgroups of $G$ out of geometric data, extracts a weak-$^*$
limit, checks that the limit is stationary, and concludes that it must be
$\delta_{\{e\}}$. This is how \cite{Samurai} shows that normalised Betti numbers
of lattices in $G$ decay as the covolume grows, and how Fr\k{a}czyk and Gelander
\cite{Fraczyk-Gelander} show that a discrete subgroup of infinite covolume has
unbounded injectivity radius in its quotient.

Both conclusions invite a rate. How fast do the Betti numbers decay? How large
an injectivity radius is guaranteed within distance $r$ of a basepoint? The
method above cannot answer either question, and the obstruction is structural
rather than technical: stationarity is a property of the limit alone, so
whatever quantitative information the approximating sequence carried is lost at
precisely the point where the structure theorem becomes available.

Our starting point is that this loss is unnecessary. The construction
\eqref{eq: cezaro average} is quantitative before any limit is taken: the
measures $\nu_n$ are not stationary, but they satisfy
\begin{equation}\label{eq: cezaro almost stationary}
	\norm{\mu*\nu_n-\nu_n}_{TV}\leq \frac{2}{n},
\end{equation}
by a two-line computation requiring no hypothesis on $X$ or on $\mu$. This
suggests that the right object of study is not the stationary measure but the
\emph{almost} stationary one: a measure satisfying
$\norm{\mu*\nu-\nu}_{TV}\leq\eps$, with stationarity recovered as the limiting
case $\eps=0$. On such measures the classical dichotomy has no meaning as
stated, and the question is whether it has a quantitative shadow, with
conclusions degrading in a controlled way as $\eps\to0$.

\emph{It does}. This paper is the first of a series on the \emph{almost structure of
measures}, and develops a quantitative structure theory for almost stationary
measures under actions of higher-rank semisimple Lie groups. We introduce
effective notions of stationarity, invariance and projective factor, and prove
that almost stationary measures obey a quantitative form of the Nevo--Zimmer
dichotomy. The following informal statement captures the main result; the
precise version is Theorem~\ref{thm: qnz}. Throughout we normalise the action so
that orbit maps $g\mapsto gx$ are $1$-Lipschitz, which costs nothing beyond
rescaling the metric on $X$.

\begin{theorem}[Quantitative Nevo-Zimmer, informal]\label{thm:A}
	Let $G$ be a higher-rank simple Lie group acting on a metric space
	$X$ with $1$-Lipschitz orbits, let $\mu$ be the time-one heat kernel on $G$, and let
	$f:X\rightarrow\mathbb{R}$ be a $1$-Lipschitz observable whose sup norm is at most $1$. There are
	constants $C,c>0$ depending only on $G$ with the following property.

	Let $\nu$ be a probability measure on $X$ that is $\eps$-almost stationary,
	that is $\norm{\mu*\nu-\nu}_{TV}\leq\eps$, and let $\eta\in(0,1]$. Then at
	least one of:
	\begin{itemize}
		\item \emph{(Almost invariance)} the distribution of $f$ is almost
		unmoved by the unit ball of $G$:
		\[
			W_1\bigl(f_*(g.\nu),\,f_*\nu\bigr)\leq
			C\bigl(\abs{\log\eps}^{-c}+\eta^{1/4}\bigr)
			\qquad\text{for every }g\in G_1
		\]
        where $W_1$ denotes the $1$-Wasserstein distance;
		\item \emph{(Almost projective factor)} there is a set $W\subset X$ with
		$\nu(W)\geq\eta$, a scale $R=R(\eps)$ tending to infinity and an error
		$\delta=\delta(\eps)$ tending to $0$ as $\eps\to0$, such that for every
		$x\in W$ the stabilizer $\stab_G(x)$ is contained, up to an error
		$\delta$ and out to scale $R$, in a parabolic subgroup $Q_x\leq G$.
	\end{itemize}
\end{theorem}

The parameter $\eta$ governs a genuine trade-off rather than a defect of the
proof: making $\eta$ small strengthens the first alternative and weakens the
second, as it must, since a measure may well be a mixture in which only a small
part sits on a projective factor.

In Section \ref{sec: decomposition qnz} we prove a dichotomy-free version of Theorem \ref{thm: qnz}, stated as Theorem \ref{thm:ergodic-qnz}.
This version gives a decomposition structure to any almost stationary measure and tracks the full mass of the measure.
The proof of this version is a straightforward application of Theorem \ref{thm: qnz} using a successive 'pealing' of the mass, where each step uses Theorem \ref{thm: qnz}.
After each application of Theorem \ref{thm: qnz}, we are left with part of the mass which is not accounted for. 
We turn this part into yet another almost stationary measure (with a worse constant), and apply again (inductively) Theorem \ref{thm: qnz} to this part.
We record here the fact that each application of Theorem \ref{thm: qnz} introduces another logarithm (due to its appearance in the first case in the dichotomy), which results in a multi-logarithm result.
For more details, see the proof of Theorem \ref{thm:ergodic-qnz}.


The theorem is already sharp enough for geometry. In the companion paper
\cite{QFG} we apply it to the Ces\`aro averages \eqref{eq: cezaro average} on
$G/\Gamma$, for $\Gamma\leq G$ discrete of infinite covolume. Since these
measures are $2/n$-almost stationary for free, the dichotomy applies at every
finite scale, and one obtains an effective form of the Fr\k{a}czyk--Gelander
theorem: the maximal injectivity radius on the ball $B_r^{X/\Gamma}(o)$ is at
least $c\log^{(4)}r$, with $c$ depending only on $\Gamma$ and $o$. This answers
the question raised in \cite[Remark~1.4(ii)]{Fraczyk-Gelander}, and it yields a
geometric characterisation of lattices in higher rank: if the injectivity radius
in $X/\Gamma$ grows more slowly than $\log^{(4)}$, then $\Gamma$ is already a
lattice. The same circle of ideas gives a new proof of the Stuck--Zimmer theorem and provides rates for the Benjamini--Schramm convergence of finite-volume locally symmetric spaces to the associated symmetric space. We
return to these in Section~\ref{subsec: outlook}.

\subsection{Context}\label{subsec: context}

No effective form of Theorem~\ref{thm: NZ} appears to have been known. The
quantitative results available for random walks on homogeneous spaces classify
or equidistribute orbits rather than describe almost stationary measures, and
each of them pays for its rate with an arithmetic or Diophantine hypothesis. In
the work of Bourgain, Furman, Lindenstrauss and Mozes on stiffness for toral
automorphisms \cite{BFLM}, quantitative equidistribution holds away from a
Diophantine exceptional set; the effective equidistribution theorems for
unipotent flows of Lindenstrauss, Mohammadi and Wang \cite{LMW} and of
Lindenstrauss, Margulis, Mohammadi and Shah \cite{LMMS} rest on arithmeticity
and spectral gap; the recent theorems of B\'enard and He \cite{Benard-He} give
effective local limit and equidistribution statements for random walks on
semisimple groups under moment and Diophantine assumptions.

Theorem~\ref{thm:A} imposes no condition of this kind on the space acted upon, and this is the feature
we would single out:
\begin{quote}
	the bounds depend only on $G$, and are uniform over all $G$-spaces
	and all $\eps$-almost stationary measures on them.
\end{quote}
Nothing in the qualitative theory suggests that such uniformity should be
available. A second constraint on the use of Theorem~\ref{thm: NZ} is the
ergodicity and mixing hypotheses it requires, which can be hard to verify: in
their study of lattice actions on manifolds of low dimension, Brown, Rodriguez
Hertz and Wang \cite{Brown-RodriguezHertz-Wang} observe as much, and establish a
dichotomy of Nevo--Zimmer type by different means. Theorem~\ref{thm:A} carries
no ergodicity hypothesis. 
To temper this, not that the algebraic and Diophantine assumptions on the space lead to theorems relevant even for measures with finite or singular support. This paper, on the contrary, concerns measures that are almost stationary with respect to the heat kernel. While this seems sufficient for most applications, we hope to remove this assumption using precise enough local limit theorems in the spirit of \cite{Koegler}.

Read from a distance, this paper belongs to a familiar tradition. Approximate
subgroups stand to subgroups as almost stationary measures stand to stationary
ones, and the structure theory of Breuillard, Green and Tao \cite{BGT} and of
Hrushovski \cite{Hrushovski} shows that a purely approximate hypothesis can
still force an algebraic conclusion; the same is true of almost representations,
where Kazhdan's $\eps$-representations \cite{Kazhdan-eps} and the stability
results of Becker, Lubotzky and Thom \cite{BLT} recover exact algebraic objects
from approximate ones, and of approximate lattices \cite{Machado-approx}. In
each case the finitary notion turns out to have content of its own rather than
merely approximating the classical one, and the tools built to handle it outlive
the first application. We expect the same here, and
Section~\ref{sec: methods} is written with that expectation in mind.

\begin{remark}[On the rate]\label{rem: rate}
	The rate we obtain is polynomial in $\abs{\log\eps}$ and is certainly not
	optimal; a polynomial dependence $\delta(\eps)=\eps^{c}$ would be the
	natural conjectural target, and is excluded by no example we know. Two
	points temper this. First, the rate we prove is already sufficient for the
	applications in \cite{QFG}. Second, it is not obvious that any effective
	rate should hold. The Ces\`aro construction yields approximate stationarity
	with no hypotheses, but the content of Theorem~\ref{thm:A} is effective
	\emph{rigidity}, an explicit dichotomy between invariance and a projective
	factor, and quantitative results of that shape normally require algebraic,
	arithmetic or Diophantine input. Here none is imposed.
\end{remark}

\subsection{Almost stationarity, almost invariance, almost factors}
\label{subsec: definitions}

Making the Nevo--Zimmer dichotomy effective requires effective substitutes for
its three protagonists. None of the three has a canonical quantitative
counterpart, and the choices made here determine what the resulting theorem is
good for. Two principles guide them: the quantitative notion should be
\emph{verifiable in practice} on the measures one actually meets, and should
\emph{degenerate to the classical notion} in the limit. 

\subsubsection{Setup and conventions}\label{sssec: setup}

Throughout the paper $G$ is a connected noncompact simple Lie group with finite
centre and real rank at least two, and $K\leq G$ is a maximal compact subgroup.
We fix a left-invariant Riemannian metric on $G$ and write
$G_R=\{g\in G: d(e,g)\leq R\}$. We also fix a faithful finite-dimensional
representation of $G$ and write $G_R^{\|}$ for the corresponding \emph{norm
ball} of radius $R$; see Section~\ref{sec: nota}. The two families are
related by $G_{c\log R}\subset G^{\|}_R\subset G_{C\log R}$, and the reader will
lose nothing on a first pass by reading $G_R^{\|}$ as $G_{\log R}$. Norm balls
are the natural scale here because the quantities we control degrade
polynomially in the norm and only exponentially in the Riemannian radius.

By a \emph{$G$-space} we mean a metric space $X$ with a continuous
action of $G$ by homeomorphisms. We insist that the space $X$ need not be compact in the full generatlity of Theorem \ref{thm: qnz}. We write $\Mm_1(X)$ for the Borel probability measures on $X$,
$g.\nu$ for the pushforward of $\nu$ under $x\mapsto gx$, and
$\mu*\nu=\int_G g.\nu\,d\mu(g)$ for $\mu\in\Mm_1(G)$.

\subsubsection{Almost stationarity and almost invariance}
\label{sssec: stat and inv}

Two metrics on measures are used. The differences matter. The first is the
total variation norm $\norm{\cdot}_{TV}$, which is insensitive to the metric on
$X$. The second is the bounded transport distance
\begin{equation}\label{eq: def Wb}
	\Wb(\nu_1,\nu_2)
	=\sup\bigl\{\abs{\nu_1(\varphi)-\nu_2(\varphi)}:\varphi\in C(X),\
	\norm{\varphi}_{\infty}\leq1,\ \mathrm{Lip}(\varphi)\leq1\bigr\},
\end{equation}
which is insensitive to sets of small measure. Deleting the constraint
$\norm{\varphi}_\infty\leq1$ turns $\Wb$ into the $1$-Wasserstein distance
$W_1$, by Kantorovich--Rubinstein duality. The two agree up to a factor
$\max(1,\tfrac12\diam X)$ and coincide on spaces of diameter at most $2$
(Section~\ref{sec: nota}). The distinction becomes more relevant when dealing with spaces of infinite diameter. For measurable $f:X\rightarrow \RR$ with $\norm{f}_\infty\leq1$ we
write $\Wf{f}(\nu_1,\nu_2)=\Wb(f_*\nu_1,f_*\nu_2)$.  Here $f(X)$ has diameter at
most $2$, so this is \emph{exactly} the $1$-Wasserstein distance between the
distributions of the observable $f$ under $\nu_1$ and $\nu_2$, which is the form
in which the applications use it.

The definition of almost stationarity is hinted at by
\eqref{eq: cezaro almost stationary}: the correct effective form of stationarity
is the one the Ces\`aro averages already provide, before any limit is taken.

\begin{definition}[Almost stationarity and almost invariance]
\label{def: almost stat and inv}
	Let $\eps>0$, let $\nu\in\Mm_1(X)$ and let $f:X\rightarrow Y$ be measurable
	with $\norm{f}_\infty\leq1$. We say that $\nu$ is
	\begin{enumerate}
		\item \emph{$\eps$-almost stationary} with respect to
		$\mu\in\Mm_1(G)$ if $\norm{\mu*\nu-\nu}_{TV}\leq\eps$;
		\item \emph{$(f,\eps)$-almost invariant} if $\Wf{f}(g.\nu,\nu)\leq\eps$
		for every $g\in G_1$.
	\end{enumerate}
	Variants of the second notion, in which $\Wf{f}$ is replaced by $\Wb$ or by
	$\norm{\cdot}_{TV}$, are recorded in Section~\ref{sec: nota} and used in
	the proof.
\end{definition}

Almost stationarity is stated in total variation, the strongest reasonable
choice, which makes the hypothesis demanding but is precisely the form in which
\eqref{eq: cezaro almost stationary} delivers it.

Almost invariance, by contrast, is stated relative to an observable. This enables us to extend the theorem's scope to spaces that are not of finite diameter or compact. 

\begin{remark}[Why two different metrics]\label{rem: why two metrics}
	One would like a dichotomy stated entirely in total variation, and this is
	impossible. Total variation is blind to the size of a group element: for
	$\nu=\delta_x$ one has $\norm{g.\nu-\nu}_{TV}=2$ whenever $gx\neq x$,
	however close $g$ is to the identity, so no quantitative notion of
	invariance survives there. Since orbit maps are $1$-Lipschitz, the coupling
	$x\mapsto(x,gx)$ gives instead $\Wb(g.\nu,\nu)\leq W_1(g.\nu,\nu)\leq
	d(g,e)$ for every $g$ and every $\nu$, which is what makes almost invariance
	a meaningful hypothesis, and what makes the normalisation on orbit maps the
	right one rather than a technical convenience. Thus $\Wb$ is the coarsest
	metric for which the conclusion can hold, and total variation the finest one
	in which the hypothesis is available for free from
	\eqref{eq: cezaro almost stationary}: the theorem carries the weakest
	hypothesis and the strongest conclusion the problem permits, which is what one
	wants from a tool.
\end{remark}

\subsubsection{Almost projective factors}\label{sssec: almost factors}

The remaining notion is the delicate one, because the existence of a projective
factor is an assertion about an equivariant map, and equivariance is not a
quantitative condition. Our way around this is to replace the factor map by its
shadow on stabilisers.

Suppose $\pi:(X,\nu)\rightarrow(G/Q,m_{G/Q})$ is a $G$-map with $Q\leq G$
parabolic. For $\nu$-almost every $x$ we have
$\stab_G(x)\subset\stab_G(\pi(x))$, and writing $\pi(x)=gQ$ this reads
\begin{equation}\label{eq: anchor}
	\stab_G(x)\subset gQg^{-1}.
\end{equation}
Conversely, \eqref{eq: anchor} very nearly characterises the situation: for the
purposes at hand, a measurable assignment $x\mapsto gQg^{-1}$ satisfying
\eqref{eq: anchor} is the same thing as a projective factor. And
\eqref{eq: anchor} is a containment of subsets of $G$, which \emph{does} admit
an honest quantitative relaxation: we ask that the stabiliser be uniformly close
to a conjugate of $Q$ rather than contained in it.

\begin{definition}[One-sided Hausdorff distance]\label{def: dh+}
	For subsets $A,B$ of a metric space, $d_H^+(A,B)=\sup_{a\in A}d(a,B)$.
\end{definition}

Let $\Qq$ denote the space of maximal parabolic subgroups of $G$, viewed as a
space of closed subsets of $G$ with the conjugation action. Since $G$ has finite
centre, $\Qq$ is a compact $G$-space, a finite disjoint union of the projective
varieties $G/Q$.

\begin{definition}[Almost projective factor]\label{def: almost factor}
	Let $X$ be a $G$-space, $W\subset X$ measurable, and $R,\delta>0$. We say
	that $W$ \emph{has an $(R,\delta)$-projective factor} if there is a
	measurable map $\pi:W\rightarrow\Qq$ such that for every $x\in W$
	\begin{equation}\label{eq: almost anchor}
		d_H^+\bigl(\stab_G(x)\cap G^{\|}_{R},\,\pi(x)\bigr)\leq\delta.
	\end{equation}
	When $W=\{x\}$ we say that the action has an $(R,\delta)$-projective factor
	at $x$.
\end{definition}

\begin{remark}[The role of the scale $R$]\label{rem: role of R}
	The truncation $\stab_G(x)\cap G_R^{\|}$ is not a technicality. Stabilisers
	are typically unbounded, and no quantitative statement can control all of an
	unbounded set at once; $R$ records how far out along the stabiliser the
	approximation is claimed. Consequently \eqref{eq: almost anchor} is
	informative only when $R$ is large \emph{and} $\delta$ is small, and it is
	vacuous in either degenerate regime: for $R$ close to $1$ the left-hand side
	is close to $0$ for any $\pi$, and for $\delta$ comparable to the diameter
	of $G_R^{\|}$ it holds trivially.  Note also that the distance is one-sided, so
	that \eqref{eq: almost anchor} asks the stabiliser to be almost contained in
	$\pi(x)$ rather than close to it. Two-sided closeness would already fail for
	the trivial stabiliser. Finally, $\pi$ is not required to be equivariant,
	nor to be defined on all of $X$. In the classical statement equivariance is
	automatic once \eqref{eq: anchor} holds on a conull set. Here $\pi$
	is produced pointwise on a set of measure bounded below.
\end{remark}

\subsubsection{The random walk}\label{sssec: random walk}

Almost stationarity is a hypothesis relative to a step measure, which we now
fix.

\begin{definition}\label{def: the measure mu intro}
	Let $p_t$ denote the $K$-bi-invariant heat kernel on $G$, that is, the lift
	to $G$ of the heat kernel at time $t$ on the symmetric space $K\backslash G$,
	normalised to be a probability density with respect to Haar measure. We set
	$\mu:=p_1$.
\end{definition}

The choice appears less canonical than it is. The heat kernel enters only
through the quantitative Furstenberg decomposition of Section~\ref{sec: QF}, and
we expect Theorem~\ref{thm: qnz} to hold for any sufficiently smooth admissible
step measure, the general case following from a local limit theorem
approximating high convolution powers by the heat kernel at an appropriate time,
in the style of Koegler \cite{Koegler}. We do not pursued this here: the paper is
long enough, and the applications in \cite{QFG} are insensitive to the choice.

\subsection{The main theorem}\label{subsec: main theorem}

\begin{theorem}[Structure of almost stationary measures]\label{thm: qnz}
	Let $G$ be a connected noncompact simple Lie group with finite center and
	real rank at least two, and let $\mu$ be as in
	Definition~\ref{def: the measure mu intro}. There are constants $C,c,c',A_G,\eps_0>0$
	depending only on $G$ with the following property.

	Let $X$ be a $G$-space, on which $G$ acts by locally Lipschitz orbits (see Definition \ref{def: locally lischitz}), let $\eps_0>\eps>0$, and let $\nu\in\Mm_1(X)$ be
	$\eps$-almost stationary with respect to $\mu$. Let $f:X\rightarrow \RR$ be
	$1$-Lipschitz, bounded by $1$ in $L^\infty$. Then for
	every $\eta\in[0,1]$ at least one of the following holds:
	\begin{enumerate}
		\item\label{item: inv} $\nu$ is
		$\bigl(f,\;C(\abs{\log\eps}^{-c}+\eta^{1/4})\bigr)$-almost invariant;
		\item\label{item: fac} for every $\zeta>\frac{A_G}{\log \abs{\log\eps}}$ there is
		a measurable $W\subset X$ with $\nu(W)\geq\eta$ having an
		$(R_\zeta,\delta_\zeta)$-projective factor, where
		\[
			R_\zeta=\abs{\log\eps}^{\frac{\zeta}{40\dim G}},
			\qquad
			\delta_\zeta=C\,\abs{\log\eps}^{\frac{c'\zeta}{40\dim G}}\,\zeta^{c}.
		\]
	\end{enumerate}
\end{theorem}

The statement carries two free parameters. The first, $\eta$, governs a trade-off: decreasing it strengthens the invariance in
\eqref{item: inv} while shrinking the set produced in \eqref{item: fac}. The second, $\zeta$, lives inside \eqref{item: fac}
and arbitrates between the two features of an almost factor discussed in
Remark~\ref{rem: role of R}: larger $\zeta$ implies a larger scale $R_\zeta$ but a weaker approximation $\delta_\zeta$. Since \eqref{item: fac} is
vacuous unless $R_\zeta$ is large and $\delta_\zeta$ small at once, the following
is the quickest way to see that both can be driven into the useful range.

\begin{remark}\label{rem: qnz qualitative}
	Fix $R\geq1$ and $\eta,\delta>0$. Choosing $\zeta$ with $R_\zeta\geq R$,
	which is admissible once $\eps$ is small, gives
	$\delta_\zeta\ll_R(\log\abs{\log\eps})^{-c}$. Hence there is
	$\eps_0=\eps_0(G,R,\eta,\delta)$ such that every $\eps$-almost stationary
	$\nu$ with $\eps<\eps_0$ is either $(f,C\eta^{1/4}+\delta)$-almost
	invariant, or carries a set $W$ of measure at least $\eta$ with an
	$(R,\delta)$-projective factor.
\end{remark}

The rate matters: the defects in
Theorem~\ref{thm: qnz} decay like a power of $\abs{\log\eps}$. This is enough
for the geometric applications, where the input measures are almost stationary
at scale $\eps=2/n$ and the phenomena of interest live at scale $\log n$. The
logarithm comes from the balance, in Section~\ref{sec: methods}, between an
entropy computation that forces to take an exponential time and an $\epsilon$-almost invariance that
survives only for time $O(\eps^{-1})$.

\begin{remark}[Free actions]\label{rem: free}
	Theorem~\ref{thm: qnz} says nothing about free actions, where all
	stabilisers are trivial and Definition~\ref{def: almost factor} is vacuous.
	That case requires a different quantitative anchor for the factor condition
	and will be treated separately.
\end{remark}

\subsection{Methods: replacing limits by rates}\label{sec: methods}

Our proof follows the architecture of \cite{NZ}.  We keep the main structure of the proof but replace the inputs. Each of those inputs is an infinitary statement --- a convergence theorem,
an almost-everywhere statement, a statement about closed subgroups --- and each
fails outright for measures that are only $\eps$-almost stationary. Replacing
them is the content of this paper, and we expect the replacements to outlast the
particular application made here.

\subsubsection{A Furstenberg decomposition with a rate}\label{sssec: method QF}

The classical decomposition $\nu=m_{G/P}*\lambda$ (appearing in \cite{FurstenbergPoisson}) produces an exactly
$P$-invariant $\lambda$ by a martingale limit; for an almost stationary measure
there is no boundary to condition on. We construct the decomposition by hand
instead, tilting the heat kernel $h_t$ by the Poisson kernel at a fixed boundary
point. The resulting measure $\eta_t$ satisfies the exact identity
$\mu_t=m_K*\eta_t$, and its $P$-invariance is quantitative:
$\sup_{p\in C}d_{\TV}((L_p)_*\eta_t,\eta_t)=O(t^{-1/2+\theta})$
(Theorem~\ref{thm: almost furst}). The tilt is an $h$-transform, so equivariance
is exact and the entire error sits in an analytic estimate, supplied by the
global heat-kernel bounds of \cite{APZ}. In effect the Poisson boundary is
replaced by the Gaussian fluctuation of the radial part of $\mu_t$ at scale
$\sqrt t$ about $2t\rho$. This is what ties us to the heat kernel.

\subsubsection{The entropigeonhole principle}\label{sssec: method entropy}

The Mautner phenomenon is invoked in \cite{NZ} through the pointwise ergodic
theorem, which is unavailable here: for an almost invariant measure there is no
invariant $\sigma$-algebra and no ergodicity to assume. What the ergodic theorem
provides is that averaging forwards and averaging backwards eventually agree,
and this we recover finitarily. By Pinsker's inequality applied to the
Jensen--Shannon divergence, any discrepancy between the forward and backward
averages at a given scale is paid for by a gain in Shannon entropy; telescoping
dyadically over scales bounds the total by the available entropy $\log\#\tau$,
which does not grow with the number of scales (Lemma~\ref{lem: key}).
Discrepancy therefore cannot be spread across all scales at once, and at a
positive proportion of dyadic times the two averages agree to within $\delta$.
Combined with the exponential contraction of $V_\theta$ under $s_{-t}$, this
yields the quantitative Mautner phenomenon of
Theorem~\ref{thm: effective mautner}. The entropy budget forces the relevant
time $T$ to be exponential in the precision, while almost invariance survives
only for time $O(\eps^{-1})$; balancing the two is the source of the logarithmic
dependence on $\eps$ in Theorem~\ref{thm: qnz}.

\subsubsection{Factor functions and fast generation}\label{sssec: method fast gen}

Applying the above to a test function yields a function almost constant along
$V_\theta$-orbits. Classically such a function descends to $G/V_\theta$. Here
there is no quotient, so we retain the function together with the domain on
which its invariance holds, the so-called \emph{factor functions}  (Definition~\ref{def: factor function}). Almost
invariance then propagates along a word at a cost linear in its length, which
means we can afford only \emph{bounded} words: this constrains us to understand when generation is \emph{fast}. Accordingly we prove a dichotomy in the spirit of growth in
groups (Lemma~\ref{lem: fast generation}): if the unit ball of $V$ and a subset $S$ of the stabiliser at a given point $x$ fail
to generate the unit ball of $G$ fast- say with words of length at most $n$-, then $d_H^+(S,V')\leq CR^{c'}/n^{c}$ for some
proper closed supergroup $V'\geq V$. Both alternatives are quantitative: either the stabiliser is quantitatively close to a proper subgroup, and $d_H^+$ is exactly what measures almost projective factors in
Definition~\ref{def: almost factor}; or almost invariantce propagates to the whole unit ball of $G$. 

The proof is a quantitative escape argument
inspired by a result of de~Saxc\'e \cite{deSaxce} and by the quantitative study
of growth in groups \cite{BGT}. The polynomial dependence matters as the
parameters coming out of the previous step are poor. 

\subsubsection{Assembling}\label{sssec: method assembly}

Higher rank enters exactly once, and decisively: the groups $\overline U_\theta$
generate $\overline V$, and $\langle P,\overline V\rangle=G$. With this, the
three ingredients fit together as
\begin{align*}
\text{almost stationary}
\ \longrightarrow\
\text{almost }P\text{-invariant}
\ \longrightarrow\
\text{factor functions}
\\ \longrightarrow\
\begin{cases}
\text{almost projective factor,}\\[0.2em]
\text{additional almost invariance,}
\end{cases}
\end{align*}
the last alternative holding for every singular direction $\theta$.

\subsection{Applications and outlook}\label{subsec: outlook}

The applications announced in Section~\ref{subsec: outlook} are proved in the
companion paper \cite{QFG}, where Theorem~\ref{thm: qnz} is applied to the
Ces\`aro averages \eqref{eq: cezaro average} on $G/\Gamma$ for $\Gamma\leq G$
discrete of infinite covolume. Three consequences illustrate the range of the
method: the effective Fr\k{a}czyk--Gelander theorem and the resulting geometric
characterisation of lattices; a quantitative rate of Benjamini--Schramm convergence to the ambient simple Lie group for quotient spaces of lattices of higher-rank simple Lie groups, without needing any additional arithmetic assumption on the lattices; and a new derivation of the Stuck--Zimmer theorem \cite{Stuck-Zimmer} that
avoids the intermediate factor theorem and amenability arguments. That a
quantitative statement should simplify the proof of the qualitative one is, we
think, the clearest evidence that almost stationarity is the right object of
study rather than a technical replacement for stationarity.

We regard the present paper as the first step of a longer programme, and several
directions are already underway.

\begin{itemize}
\item \emph{Free actions.} A dichotomy covering free actions requires a
	quantitative anchor other than stabilisers (Remark~\ref{rem: free}), and
	will be treated separately.

	\item \emph{General step measures.} Extending the quantitative Furstenberg
	decomposition of Section~\ref{sec: QF} to all sufficiently smooth admissible
	measures, via a local limit theorem in the style of Koegler's \cite{Koegler}, would extend Theorem~\ref{thm: qnz}
	with it.

	\item \emph{Effective convergence of invariants.} The Stuck--Zimmer
	theorem underlies the asymptotics of Betti numbers along sequences of
	lattices \cite{Samurai}. In upcoming work we use the tools developed here to
	make Benjamini--Schramm convergence, and with it the convergence of
	normalised Betti numbers, effective in terms of the covolume; the result is
	a single theorem containing the Fr\k{a}czyk--Gelander
	\cite{Fraczyk-Gelander}, Stuck--Zimmer \cite{Stuck-Zimmer} and the seven samurai's theorems
	\cite{Samurai}.
\end{itemize}

More broadly, the entropigeonhole method, factor functions and fast generation
are not tied to the Nevo--Zimmer setting. They are quantitative substitutes for
the Mautner phenomenon, for factor maps and for the generation of unipotent
directions, and we expect them to be useful wherever rigidity phenomena in
higher rank are currently accessible only in the limit.

\subsection{Organisation of the paper}\label{subsec: organisation}

Section~\ref{sec: nota} fixes notation and collects the facts about $G$, its
parabolic subgroups and the metrics of Section~\ref{subsec: definitions} used
throughout, including the variants of almost invariance omitted from
Definition~\ref{def: almost stat and inv}. Section~\ref{sec: QF} proves the
quantitative Furstenberg decomposition. Section~\ref{sec: QM} develops the
entropigeonhole principle and the effective Mautner phenomenon.
Section~\ref{sec: fast generation} introduces factor functions, and the fast generation dichotomy.
Theorem~\ref{thm: qnz} is assembled in Section~\ref{sec: final proof}.

\subsection{Acknowledgments}
We thank Emmanuel Breuillard, Marc Burger, Manfred Einsiedler, Mikolaj Fraczyk, Alex Furman Tsachik Gelander, Segev Gonen Cohen, Alex Gorodnik, Arie Levit, Elon Lindenstrauss, Amos Nevo, and Barak Weiss for many discussions during the process of making this paper and its writing. 
We would also like to thank Yuval Gorfine and Michael Glasner for many long discussions and for listening to many failed strategies. 
The last named author would like to acknowledge the positive influence of L.R. on the making of this project.
The first named author was supported by ISF grant number 3423/24. 

The last named author was supported by SNF grant 200020--212617.

Last but not least, the last named author would like to thank Shoshana Chaya for her unconditioned support and infinite patience.
This paper is dedicated to her with a lot of love. 

\section*{Statement on the use of AI assistants}

The authors used AI assistants (large language models), specifically Claude and Chatgpt in the preparation of
this paper, as follows.

\emph{Writing and presentation.} AI assistants were used to help draft and
rephrase exposition, to suggest reorganizations of the material, and to
proofread. Every sentence was written and every modification reviewed, and where necessary rewritten, by the
authors.

\emph{Discussion.} The authors used AI assistants as a soundboard: to test
technical claims and to
stress-test strategies before committing to them.

\emph{Mathematical content.} With one exception, all mathematical ideas,
statements and proofs in this paper are the authors' own. The first exception occurs
in Section~\ref{sec: QF}. The idea of taking the heat kernel as the step
measure is due to the authors; the idea of twisting the heat-kernel density by
the Poisson kernel, and the computations this leads to in the proof of
Theorem~\ref{thm: almost furst}, originated in an exchange with an AI
assistant. The resulting argument was checked line by line by the authors.

\section{Symbols used in the paper}
For the convenience of the reader, we collect here the notation that
remains fixed throughout the paper.
\begin{center}
\renewcommand{\arraystretch}{1.2}
\begin{tabular}{ll}
\hline
Symbol & Meaning \\
\hline
$G$ & A semisimple Lie group of rank at least two.\\
$\Phi$ & A root system. \\
$\Phi^+$ & Corresponding set of positive roots. \\
$\Delta$ & The subset of $\Phi^+$ of simple roots  of $G$.\\
$\theta$ & A fixed subset of $\Delta$ as in Definition~\ref{def: nz setup}.\\
$s_\theta$ & The singular direction associated with $\theta$.\\
$C_\theta$ & The centralizer associated with $s_\theta$.\\
$V_\theta$ & The expanded subgroup associated with $s_\theta$.\\
$\overline U_\theta$ & The commuting part with $s_\theta$ of $C_\theta$ in $\overline V$.\\
$G^{-}$ & The subgroup of elements not expanding under conjugation by
$s_\theta$.\\
$G_R$ & The ball of radius $R$ in the Riemannian metric in $G$.\\
$G_R^{\|}$ & The norm ball of radius $R$ in $G$.\\
$X$ & A $G$-space.\\
$\mu$ & The step of the random walk on $G$.\\
$\lambda$ & A $P$-almost invariant probability measure on $X$.\\
$\Ee_{s_\theta}^{T}$ & The stopped Cesàro averaging operator along
$s_\theta$.\\
$Y_R$ & The space of $1$-Lipschitz test functions at scale $R$.\\
$\hat f_R$ & The factor function associated with $f\in Y_R$.\\
$G_1(Y,x)$ & The good set associated with a factor domain $Y$ \ref{def: g1 set}.\\
$\operatorname{stab}_G(x)$ & The stabilizer of $x$ in $G$.\\
$d_H^+$ & The one-sided Hausdorff distance \ref{def: dh+}.\\
\hline
\end{tabular}
\end{center}

Unless explicitly stated otherwise, all implied constants in $O(\cdot)$,
$\ll$, and $\gg$ are absolute and independent of the parameters appearing
in the statement under consideration.

 

\section{Definitions and notation}\label{sec: nota}

\subsection*{Geometric and root-theoretic notation}

\begin{definition}[Locally Lipschitz orbits]\label{def: locally lischitz}
    Given an action of a Lie group $G$ with metric $d_G$ on a metric space $(X,d_X)$, we say that the action is by locally Lipschitz orbits if for every $x\in X$, the orbit map $o_x:G_1\rightarrow X$ defined on the $d_G$ unit ball in $G$, is $1$-Lipschitz.
\end{definition}

\begin{nota}[Balls and norms]
Recall that $d$ denotes the fixed left-invariant Riemannian metric on
$G$, and that the norm metric is induced by the fixed faithful
finite-dimensional representation of $G$. For $R>0$, write
\[
    G_R:=\{g\in G:d(e,g)\leq R\},
\]
and let $G_R^{\|}$ denote the corresponding norm ball of radius $R$.
If $H\leq G$, we abbreviate
\[
    H_R:=H\cap G_R,
    \qquad
    H_R^{\|}:=H\cap G_R^{\|}.
\]
In particular, $G_1$ denotes the unit Riemannian ball in $G$.
\end{nota}

Our root-theoretic conventions follow those of~\cite[Section~2]{NZ}.

\begin{definition}[Root-theoretic notation]\label{def: nz setup}
Let $G$ be a connected semisimple Lie group. Fix a maximal
$\RR$-split torus $S<G$, choose a system $\Phi^+$ of positive roots,
and let $\Delta\subseteq\Phi^+$ be the corresponding set of simple
roots. For $\theta\subseteq\Delta$, set
\[
    \mathfrak s_\theta
    :=
    \bigcap_{\alpha\in\theta}\ker\alpha,
\]
and let $S_\theta<S$ be the connected subgroup with Lie algebra
$\mathfrak s_\theta$. Define the associated Levi subgroup by
\[
    L_\theta:=Z_G(S_\theta).
\]
Let $V_\theta$ and $\overline V_\theta$ denote the unipotent radicals
of the standard parabolic subgroup $P_\theta$ and its opposite
$\overline P_\theta$, respectively. Thus
\[
    P_\theta=L_\theta\ltimes V_\theta,
    \qquad
    \overline P_\theta=L_\theta\ltimes\overline V_\theta.
\]
Moreover,
\[
    L_\theta=M_\theta S_\theta,
\]
where $M_\theta$ is semisimple. When $\theta=\varnothing$, we write
\[
    P:=P_\varnothing=L_\varnothing\ltimes V,
    \qquad
    L_\varnothing=MS.
\]
\end{definition}

Whenever a fixed element $s_\theta\in S_\theta$ is chosen later in
the paper, the notation $s_\theta^t$, $t\in\RR$, denotes its real
powers. It does not mean that $S_\theta$ is a one-parameter subgroup.

\subsection*{Measures, partitions, and entropy}

\begin{nota}[Probability measures, partitions, and entropy]
\label{nota: shanon}
Let $X$ be a metric space and let $\Mm_1(X)$ denote the space of Borel
probability measures on $X$. For a measurable set $A\subseteq X$, set
\[
    \Pp_1(A)
    :=
    \{\eta\in\Mm_1(X):\eta(A)=1\}.
\]

A finite measurable partition $\tau$ of $X$ is a finite collection of
pairwise disjoint measurable sets whose union is $X$. We say that a
set is $\tau$-measurable if it is a union of elements of $\tau$, and
define
\[
    \delta(\tau):=\max_{A\in\tau}\diam(A).
\]

For a probability vector $(p_i)_{i=1}^k$, put
\[
    H((p_i)_{i=1}^k)
    :=
    -\sum_{i=1}^k p_i\log p_i,
\]
with the convention $0\log0=0$. If $\eta\in\Mm_1(X)$ and $\tau$ is
fixed, we abbreviate
\[
    H_\tau(\eta)
    :=
    H((\eta(A))_{A\in\tau})
\]
to $H(\eta)$.
\end{nota}

\begin{definition}[Metrics on measures]
\label{def: metric on measures}\label{nota: f dist}
Let $\mu_1,\mu_2\in\Mm_1(X)$. Their total variation norm is
\begin{equation}
    \norm{\mu_1-\mu_2}_{TV}
    :=
    \sup_{\norm{\varphi}_\infty\leq1}
    \abs{\mu_1(\varphi)-\mu_2(\varphi)}
    =
    2\sup_{A\subseteq X\text{ measurable}}
    \abs{\mu_1(A)-\mu_2(A)}.
\end{equation}

Their bounded-Lipschitz distance is
\begin{equation}
    d_w(\mu_1,\mu_2)
    :=
    \Wb(\mu_1,\mu_2)
    :=
    \sup\left\{
        \abs{\mu_1(\varphi)-\mu_2(\varphi)}:
        \on{Lip}(\varphi)\leq1,
        \ \norm{\varphi}_\infty\leq1
    \right\}.
\end{equation}

If $Y$ is a metric space and $f:X\to Y$ is measurable, define the
$f$-weak distance by
\begin{equation}
    \Wf{f}(\mu_1,\mu_2)
    :=
    \Wb(f_*\mu_1,f_*\mu_2).
\end{equation}
When $Y=\RR$, this is also denoted by
$\Wf{f}(\mu_1,\mu_2)$. Equivalently,
\[
    \Wf{f}(\mu_1,\mu_2)
    =
    \sup_{\substack{\phi:f(X)\to\RR,\ \on{Lip}(\phi)\leq1\\
                     \norm{\phi}_\infty\leq1}}
    \abs{\mu_1(\phi\circ f)-\mu_2(\phi\circ f)}.
\]
For a family $\Ff$ of measurable maps from $X$ to metric spaces, set
\[
    d_\Ff(\mu_1,\mu_2)
    :=
    \sup_{f\in\Ff}\Wf{f}(\mu_1,\mu_2).
\]

Finally, if $\tau$ is a finite measurable partition of $X$, define
\begin{equation}
    d_\tau(\mu_1,\mu_2)
    :=
    \max\left\{
        \abs{\mu_1(A)-\mu_2(A)}:
        A\text{ is $\tau$-measurable}
    \right\}.
\end{equation}
\end{definition}

\begin{claim}\label{cl: internet inequality}
For every finite measurable partition $\tau$ and every
$\mu_1,\mu_2\in\Mm_1(X)$,
\[
    d_\tau(\mu_1,\mu_2)
    =
    \frac12\sum_{A\in\tau}
    \abs{\mu_1(A)-\mu_2(A)}.
\]
\end{claim}

\begin{proof}
Put $a_A:=\mu_1(A)-\mu_2(A)$. Since
$\sum_{A\in\tau}a_A=0$, the sum of the positive $a_A$ is equal to the
absolute value of the sum of the negative $a_A$, and both are equal to
$\frac12\sum_{A\in\tau}\abs{a_A}$. Taking the union of the atoms for
which $a_A\geq0$ realizes the maximum in the definition of $d_\tau$.
\end{proof}

\subsection*{Almost invariance}

\begin{definition}[Almost invariance]
Let $\nu\in\Mm_1(X)$ and let $\eps>0$. We say that $\nu$ is
\begin{enumerate}
    \item \emph{norm-$\eps$-invariant under $G$} if
    \[
        \norm{g.\nu-\nu}_{TV}\leq\eps
        \qquad\text{for every }g\in G_1;
    \]

    \item \emph{$(f,\eps)$-almost invariant under $G$}, for a
    measurable map $f:X\to Y$ into a metric space $Y$, if
    \[
        \Wf{f}(g.\nu,\nu)\leq\eps
        \qquad\text{for every }g\in G_1;
    \]

    \item \emph{$(\Ff,\eps)$-almost invariant under $G$}, for a family
    $\Ff$ of such maps, if
    \[
        d_\Ff(g.\nu,\nu)\leq\eps
        \qquad\text{for every }g\in G_1;
    \]

    \item \emph{weakly $\eps$-invariant under $G$} if
    \[
        \Wb(g.\nu,\nu)\leq\eps
        \qquad\text{for every }g\in G_1.
    \]
\end{enumerate}
When the acting group is clear, we use shorter expressions such as
``$G$-norm-$\eps$-invariant'' and ``$f$-$\eps$-almost invariant.''
\end{definition}

\subsection*{Heat-kernel notation}

\begin{definition}[Heat-kernel notation]\label{def: the measure mu}
Let $p_t$ be the $K$-bi-invariant heat-kernel density fixed in
Definition~\ref{def: the measure mu intro}, and let $m_G$ denote Haar
measure on $G$. We identify $p_t$ with the corresponding probability
measure $p_t\,dm_G$ and put
\[
    \mu:=p_1,
    \qquad
    q_n:=\frac{d\mu^{*n}}{dm_G}.
\]
By the heat-kernel semigroup property,
\[
    q_n=p_n.
\]
\end{definition}

\section{Furstenberg decomposition for almost stationary measures}\label{sec: QF}

We take $S$ compatible with $K$ and identify $\mathfrak{s}^{*}$ with
$\mathfrak{s}$ using the Cartan inner product. Put
$\rho=\frac12\sum_{\alpha\in\Phi^{+}}m_{\alpha}\alpha$, and write $Z(g)$ and
$\kappa(g)$ for the Iwasawa and Cartan projections, with conventions
\[
g=v(g)\exp Z(g)k(g),\qquad
g\in K\exp\kappa(g)K,\qquad
\kappa(g)\in\overline{\mathfrak{s}^{+}}.
\]
Let $B=G/P$, $b_{0}=P$, and define
\[
\mathcal{J}(g,b)=\frac{d(g.m_{G/P})}{dm_{G/P}}(b),\qquad
\chi(p)=\mathcal{J}(p,b_{0})\quad(p\in P).
\]
For $p=m\exp(H)v$, the Jacobian of its action at $b_0$ is
$e^{-2\rho(H)}$. The Iwasawa decomposition therefore gives
\begin{equation}\label{af:poisson}
\mathcal{J}(g,kb_{0})
=\exp\bigl(2\rho(Z(k^{-1}g))\bigr)
\end{equation}
and in particular $\mathcal{J}$ is smooth
\cite[Theorem 6.46]{KnappBeyond}.
We will use the cocycle identity
\[
\mathcal{J}(g_{1}g_{2},b)
=\mathcal{J}(g_{1},b)\mathcal{J}(g_{2},g_{1}^{-1}b).
\]
Thus $\chi$ is a positive character of $P$. Also,
$\mathcal{J}(k,b)=1$ for $k\in K$ and
$\int_B\mathcal{J}(g,b)\,dm_{G/P}(b)=1$.
Write $\mu_t=p_t\,dm_G$, so that $\mu_n=\mu^{*n}$.

\begin{theorem}[Polynomial-rate decomposition]\label{af:decomposition}
For $t>0$, the measure
\[
d\eta_t(g)=\mathcal{J}(g,b_0)p_t(g)\,dm_G(g)
\]
is a probability and $\mu_t=m_K*\eta_t$.
There is a probability measure $\sigma_t$ supported on $P$ such that
\[
\eta_t=\sigma_t*m_K,\qquad \mu_t=m_K*\sigma_t.
\]
For every compact $C\subset P$ and every $\gamma\in(0,\frac12)$, there is
$A_{C,\gamma}<\infty$ such that, for $t\geq1$,
\begin{equation}\label{af:rate}
\sup_{p\in C}\|p.\eta_t-\eta_t\|_{TV}
=\sup_{p\in C}\|p.\sigma_t-\sigma_t\|_{TV}
\leq A_{C,\gamma}t^{-1/2+\gamma}.
\end{equation}
In particular, $\gamma=1/4$ gives the rate $A_Cn^{-1/4}$ for
$\eta_n$ and $\sigma_n$.
\end{theorem}

Fix a $K$-invariant Riemannian distance $d_B$ on $B$, and put
$\operatorname{depth}(Z)=\min_{\alpha\in\Delta}\alpha(Z)$.

\begin{lemma}[Boundary contraction]\label{af:boundary}
There are $c,C>0$ such that, for $Z\in\overline{\mathfrak{s}^{+}}$,
\begin{equation}\label{af:contraction}
\int_B d_B(\exp(Z)b,b_0)\,dm_{G/P}(b)
\leq Ce^{-c\operatorname{depth}(Z)}.
\end{equation}
\end{lemma}

\begin{proof}
Passing to the adjoint group does not change $B$ or its measure.
In the associated real algebraic group, the multiplication map defining
the big cell $\bar V P$ is an open immersion
\cite[Theorem 13.33(d)]{Milne2017}. Its inverse is regular,
and the logarithm on $\bar V$ is polynomial. Since principal open sets
form a basis of the Zariski topology, choose a regular function $F$ with
$F(e)\neq0$ and $\{F\neq0\}\subset\bar V P$.
Writing $g=\bar v(g)p(g)$ on the big cell, the coordinates of
$\log\bar v(g)$ on $\{F\neq0\}$ have the form $Q_j/F^{r_j}$,
with $Q_j$ regular. Taking $r=1+\max_j r_j$ and using the common
denominator $D=F^r$, we obtain
\[
kb_0=\exp X(k)b_0,\qquad X(k)\in\bar{\mathfrak v},\qquad
\|X(k)\|\leq \frac{C}{|D(k)|}
\quad\text{if }D(k)\neq0.
\]
Here the numerators are bounded on $K$ by compactness.
Since $K$ is connected, the analytic function $D|_K$ has finite order at
every zero. Applying \cite[Proposition 3.4]{KleinbockMargulis} in
finitely many analytic charts near its zeros, where Haar measure has a bounded smooth
density, gives $m_K\{|D|\leq u\}\leq Cu^a$ for some $a>0$ and
$0<u\leq1$. Consequently $X(k)$ is defined almost everywhere and
\begin{equation}\label{af:coordinate-tail}
m_K\{k:\|X(k)\|>A\}\leq CA^{-a}\qquad(A\geq1).
\end{equation}

Put $h=\operatorname{depth}(Z)$. In a norm making the restricted root
spaces orthogonal,
\[
\|\operatorname{Ad}(\exp Z)X\|\leq e^{-h}\|X\|
\qquad(X\in\bar{\mathfrak v}).
\]
Since $\exp Z$ fixes $b_0$, smoothness near $0$ and compactness of $B$ give
\[
d_B(\exp(Z)kb_0,b_0)
\leq C\min\{1,e^{-h}\|X(k)\|\}.
\]
Splitting the integral at $\|X(k)\|=e^{h/2}$ and using
\eqref{af:coordinate-tail}, we obtain
\[
\int_K d_B(\exp(Z)kb_0,b_0)\,dm_K(k)
\leq C\bigl(e^{-h/2}+e^{-ah/2}\bigr).
\]
This proves the lemma.
\end{proof}

\begin{proof}[Proof of Theorem~\ref{af:decomposition}]
The cocycle identity gives
\[
\int_K\mathcal{J}(k^{-1}g,b_0)\,dm_K(k)
=\int_B\mathcal{J}(g,b)\,dm_{G/P}(b)=1.
\]
Since $p_t$ is left $K$-invariant, it follows by Tonelli that
$m_K*\eta_t=\mu_t$ and $\eta_t(G)=1$. Moreover, for every bounded
left $K$-invariant function $F$,
\begin{equation}\label{af:radial}
\int_G F\,d\eta_t=\int_G F\,d\mu_t.
\end{equation}

Choose left Haar measure $m_P$ so that the Iwasawa integration formulas
\cite[Propositions 8.43--8.44]{KnappBeyond} read
\begin{equation}\label{af:haar}
\begin{split}
\int_G F\,dm_G
&=\int_P\int_K F(pk)\,dm_K(k)\,dm_P(p)\\
&=\int_K\int_P F(kp)\chi(p)\,dm_P(p)\,dm_K(k).
\end{split}
\end{equation}
Define
\[
d\sigma_t(p)=\chi(p)p_t(p)\,dm_P(p).
\]
Using $\mathcal{J}(pk,b_0)=\chi(p)$ and the $K$-bi-invariance of $p_t$,
the two formulas in \eqref{af:haar} give respectively
$\eta_t=\sigma_t*m_K$ and $\mu_t=m_K*\sigma_t$.
In particular, $\sigma_t(P)=1$.
For $p\in P$, the cocycle identity also gives
\begin{equation}\label{af:ratio}
R_{t,p}(g):=\frac{d(p.\eta_t)}{d\eta_t}(g)
=\chi(p)^{-1}\frac{p_t(p^{-1}g)}{p_t(g)}.
\end{equation}
The same formula, restricted to $P$, is the density of $p.\sigma_t$
with respect to $\sigma_t$. Since $R_{t,p}$ is right $K$-invariant,
the first formula in \eqref{af:haar} yields
\begin{equation}\label{af:isometry}
\|p.\sigma_t-\sigma_t\|_{TV}
=\int_G|R_{t,p}-1|\,d\eta_t
=\|p.\eta_t-\eta_t\|_{TV}.
\end{equation}

Fix $C\subset P$ compact and $\gamma\in(0,\frac12)$. Set
\[
r_t=t^{1/2+\gamma},\qquad
\Omega_t=\{Z\in\overline{\mathfrak{s}^{+}}:
                 \|Z-2t\rho\|\leq r_t\},\qquad
E_t=K\exp(\Omega_t)K.
\]
The heat-kernel estimates of
\cite[Lemma 2.1 and Proposition 3.7]{APZ} give,
for every $N>0$ and all sufficiently large $t$,
\begin{align}
\mu_t(E_t^c)&\leq C_{N,\gamma}t^{-N\gamma},
\label{af:heat-tail}\\
\left|\frac{p_t(p^{-1}g)}{p_t(g)}-\mathcal{J}(p,kb_0)\right|
&\leq A_{C,\gamma}t^{-1/2+\gamma}
\quad(p\in C,\ g=k\exp(Z)k',\ Z\in\Omega_t).
\label{af:heat-ratio}
\end{align}
The second estimate is uniform on $C$. Indeed, in the proofs of
Proposition 3.7 and Lemma 3.8 of that paper, the Cartan displacement is
bounded by $\sup_{p\in C}\|\kappa(p)\|$, and the Iwasawa components of
$k^{-1}p$, for $(k,p)\in K\times C$, range over compact sets.
For $p\in C\cap K$, left $K$-invariance of $p_t$ and $K$-invariance
of $m_{G/P}$ give
\[
\frac{p_t(p^{-1}g)}{p_t(g)}=\mathcal{J}(p,kb_0)=1,
\]
so the error in \eqref{af:heat-ratio} is zero.

The cocycle identity and $K$-invariance of $d_B$ imply
\begin{equation}\label{af:angular}
\begin{split}
&\int_K d_B(kb_0,b_0)\mathcal{J}(k\exp Z,b_0)\,dm_K(k)\\
&\hspace{1cm}=\int_B d_B(b,b_0)\mathcal{J}(\exp Z,b)\,dm_{G/P}(b)\\
&\hspace{1cm}=\int_B d_B(\exp(Z)b,b_0)\,dm_{G/P}(b).
\end{split}
\end{equation}
Here we used $d_B(kb_0,b_0)=d_B(k^{-1}b_0,b_0)$.
For $Z\in\Omega_t$, we have $\operatorname{depth}(Z)\geq c_Gt$ once
$t$ is sufficiently large. Lemma~\ref{af:boundary} therefore bounds
\eqref{af:angular} by $Ce^{-ct}$.
On the other hand, smoothness of $\mathcal{J}$ and
$\inf_{p\in C}\chi(p)>0$ give, by
\eqref{af:ratio} and \eqref{af:heat-ratio},
\[
|R_{t,p}(k\exp(Z)k')-1|
\leq A_{C,\gamma}\bigl(d_B(kb_0,b_0)+t^{-1/2+\gamma}\bigr)
\quad(Z\in\Omega_t).
\]
Cartan integration and \eqref{af:angular} now yield
\begin{equation}\label{af:interior}
\sup_{p\in C}\int_{E_t}|R_{t,p}-1|\,d\eta_t
\leq A_{C,\gamma}\bigl(t^{-1/2+\gamma}+e^{-ct}\bigr),
\end{equation}
since the radial mass is at most $1$.

Finally, $\int_G R_{t,p}\,d\eta_t=1$, so
\begin{align*}
\|p.\eta_t-\eta_t\|_{TV}
&=2\int_G(1-R_{t,p})_+\,d\eta_t\\
&\leq2\int_{E_t}|1-R_{t,p}|\,d\eta_t+2\eta_t(E_t^c).
\end{align*}
Equation~\eqref{af:radial} gives $\eta_t(E_t^c)=\mu_t(E_t^c)$.
Choose $N=N(\gamma)>0$ with $N\gamma\geq\frac12-\gamma$.
Equations~\eqref{af:heat-tail} and \eqref{af:interior} then give
\begin{align*}
\sup_{p\in C}\|p.\eta_t-\eta_t\|_{TV}
&\leq2A_{C,\gamma}\bigl(t^{-1/2+\gamma}+e^{-ct}\bigr)
       +2C_{N,\gamma}t^{-N\gamma}\\
&\leq A'_{C,\gamma}t^{-1/2+\gamma}
\end{align*}
for large $t$, since $t^{-N\gamma}\leq t^{-1/2+\gamma}$ and
$e^{-ct}=O_\gamma(t^{-1/2+\gamma})$. Increasing the constant
covers $1\leq t\leq t_0$, since the total variation distance between
probabilities is at most $2$. Equation~\eqref{af:isometry} gives the
same bound for $\sigma_t$.
\end{proof}

\begin{theorem}\label{thm: almost furst}
If $0<\epsilon\leq1$ and $\nu$ is an $\epsilon$-almost stationary
probability measure on $X$, there is $\lambda\in\mathcal M_1(X)$ such that
\[
\|\nu-m_K*\lambda\|_{TV}\leq C\epsilon^{1/5},\qquad
\sup_{p\in P\cap G_1}\|p.\lambda-\lambda\|_{TV}
\leq C\epsilon^{1/5}.
\]
\end{theorem}

\begin{proof}
Take $n=\lceil\epsilon^{-4/5}\rceil$ and $\lambda=\eta_n*\nu$.
Then $m_K*\lambda=\mu^{*n}*\nu$, and telescoping gives
\[
\|m_K*\lambda-\nu\|_{TV}
\leq n\|\mu*\nu-\nu\|_{TV}
\leq2\epsilon^{1/5}.
\]
Convolution contracts total variation, so
Theorem~\ref{af:decomposition}, with $\gamma=1/4$, gives
\[
\sup_{p\in P\cap G_1}\|p.\lambda-\lambda\|_{TV}
\leq Cn^{-1/4}\leq C\epsilon^{1/5}.
\]
\end{proof}

\section{Factor Functions and the Fast Generation Dichotomy}\label{sec: fast generation}
This section contains, in a sense, the heart of the main argument of the proof of Theorem \ref{thm: qnz}.
We start by defining \emph{factor functions}, which play the role of the quantitative counterparts of functions on factor spaces, which are central to our argument (and to this of \cite{NZ}).
\begin{definition}\label{def: factor function}
	Let $X$ be a $G$ space and let $R>0$.
	Let $V\leq G$ be a subgroup.
	Given $\eps>0$ and subsets $X'\subset X$, $Y\subset X'\times X'$, we say that a function $f:X'\rightarrow \RR$ is a $(G/V,R,\eps)$ factor function if for every $x\in X$ and for every $g\in G_R^{||},v\in V_1$ such that $(gx,gvx)\in Y$ we have:
	\begin{equation}\label{eq: ff invariance}
		\abs{f(gvx)-f(gx)}\leq \eps.
	\end{equation}
    The set $Y$ is referred to as the factor domain of the factor function $f$ and is denoted $Y(f)$.
\end{definition}
\begin{remark}\label{rem: factor for smaller r}
    Note that if $0<R'\leq R$ and $f$ is a $(G/V,R,\eps)$ factor function, it is in particular a $(G/V,R',\eps)$-factor function.
\end{remark}
The following three definitions and one notation give us the terminology required to discuss existence and properties of factor functions. 
\begin{definition}\label{def: factor tuple}
    In this definition we use the setup of Definition \ref{def: factor function}.
    For $gx,gvx\in X'$, we say that $(x,g,v)$ is a factor triple if $(gx,gvx)\in Y(f)$.
    In general, given $\overline g=(e,g_1,\dots,g_n)\in G^{\times (n+1)}$ and $\overline v=(e,v_1,\dots, v_n)\in V_1^{\times(n+1)}$, we say that $(x,\overline g,\overline v)$ is a factor tuple if for every $0\leq k\leq n$ we have $(g_1v_1\cdots g_kx,g_1v_1\cdots g_kv_{k+1}x)\in Y$.
    Note that we denote the set of $n$-tuples of elements from $V_1$ by $V_1^{\times n}$ to distinguish it from the product (in the group) of $V_1$ with itself $n$ times.
\end{definition}

\begin{definition}\label{def: n step generation}
	Let $V\leq G$ be a closed subgroup, let $S\subset G$ be a symmetric finite subset containing $e\in G$ and let $n\in \NN$.
	We say that $V,S$ generate $G$ in $n$ steps if
	\begin{equation}
		(SV_1^{||})^n\supset G_1.
	\end{equation}
\end{definition}

\begin{nota}\label{nota: generation}
    Suppose $V,S$ generate $G$ in $n$ steps. 
    Let $g\in G_1$, and write $g=s_1v_1\cdots s_nv_n$ (which is possible since $V,S$ generate $G$ in $n$ steps). Denote $\overline s_g=(e,s_1,\dots,s_n)\in S^{\times (n+1)}$ and $\overline v_g=(e,v_1,\dots,v_n)\in V_1^{\times (n+1)}$.
\end{nota}

\begin{definition}\label{def: g1 set}
    Suppose $V,S$ generate $G$ in $n$ steps. 
    Given $Y\subset X\times X$ and $x\in X$, we define the subset $G_1(Y,x)$ by:
    \begin{align}
        G_1(Y,x)=\{g:g\in G_1\text{ and }(x,\overline s_g,\overline v_g)\text{ is a factor tuple for }Y\text{ for some }\overline s_g,\overline v_g\text{ as in \ref{nota: generation}}\}.
    \end{align}
\end{definition}

The following lemma, which proves extra invariance of factor functions, is where  factor functions are shown to be a natural counterpart to functions on factor spaces. 
\begin{lemma}\label{lem: extra invariance}
	Let $R>0$ and $n\in \NN$. Suppose $f$ is a $(G/V,R^{2n},\eps)$-factor function on some $X'\subset X$ and suppose that for some $x\in X'$, it holds that $V,\stab_G(x)\cap G_{R}^{||}$ generate $G$ in $n$ steps and denote $Y=Y(f)$.
	Then for every $g\in G_1(Y,x)$ we have:
	\begin{equation}
		\abs{f(gx)-f(x)}\leq n\eps.
	\end{equation}
\end{lemma}

\begin{proof}
By the assumption that $V,\stab_G(x)\cap G_R^{||}$ generate $G$ in $n$ steps, we can write $g=\gamma_1v_1\cdots \gamma_nv_n$ for some $v_i\in V_1$ and $\gamma_i\in \stab_G(x)\cap G_{R}^{||}$.
    Moreover, by the assumption that $g\in G_1(Y,x)$, we may choose $\overline v=(e,v_1,\dots,v_n)$ and $\overline \gamma = (e,\gamma_1,\dots,\gamma_n)$ such that $(x,\overline \gamma,\overline v)$ is a factor tuple (recall Definition \ref{def: factor tuple})
        
    Since $(x,\overline \gamma,\overline v)$ is a factor tuple, it follows that for every $k=0,\dots, n-1$ we have:
    \begin{equation}\label{eq: in y}
        (\gamma_1v_1\cdots \gamma_kx,\gamma_1v_1\cdots \gamma_kv_{k}x)\in Y.
    \end{equation}

    Define $g_k=\gamma_1v_1\cdots \gamma_{k}v_k\in G_{R^{2n}}$.
    Since $\gamma_{k+1}\in \stab_G(x)$, we can write:
    \begin{equation}
        f(\gamma_1v_1\cdots v_{k}\gamma_{k+1}x)=f(\gamma_1v_1\cdots \gamma_{k}v_{k}x)
    \end{equation}
    Since $f$ is a $(G/V,R^{2n},\eps)$-factor function, by Definition \ref{def: factor function} for $g_k\gamma_{k+1}\in G_{R^{2n}}$ and $v=v_{k+1}$ (which can be applied by Eq. \eqref{eq: in y}), we have:
    \begin{equation}
        \abs{f(g_{k+1}x)-f(g_k\gamma_{k+1}x)}=\abs{f(\gamma_1v_1\cdots \gamma_{k}v_{k}\gamma_{k+1}v_{k+1}x)-f(\gamma_1\cdots v_{k}\gamma_{k}v_k\gamma_{k+1}x)}\leq \eps.
    \end{equation}
    We obtain 
    \begin{equation}
        \abs{f(g_{k+1}x)-f(g_kx)}\abs{f(g_{k+1}x)-f(g_k\gamma_{k+1}x)}\leq \eps.        
    \end{equation}
    Using the triangle inequality and the above equation for $k=1,\dots,n$ we can deduce that
    \begin{equation}
        \abs{f(gx)-f(x)}=\abs{f(\gamma_1v_1\cdots \gamma_nv_n x)-f(x)}\leq n\eps
    \end{equation}
    as desired.  
\end{proof}

\subsection{Fast generation}
The following geometric lemma concludes this section by giving a restrictive condition to not having fast generation - the main condition used in Lemma \ref{lem: extra invariance} to attain extra invariance.
This condition will be the pivot for the dichotomy of Theorem \ref{thm: qnz} as displayed in its proof in Section \ref{sec: final proof}.

\begin{lemma}[fast generation dichotomy]\label{lem: fast generation}
Let $V<G$ be a fixed nontrivial connected proper algebraic subgroup
and let $S\subset G_R^{\|}$ be a symmetric subset containing $e$.
Suppose $R\geq 2$. If $V,S$ do not generate $G$ in $n$ steps, then
there exists a proper closed subgroup
\[
        V\leq V'<G
\]
such that
\[
        d_H^+(S,V')
        \leq
        \frac{CR^{c'}}{n^c}.
\]
Here $C,c,c'>0$ depend only on $G$ and $V$. In particular, when $V$
ranges over the finitely many groups $V_\theta$ occurring in
Definition~\ref{def: nz setup}, these constants may be chosen to depend
only on $G$.
\end{lemma}

Throughout, $G$ is a connected non-compact simple Lie group with trivial
centre, $\fg$ its Lie algebra, $D=\dim\fg$, and $\fg$ carries a fixed
Euclidean structure. Since $Z(G)=\{e\}$, the adjoint representation
$\Ad:G\to\GL(\fg)$ is faithful and $\Ad(G)$ is a closed semialgebraic
subgroup of $\GL(\fg)$ \cite[2.4.5, 2.9.10]{BCR}. We use it as our
algebraic model of $G$ and set
\[
        \|g\|_*:=\max\bigl\{\|\Ad g\|,\|\Ad g^{-1}\|\bigr\}\ \ (\geq 1),
        \qquad
        G_R^{\|}:=\{g\in G:\|g\|_*\leq R\}.
\]
We write $d$ for the left-invariant Riemannian metric attached to the
Euclidean structure on $\fg$, and $G_r:=B_G(e,r)$. We use repeatedly the
standard comparison
\begin{equation}\label{eq:fg-log}
        C^{-1}\log\|g\|_*-C\ \leq\ d(g,e)\ \leq\ C\log\|g\|_*+C ,
\end{equation}
which follows from the $KAK$ decomposition.

A general $G$ with finite centre reduces to this case: every object below
($d_{\mathrm{Gr}}$, $H_U$, $\mathcal X_k^V$, $m_V$) is defined through $\Ad$ and
so is insensitive to the centre, and $Z(G)\leq H_U$ always.

Let $V<G$ be a fixed nontrivial connected proper algebraic subgroup and
$\fv=\operatorname{Lie}(V)$, so $0<\dim\fv<D$. Fix once and for all a
compact symmetric semialgebraic identity neighbourhood
$\mathcal K\subset V$ and $r_0>0$ with
$\exp\bigl(B_{\fv}(0,r_0)\bigr)\subset\mathcal K$. Since $V$ is
connected, $\mathcal K$ generates $V$. The pair $V,S$ is said to
\emph{generate $G$ in $n$ steps} if $(\mathcal K S)^n\supset G_1$.
Enlarging $R$ by a bounded factor, we assume throughout that
$\max_{v\in\mathcal K}\|v\|_*\leq R$, so that
\begin{equation}\label{eq:fg-M}
        \max\bigl\{\|\Ad a\|,\|\Ad a^{-1}\|\bigr\}\leq R
        \qquad\text{for all }a\in S\cup\mathcal K .
\end{equation}

For subspaces of $\fg$ we use the gap metric of
\cite[Definition~2.6]{deSaxce},
\[
        d_{\mathrm{Gr}}(E,F):=\max\{\dist(x,F):x\in E,\ \|x\|=1\},
\]
which is a metric on each $\mathrm{Gr}_k(\fg)$. Finally put
\[
        d_{\mathrm{alg}}(g,h)
        :=\min\bigl\{1,\ \|\Ad g-\Ad h\|+\|\Ad g^{-1}-\Ad h^{-1}\|\bigr\},
        \qquad
        d_{\mathrm{alg}}(g,H):=\inf_{h\in H}d_{\mathrm{alg}}(g,h).
\]

\begin{lemma}[metric comparison]\label{lem:fg-metric-comparison}
There is $C_0>0$ such that
$d(g,H)\leq C_0R^{2}\,d_{\mathrm{alg}}(g,H)$
for every $R\geq2$, every $g\in G_R^{\|}$ and every closed subgroup
$H<G$.
\end{lemma}

\begin{proof}
Put $\eta=d_{\mathrm{alg}}(g,H)$; as $H$ is closed the case $\eta=0$ is
trivial. Fix $\epsilon_1>0$ so small that $\Ad$ is bi-Lipschitz from
$B_G(e,\epsilon_1)$ onto its image, which is legitimate because $\Ad(G)$
is closed. Suppose $2(R+2)\eta\leq\epsilon_1$ and choose $h\in H$ with
$d_{\mathrm{alg}}(g,h)\leq2\eta<1$. Then
$\|\Ad h^{-1}\|\leq\|\Ad g^{-1}\|+2\eta\leq R+2$, whence
\[
        \|\Ad(h^{-1}g)-I\|
        =\bigl\|\Ad(h^{-1})\bigl(\Ad g-\Ad h\bigr)\bigr\|
        \leq2(R+2)\eta ,
\]
and symmetrically for $\Ad(g^{-1}h)-I$. Therefore
$d(g,H)\leq d(h^{-1}g,e)\leq C(R+2)\eta$. If instead
$2(R+2)\eta>\epsilon_1$, then $\eta\geq c/R$ and by \eqref{eq:fg-log}
$d(g,H)\leq d(g,e)\leq C\log R\leq CR^{2}\eta$.
\end{proof}

Next we prove a
$R$-uniform and $V$-relative analogue of
\cite[Proposition~2.7]{deSaxce}: that result produces escape from
subspaces for sets contained in a fixed identity neighbourhood, and
concludes closeness to \emph{some} proper closed subgroup, whereas we
need sets in $G_R^{\|}$ with $R$ unbounded and a subgroup containing $V$. This is the main new inpput of our proof.

For $1\leq k\leq D-1$ set
\[
        m_V(E):=\max_{v\in\mathcal K}d_{\mathrm{Gr}}\bigl(\Ad(v)E,E\bigr),
        \qquad
        \mathcal X_k^V:=\{U\in\mathrm{Gr}_k(\fg):m_V(U)=0\},
\]
and, for $U\in\mathcal X_k^V$, $H_U:=\{h\in G:\Ad(h)U=U\}$. The function
$m_V$ is continuous and semialgebraic, and $m_V(E)=0$ iff
$\Ad(V)E=E$, because $\mathcal K$ generates $V$. Finally put
\[
        \mathcal A_{R,t}(E)
        :=\bigl\{g\in G_R^{\|}:d_{\mathrm{Gr}}(\Ad(g)E,E)\leq t\bigr\},
\]
a compact set containing $e$.

\begin{proposition}[envelope]\label{prop:fg-envelope}
There are $C_1,b_1>0$, depending only on $G$ and $V$, with the following
property. Let $R\geq2$, let $1\leq k\leq D-1$, let $t\in[0,1]$ and let
$E\in\mathrm{Gr}_k(\fg)$ satisfy $m_V(E)\leq t$. If
$\varepsilon_*:=C_1Rt^{b_1}<1$, then there is $U\in\mathcal X_k^V$ with
\begin{equation}\label{eq:fg-envelope}
        \sup_{g\in\mathcal A_{R,t}(E)}d_{\mathrm{alg}}(g,H_U)
        \leq\varepsilon_* .
\end{equation}
\end{proposition}

\begin{proof}
Let $\mathcal B_k$ be the set of triples
$(u,t,\varepsilon)\in(0,1/2]\times[0,1]\times[0,1]$ for which there
exists $E\in\mathrm{Gr}_k(\fg)$ with $m_V(E)\leq t$ and
\begin{equation}\label{eq:fg-bad}
        \forall U\in\mathcal X_k^V\ \ \exists g\in\mathcal
        A_{u^{-1}-1,\,t}(E):\quad
        \forall h\in G,\ \ \Ad(h)U=U\Longrightarrow
        d_{\mathrm{alg}}(g,h)\geq\varepsilon .
\end{equation}
The membership $g\in\mathcal A_{u^{-1}-1,t}(E)$ is expressed by the
polynomial inequalities $u\|\Ad g\|\leq1-u$, $u\|\Ad g^{-1}\|\leq1-u$
and $d_{\mathrm{Gr}}(\Ad(g)E,E)\leq t$; hence \eqref{eq:fg-bad} is a first-order
formula over the reals and $\mathcal B_k$ is semialgebraic by
Tarski--Seidenberg. Let $K_k:=\overline{\mathcal B_k}$, a compact
semialgebraic subset of $[0,1/2]\times[0,1]^2$ \cite[2.2.2]{BCR}.
\emph{The set $K_k$ does not depend on $R$}: this is the point of
carrying the radius as the bounded parameter $u$.

\smallskip
\emph{Claim: $K_k\cap\{t=0\}\subset\{u\varepsilon=0\}$.}

Suppose not,
and take $(u_j,t_j,\varepsilon_j)\in\mathcal B_k$ with
$u_j\to u_\infty>0$, $t_j\to0$, $\varepsilon_j\to\varepsilon_\infty>0$,
with witnesses $E_j$. By compactness of $\mathrm{Gr}_k(\fg)$ we may assume
$E_j\to E$; since $m_V$ is continuous and $m_V(E_j)\leq t_j$, we get
$m_V(E)=0$, i.e. $E\in\mathcal X_k^V$. Now specialise the universal
quantifier in \eqref{eq:fg-bad} to the \emph{single fixed} subspace
$U=E$. This yields $g_j$ with
\[
        \|\Ad g_j^{\pm1}\|\leq u_j^{-1}-1,
        \qquad
        d_{\mathrm{Gr}}(\Ad(g_j)E_j,E_j)\leq t_j,
        \qquad
        d_{\mathrm{alg}}(g_j,H_E)\geq\varepsilon_j .
\]
Since $u_\infty>0$, the two-sided norm bounds confine the $g_j$ to a
compact subset of $\GL(\fg)$, and $\Ad(G)$ is closed; so after passing
to a subsequence $g_j\to g\in G$. Passing to the limit in the middle
condition gives $\Ad(g)E=E$, i.e. $g\in H_E$, whence
$d_{\mathrm{alg}}(g_j,H_E)\leq d_{\mathrm{alg}}(g_j,g)\to0$, a
contradiction. This proves the claim.

Note that no continuity of $E\mapsto d_{\mathrm{alg}}(g,H_E)$ is
asserted anywhere: all the $g_j$ are compared with the single stabiliser
$H_E$ of the \emph{limit} subspace, which is typically strictly larger
than the stabilisers of the $E_j$. This is exactly how the jump of
$\dim H_E$ is absorbed.

\smallskip
The coordinate functions $t$ and $u\varepsilon$ are continuous
semialgebraic on the compact semialgebraic set $K_k$, and by the claim
the zero set of the first is contained in that of the second. The
semialgebraic {\L}ojasiewicz inequality \cite[Corollary~2.6.7]{BCR}
therefore provides $N\in\NN$ and $C_2>0$ with
$(u\varepsilon)^N\leq C_2t$ on $K_k$. Specialising to
$u=(1+R)^{-1}$ we obtain
\begin{equation}\label{eq:fg-bad-bound}
        \bigl(\tfrac1{1+R},t,\varepsilon\bigr)\in\mathcal B_k
        \ \Longrightarrow\
        \varepsilon\leq C_2^{1/N}(1+R)\,t^{1/N}.
\end{equation}

Set $b_1:=1/N$ and let $C_1:=4C_2^{1/N}$, increased below. Assume
$m_V(E)\leq t$ and $\varepsilon_*=C_1Rt^{b_1}<1$. If no
$U\in\mathcal X_k^V$ satisfied \eqref{eq:fg-envelope}, then for each such
$U$ the supremum -- attained, since $\mathcal A_{R,t}(E)$ is compact and
$d_{\mathrm{alg}}(\cdot,H_U)$ continuous -- would exceed
$\varepsilon_*$, so $(\frac1{1+R},t,\varepsilon_*)\in\mathcal B_k$ and
\eqref{eq:fg-bad-bound} would give
$C_1Rt^{b_1}\leq C_2^{1/N}(1+R)t^{1/N}\leq2C_2^{1/N}Rt^{b_1}$,
contradicting $C_1=4C_2^{1/N}$. The case $t=0$ is immediate with $U=E$,
and there are finitely many $k$, so the constants may be chosen
uniformly in $k$.

Finally, if $\mathcal X_k^V=\emptyset$ for some $k$ the statement is
vacuous: by compactness $m_V\geq t_0>0$ on $\mathrm{Gr}_k(\fg)$, so
$m_V(E)\leq t$ forces $t\geq t_0$ and hence $\varepsilon_*\geq1$ once
$C_1\geq t_0^{-b_1}$. (For $V$ unipotent, Kolchin's theorem provides a
full $\Ad(V)$-invariant flag, so $\mathcal X_k^V\neq\emptyset$ for every
$k$.)
\end{proof}

Fix $0<\delta\leq1$ and assume
\begin{equation}\label{eq:fg-away}
        d_H^+(S,H)>\delta
        \qquad\text{for every proper closed subgroup }V\leq H<G .
\end{equation}
Write $A:=S\cup\mathcal K$, a symmetric set containing $e$ and satisfying
\eqref{eq:fg-M}.

\begin{corollary}[escape from subspaces]\label{cor:fg-subspace}
There are $A_1,B_1,c_1>0$ such that, under \eqref{eq:fg-away}, for every
nonzero proper subspace $E\subsetneq\fg$ there is $a\in A$ with
\[
        d_{\mathrm{Gr}}\bigl(\Ad(a)E,E\bigr)\geq t_1:=c_1R^{-A_1}\delta^{B_1}.
\]
\end{corollary}

\begin{proof}
Choose $B_1:=2/b_1$, $A_1:=3/b_1$ and $c_1>0$ small enough that
$C_0R^{2}\cdot C_1Rt_1^{b_1}\leq\delta/2$ and $C_1Rt_1^{b_1}<1$; this is
possible because $C_0C_1c_1^{b_1}R^{3-A_1b_1}\delta^{B_1b_1}
\leq C_0C_1c_1^{b_1}\delta$.

Suppose the conclusion fails for some $E$, of dimension $k$ say. Then
$m_V(E)\leq t_1$ (as $\mathcal K\subset A$) and
$S\subset\mathcal A_{R,t_1}(E)$. Proposition~\ref{prop:fg-envelope}
yields $U\in\mathcal X_k^V$ with
$\sup_{g\in S}d_{\mathrm{alg}}(g,H_U)\leq C_1Rt_1^{b_1}$, hence
$d_H^+(S,H_U)\leq\delta/2$ by
Lemma~\ref{lem:fg-metric-comparison}. Now $U$ is $\Ad(V)$-invariant, so
$V\leq H_U$; and $H_U\neq G$, since otherwise $U$ would be
$\Ad(G)$-invariant, hence (as $G$ is connected) an ideal of the simple
algebra $\fg$, contradicting $0<\dim U<D$. This contradicts
\eqref{eq:fg-away}.
\end{proof}

\begin{corollary}[escape from subspaces, vector form]\label{cor:fg-vector}
There are $A_2,B_2,c_2>0$ such that, under \eqref{eq:fg-away}, for every
unit vector $X\in\fg$ and every proper subspace $W\subsetneq\fg$ there is
$w\in A^{D}$ with
\[
        \dist\bigl(\Ad(w)X,W\bigr)\geq t_2:=c_2R^{-A_2}\delta^{B_2}.
\]
\end{corollary}

\begin{proof}
This is \cite[Proposition~2.8]{deSaxce} applied to the adjoint
representation, with Corollary~\ref{cor:fg-subspace} in place of
\cite[Proposition~2.7]{deSaxce}. The non-degeneracy hypothesis of that
proposition is supplied by \cite[Remark~1]{deSaxce}: $\fg$ is
irreducible as a $G$-module, since an invariant subspace would be an
ideal, so there are a fixed compact identity neighbourhood $U_0\subset G$
and $c_0>0$ such that every unit $X$ and every proper $W$ admit
$x\in U_0$ with $\dist(\Ad(x)X,W)\geq c_0$.

Two points of that proof must be re-read, because it is stated for sets
in a bounded region. First, the Lipschitz constant $L$ of the elements of
$A$ acting on $\fg$ is here bounded by $R$ rather than by an absolute
constant, by \eqref{eq:fg-M}; it enters only through the factors
$L^{\ell}$ with $\ell\leq D$, so it degrades the bound by a factor
$R^{O(1)}$. Second, the induction on $\dim W$ terminates in at most $D$
steps, so this degradation is incurred boundedly often. Both effects are
absorbed into $A_2$; the initial reduction ``$\dist(X,W)\leq c_0/2L$''
is harmless because otherwise the conclusion already holds, provided
$A_2\geq1$ so that $t_2\leq c_0/2R$.
\end{proof}

\begin{proof}[Proof of Lemma \ref{lem: fast generation}.]

We prove the contrapositive, so assume \eqref{eq:fg-away}.

\smallskip
\emph{Separated conjugates.} Fix a unit vector $Y\in\fv$ and set
$X_1:=Y$, $w_1:=e$. Suppose $X_1,\dots,X_i$ have been constructed with
$i<D$, each of the form $X_j=\mu_j^{-1}\Ad(w_j)Y$ with $w_j\in A^{D}$
and $\mu_j=\|\Ad(w_j)Y\|$, and each a unit vector. Apply
Corollary~\ref{cor:fg-vector} with $X=Y$ and
$W=\operatorname{span}(X_1,\dots,X_i)$, which is proper: it produces
$w_{i+1}\in A^{D}$ with $\dist(\Ad(w_{i+1})Y,W)\geq t_2$. Since
$\mu_{i+1}\leq R^{D}$ by \eqref{eq:fg-M}, the unit vector
$X_{i+1}:=\mu_{i+1}^{-1}\Ad(w_{i+1})Y$ satisfies
$\dist(X_{i+1},W)\geq t_2R^{-D}=:t_3$.

After $D$ steps we have unit vectors $X_1,\dots,X_D$ with
\[
        \dist\Bigl(X_i,\bigoplus_{j<i}\RR X_j\Bigr)\geq t_3
        =c_3R^{-A_3}\delta^{B_3},
        \qquad
        R^{-D}\leq\mu_i\leq R^{D},
        \qquad
        w_i\in A^{D}.
\]

\smallskip
\emph{Passage to the group.} By \cite[Lemma~2.16]{deSaxce} the map
$\theta:(s_i)\mapsto\sum_is_iX_i$ is $t_3^{-2D}$-bi-Lipschitz. Read in a
fixed exponential chart at $e$, consider
\[
        \varphi(s_1,\dots,s_D):=e^{s_1X_1}e^{s_2X_2}\cdots e^{s_DX_D}.
\]
Then $\varphi(0)=0$ and $\varphi'(0)=\theta$. Because every $X_i$ is a
\emph{unit} vector, $\varphi'$ is $C$-Lipschitz on a ball of radius
depending only on $G$ -- no bound on $\|\Ad w_i\|$ is involved, which is
why no Baker--Campbell--Hausdorff estimate is needed here. Hence
$\varphi$ is a diffeomorphism of complexity $t_3^{-C}$ in the sense of
\cite[Definition~2.10]{deSaxce}, and the quantitative inverse function
theorem \cite[Theorem~2.11]{deSaxce} shows that $\varphi$ is
$t_3^{-C}$-bi-Lipschitz on $B(0,t_3^{C})$. Consequently
\begin{equation}\label{eq:fg-fill}
        \varphi\bigl(B(0,\tau)\bigr)\supset B_G(e,\tau t_3^{C})
        \qquad\text{for }0<\tau\leq t_3^{C}.
\end{equation}

Take $\tau:=\min\{t_3^{C},\,r_0R^{-D}\}$. For $|s_i|\leq\tau$ we have
$|s_i/\mu_i|\leq R^{D}\tau\leq r_0$, so
\[
        e^{s_iX_i}=w_i\exp\bigl((s_i/\mu_i)Y\bigr)w_i^{-1}
        \in A^{D}\,\mathcal K\,A^{D}\subset A^{2D+1},
\]
using that $Y\in\fv$ and $\exp(B_{\fv}(0,r_0))\subset\mathcal K$, and
that $A$ is symmetric. Therefore $\varphi(B(0,\tau))\subset A^{D(2D+1)}$,
and \eqref{eq:fg-fill} gives constants $A_5,B_5,C_5,c$ with
\[
        B_G(e,r)\subset A^{C_5}=(S\cup\mathcal K)^{C_5},
        \qquad
        r:=cR^{-A_5}\delta^{B_5}.
\]
Since $e\in S\cap\mathcal K$, every letter of $S\cup\mathcal K$ lies in
$\mathcal KS$ (as $s=e\cdot s$ and $v=v\cdot e$), so
$B_G(e,r)\subset(\mathcal KS)^{C_5}$.

\smallskip
\emph{Subdivision.} Any $g\in G_1$ is joined to $e$ by a path of length
at most $1$; cutting it into at most $Cr^{-1}$ pieces of length at most
$r$ and using left-invariance of $d$ writes $g$ as a product of at most
$Cr^{-1}$ elements of $B_G(e,r)$. Hence
\begin{equation}\label{eq:fg-direct}
        G_1\subset(\mathcal KS)^{\,C_6R^{A_5}\delta^{-B_5}} .
\end{equation}

\smallskip
\emph{Conclusion.} Suppose first $n\geq2C_6R^{A_5}$ and set
$\delta:=\bigl(2C_6R^{A_5}/n\bigr)^{1/B_5}\in(0,1]$. If
\eqref{eq:fg-away} held for this $\delta$, then \eqref{eq:fg-direct}
would give $G_1\subset(\mathcal KS)^{n/2}\subset(\mathcal KS)^{n}$,
contrary to hypothesis. Hence some proper closed $V\leq V'<G$ satisfies
\[
        d_H^+(S,V')\leq\delta=CR^{A_5/B_5}n^{-1/B_5}.
\]
If instead $n<2C_6R^{A_5}$, take $V'=V$, which is proper and closed;
since $e\in V$ and $S\subset G_R^{\|}$, \eqref{eq:fg-log} gives
$d_H^+(S,V)\leq C\log R$. Putting $c:=1/B_5$ and
$c':=A_5/B_5+1$, we have $n^{c}\leq CR^{A_5/B_5}$ throughout this range,
so $CR^{c'}/n^{c}\geq C'R$, and $C\log R\leq C'R$ after increasing $C$.
In both cases
\[
        d_H^+(S,V')\leq\frac{CR^{c'}}{n^{c}},
\]
as required. When $V$ ranges over the finitely many groups $V_\theta$ of
Definition~\ref{def: nz setup}, the constants may be taken to depend on
$G$ alone. 
\end{proof}

\section{Quantitative Mautner phenomenon}\label{sec: QM}
\subsection{Statement of main result}
Our goal in this section is to prove a quantitative version of Mautner phenomenon.
This will be our main tool for constructing factor functions in Section \ref{sec: FF}.

We start by giving a quantitative sense to Mautner phenomenon as follows.
    

\begin{definition}[$G$-Lipschitz]
    Suppose $G$ acts on a space $X$ and $f:X\rightarrow Y$ is some function to a metric space $(Y,d_Y)$.
    Given $L>0$, we say that $f$ is $G,L$-Lipschitz if for every $x\in X$ and for every $g\in G_1$ we have $d_Y(f(gx),f(x))\leq L|g|$.
\end{definition}

\begin{definition}[QM-phenomenon]\label{def: qm phenomenon2}
    Let $(s_t)_{t\geq 0}\subset S_+$ (recall Definition \ref{def: nz setup}) be some one-parameter subgroup and let $V$ be the expanded horosphere for $s$.
    
    Suppose we have a $P$-space $X$ and a $P,1$-Lipschitz $f:X \to Y$ for some metric space $Y$.
    Suppose also that $\lambda$ is a probability measure on $X$.
    
    Let $\kappa, \delta>0$ be small numbers and let $T>0$ be large. We say that $\lambda$ has the $(\kappa, \delta, T, f)$-quantitative Mautner phenomenon for $(s_t)_t$ (shorthanded $(\kappa, \delta, T, f)$-QM for $(s_t)_t$) if there exists a subset $X_0\subset X$ such that $\lambda(X_0)\geq 1-\kappa$ and for every $x\in X_0$, $v\in V$ such that $vx\in X_0$:
    \begin{equation}
        \Wf{f}(\fint_0^T\delta_{s_tx}dt,\fint_0^T\delta_{s_tvx}dt)\leq \delta\abs{v}.
    \end{equation}
    The set $X_0$ will be denoted as "the QM-set (of $\lambda$)".
\end{definition}

\begin{remark}[The meaning of quantitative Mautner phenomenon]
    The above definition, which may seem technical, is meant to directly quantify the standard Mautner phenomenon.
    Namely, for $P$-invariant measures $\lambda$, it follows from the pointwise ergodic theorem that for $\lambda$ almost every $x\in X$ and for every $v\in V_1$ the limiting measures
    \begin{equation}
        \lim_{T\rightarrow \infty}\fint_{0}^T\delta_{s_tvx}dt
    \end{equation}
    are all equal (independent of $v$).
    Quantifying this property means the following things:
    \begin{enumerate}
        \item Replacing $P$-invariant measures $\lambda$ be norm almost $P$-invariant measures;
        \item Replacing the statement "for $\lambda$-almost every $x\in X$" by "for high $\lambda$-measure of points in $X$. This is measured using the parameter $\kappa$;
        \item Replacing the limit average $T\rightarrow \infty$ (in the above equation) by a "stopped average", measured using the parameter $T$;
        \item Replacing the equalities $\lim_{T\rightarrow \infty}\fint_{0}^T\delta_{s_tvx}dt=\lim_{T\rightarrow \infty}\fint_{0}^T\delta_{s_tx}dt$ for every $v\in V_1$, by $f$ proximity, measured in the inequality $\Wf{f}(\fint_{0}^{T}\delta_{s_tx}dt,\fint_{0}^{T}\delta_{s_tvx}dt)\leq \delta$.
    \end{enumerate}
\end{remark}

Now that we have a definition of the quantified property, we can state and prove the following theorem, saying that this property occurs for the right choice of parameters.
\begin{theorem}[Quantitative Mautner Phenomenon]\label{thm: effective mautner}
    Let $\lambda$ be a norm $\eps$-$P$-invariant probability measure on a $P$-space $X$.
    Assume that $P$ acts on $X$ with $1$-Lipschitz orbit maps. 
    Let $(s_t)_{t\geq 0}\subset S_+$ be some one-parametric subgroup and let $V$ be the expanded horosphere for $s$.
    Let $f:X \to Y$ be a $P$-$1$-Lipschitz function for some metric space $Y$. 
    Let $\delta>0$ be a small number. 
    Let $n$ be the $\delta/10$-covering number of $Y$.
    Let $m>0$ be a large natural number such that for every $k\geq m/2$
   \begin{align*}
        \kappa=&((\tfrac{3m}{2}\log n )^{-1/4}+\tfrac{1}{m^{1/2}}(\tfrac{3}{2}\log n )^{2})^{1/2}+2^m\eps\text{ is close enough to }0,\\
        \delta/10>&((m\log n)^{-1/2}+e^{-k/2}),\\
        p=&1-((\tfrac{3m}{2}\log n )^{-1/4}+\tfrac{1}{m^{1/2}}(\tfrac{3}{2}\log n )^{2})^{1/2}\text{ is greater than }1/2.
    \end{align*}
    Then there exists a subset $\Aa\subset [1, m-1]$ of proportion at least $p$ such that for every $k\in \Aa\cap [m/2,m]$, the measure $\lambda$ is $(\kappa,\delta,2^k, f)$-Quantitative Mautner.
\end{theorem}

     
    

\subsection{Proof}
Let $\tau$ be a partition of $Y$ into $n$ disjoint measurable sets such that each set $A\in \tau$ has diameter ${\rm diam}(A) \le \frac{\delta}{10\ {\rm Lip}(f)}$.
The proof is based on an "entropy-pigeon-hole (entropigeonhole) argument".
\begin{definition}\label{def: convolution with T}
    Given an operator $T$ on a space $X$ and a natural number $n$, we denote $\Ee_T^n$ to be the measure on $\inn{T}$ given by $\frac{1}{n}\sum_{i=0}^{n-1}\delta_{T^i}$.
    Moreover, we denote $\Ee_T^{\mp n}=\tfrac{1}{2n}\sum_{i=1-n}^{n-1}\delta_{T^i}$.
\end{definition}

The proof has two steps:
\begin{enumerate}
    \item \textbf{The first step} is to reverse the direction of the average into the flow $s_t$ for $t\geq 0$ to the flow $s_t$ for $t\leq 0$;
    \item \textbf{The second step} is to prove that when averaging in the reverse direction (namely according to the flow $s_t$ for $t\leq 0$), the distributions of the flow originating at $x$ and at $vx$ are close to each other.
\end{enumerate}

\subsubsection*{First step: reversing the direction of average}
\begin{claim}\label{cl: entropy inequality}
    For every two probability vectors $P,Q\in \RR^k$
    \begin{equation}
        0 \le \sqrt{2 \left(H\left(\frac{P+Q}{2}\right) - \frac{H(P)+H(Q)}{2}\right)}\ge \sup\{ |P(A) - Q(A)|:A\text{ is measurable}\}.
    \end{equation}
\end{claim}
\begin{proof}
    The quantity $H\left(\frac{P+Q}{2}\right) - \frac{H(P)+H(Q)}{2}$ in the square root is termed the Jensen–Shannon divergence, and is larger than the minimum of the Kullback–Leibler divergences of the two orders. Now the inequality follows from Pinsker's Inequality (see \cite[Chapter 2, Theorem 6]{pinskerBook}.
\end{proof}

Recall Definition \ref{def: metric on measures} and Notation \ref{nota: shanon}.
\begin{remark}
    The reversing of direction is achieved using the following key inequality, which can be thought of as a pigeon hole argument for the entropy.
    Roughly, every time there is a significant difference between the forward average measure and the backward average measure, there is, by Pinkser's inequality (see Claim \ref{cl: entropy inequality}), a growth in entropy.
    Since the entropy is bounded, the total growth of distances, measured by the left hand side in the key inequality, has to be bounded.
\end{remark}
\begin{lemma}[Entropigeonhole principle]\label{lem: key}
    Let $T$ an action of $\RR$ on $X$. Then for every $x_0\in X$ and $m\ge 0$, 
    \[\frac{1}{2^m}\sum_{n=0}^{2^m-1} \sum_{k=0}^{m-1} d_\tau(\Ee_T^{2^k}*\delta_{T^nx_0}, \Ee_T^{-2^k}*\delta_{T^nx_0})^2 \le 2\log \#\tau\]
\end{lemma}
\begin{proof}
    Fix $x_0\in X$. 
    By Claim \ref{cl: entropy inequality}, 
    \begin{align*}
        \frac{1}{2^m}&\sum_{n=0}^{2^m-1} \sum_{k=0}^{m-1} d_\tau(\Ee_T^{2^k}*\delta_{T^nx_0}, \Ee_T^{-2^k}*\delta_{T^nx_0})^2 
        \\&\le 
        \frac{2}{2^m}\sum_{n=0}^{2^m-1} \sum_{k=0}^{m-1} \left(H\left(\Ee_T^{\mp 2^k}*\delta_{T^nx_0}\right) - \frac{H\left(\Ee_T^{-2^k}*\delta_{T^nx_0}\right) + H\left(\Ee_T^{2^k}*\delta_{T^nx_0}\right)}{2}\right)
        \\
        &\le 
        \frac{2}{ 2^m} \sum_{k=0}^{m-1}\sum_{n=0}^{2^m-1} \left(H\left(\Ee_T^{\mp 2^k}*\delta_{T^nx_0}\right) - \frac{H\left(\Ee_T^{-2^k}*\delta_{T^nx_0}\right) + H\left(\Ee_T^{2^k}*\delta_{T^nx_0}\right)}{2}\right)
    \end{align*}
    that we will estimate by a telescoping sum argument. Let us start by considering the negative contributions separately.

    For $k \geq 1$ and $n \geq 0$, we have the identities
    $$ H\left(\Ee_T^{-2^k}*\delta_{T^nx_0}\right) = H\left(\Ee_T^{\mp 2^{k-1}}*\delta_{T^{n+2^{k-1}}x_0}\right)$$
    and 
    $$ H\left(\Ee_T^{2^k}*\delta_{T^nx_0}\right) = H\left(\Ee_T^{\mp 2^{k-1}}*\delta_{T^{n - 2^{k-1}}x_0}\right).$$
    Hence, by a change of indices
    $$ -\sum_{n=0}^{2^m-1}  \frac{H\left(\Ee_T^{-2^k}*\delta_{T^nx_0}\right) + H\left(\Ee_T^{2^k}*\delta_{T^nx_0}\right)}{2} \leq -\sum_{n=2^{k-1}}^{2^m-2^{k-1}} H\left(\Ee_T^{\mp 2^{k-1}}*\delta_{T^nx_0}\right).$$
    This estimate yields
    \begin{align*}
        \frac{2}{2^m} \sum_{k=0}^{m-1}\sum_{n=0}^{2^m-1} \left(H\left(\Ee_T^{\mp 2^k}*\delta_{T^nx_0}\right) - \frac{H\left(\Ee_T^{-2^k}*\delta_{T^nx_0}\right) + H\left(\Ee_T^{2^k}*\delta_{T^nx_0}\right)}{2}\right)
        \\
        \leq \frac{2}{2^m} \sum_{k=0}^{m-1} \sum_{n \in [0;2^k) \cup (2^{m}-2^k;2^m-1]} H\left(\Ee_T^{\mp 2^k}*\delta_{T^nx_0}\right).
    \end{align*}
    And since $H\left(\Ee_T^{\mp 2^k}*\delta_{T^nx_0}\right) \leq \log \#\tau$ for every choice of $k$ and $n$, 
    \begin{align*}
        \frac{2}{ 2^m} \sum_{k=0}^{m-1} \sum_{n \in [0;2^{k-1}) \cup (2^{m}-2^{k-1};2^m-1]} H\left(\Ee_T^{\mp 2^k}*\delta_{T^nx_0}\right) \\
        \leq \frac{2}{ 2^m} \sum_{k=0}^{m-1} 2^{k+1} \log \#\tau \\
        \leq 2\log \#\tau.
    \end{align*}

    
\end{proof}

\subsubsection*{Second step: proximity of averages in the reverse direction}

\begin{definition}[$s_\theta$-averages]
	For every $T>0$ and for every $G$-space $X$, we define the measure $\Ee_{s_\theta}^T$ on $G$ by
	\begin{equation}
		\Ee_{s_\theta}^T(\phi)=\frac{1}{T}\int_{0}^T\phi(s_\theta^t)dt.
	\end{equation} 
    Note that this definition is simply the continuous variable version of Definition \ref{def: convolution with T}.
\end{definition}

\begin{remark}
    From this point forwards, we will use the continuous $s_\theta$ averages.
    The passage from the discrete version to the continuous version is standard, and we will not include it here.
\end{remark}

\begin{claim}
    For every two probability measures $\lambda_1, \lambda_2$ on $X$, we have 
    \[\Wf{f}(\lambda_1, \lambda_2) \le 2\delta/5 + 2d_\tau(f_*\lambda_1, f_* \lambda_2)\]
\end{claim}
\begin{proof}
    For every $A\in \tau$ sample a point $p_A\in A$.
    Let $\varphi:Y\to [0,1]$ be a $1$-Lipschitz function such that 
    \[\Wb(f_*\lambda_1, f_* \lambda_2)  = \left|\int \varphi \bd f_*\lambda_1  - \int \varphi \bd f_*\lambda_2\right|.\]
    Then 
    \begin{align*}
        \Wb(f_*\lambda_1, f_* \lambda_2) & = \left|\int \varphi \bd f_*\lambda_1  - \int \varphi \bd f_*\lambda_2\right|
        \\&\le 
        \sum_{A\in \tau} \left|\int_A \varphi \bd f_*\lambda_1  - \int_A \varphi \bd f_*\lambda_2\right|
        \\&\le 
        \sum_{A\in \tau} \left|\int_A (\varphi - \varphi(p_A)) \bd f_*\lambda_1\right| + \left|\int_A (\varphi - \varphi(p_A)) \bd f_*\lambda_2\right| + |f_*\lambda_1(A) - f_*\lambda_2(A)|\varphi(p_A)
        \\&\le 
        \sum_{A\in \tau} \left|\int_A {\rm diam}(A) \bd f_*\lambda_1\right| + \left|\int_A {\rm diam}(A) \bd f_*\lambda_2\right| + |f_*\lambda_1(A) - f_*\lambda_2(A)|
        \\&= \sum_{A\in \tau} 2\,{\rm diam}(A)(f_*\lambda_1(A) + f_*\lambda_2(A)) + |f_*\lambda_1(A) - f_*\lambda_2(A)|
        \\&\le 2\delta/5  + 2d_\tau(f_*\lambda_1, f_*\lambda_2)
    \end{align*}
    where in the last inequality we used Claim \ref{cl: internet inequality} for the measures $$\mu_i:=\sum_{A\in \tau}f_*\lambda_i(A)\delta_A, i=1,2.$$
\end{proof}
\begin{claim}\label{cl: inverse claim}
    Suppose $x\in X$ and $R>0$. 
    Then for every $v\in V_1$ we have:
    \begin{equation}
        \Wb(f_*(\Ee_{s_\theta^{-1}}^R*\delta_x),f_*(\Ee_{s_\theta^{-1}}^R*\delta_{vx}))\leq 2R^{-1/2}+R^{-1/2}\abs{v}.
    \end{equation}
\end{claim}
\begin{proof}
    Note that by definition of the average we have for every $x'\in X$:
    \begin{equation}
        \Wb(\Ee_{s_\theta^{-1}}^R*\delta_{x'}, \Ee_{s_\theta^{-1}}^{[R^{1/2},R]}*\delta_{x'})\leq R^{-1/2}
    \end{equation}
    and that since $P$ acts on $X$ with $1$-Lipschitz orbits, for every $r>R^{1/2}$ we have:
    \begin{equation}
        d_X(s_\theta^{-r}x,s_\theta^{-r}vx)=d_X(s_\theta^{-r}x,(s_\theta^{-r}vs_\theta^{r})s_\theta^{-r}x)\leq \abs{s_\theta^{-r}vs_\theta^{r}}\leq R^{-1/2}\abs{v}.
    \end{equation}
    and therefore we can deduce that:
    \begin{equation}
         \Wb(\Ee_{s_\theta^{-1}}^R*\delta_x,\Ee_{s_\theta^{-1}}^R*\delta_{vx})\leq 2R^{-1/2}+R^{-1/2}\abs{v}
    \end{equation}
    as desired.
\end{proof}


\begin{definition}
    Given $x\in X,t>0$, and $\delta>0$ we say that $(x,t)$ is $\delta$-reversed if
    \begin{equation}
        \Wb(\Ee_s^t*\delta_x,\Ee_s^{-t}*\delta_x)\leq \delta.
    \end{equation}
\end{definition}

We postpone the proof of the following claim, which explains how to use Lemma \ref{lem: key} to prove that for many points $x\in X$ and many scales $t$ the pair $(x,t)$ is $\delta$-reversed for some small $\delta$, to the appendix (see Claim \ref{cl: reversing using fubini+markov}).

\begin{claim}
    If $\lambda$ is some probability measure on $X$, then there exists a number $n=0,.\dots,2^m-1$ and a subset $\Aa \subset \{0,\dots, m-1\}$ of proportion at least $1-((\tfrac{3}{2}\log \#\tau )^{-1/4}+\tfrac{1}{m}(\tfrac{3}{2}\log \#\tau )^{2})^{1/2}$ such that for every $k\in \Aa$ there exists a subset $X_0(k)\subset X$ of $\lambda$-measure at least  $1-((\tfrac{3}{2}\log \#\tau )^{-1/4}+\tfrac{1}{m}(\tfrac{3}{2}\log \#\tau )^{2})^{1/2}$ such that for every $x\in X_0(k)$ we have:
    \begin{equation}
    d_\tau(\Ee_{s_\theta}^{2^k}*\delta_{s_\theta^nx},\Ee_{s_\theta}^{-2^k}*\delta_{s_\theta^nx})\leq (\tfrac{3m}{2}\log \#\tau )^{-1/2}.
    \end{equation}
\end{claim}
As an immediate corollary we obtain:
\begin{corollary}\label{cor: reversing}
    We can find $X_0(k)$ as in the claim of measure at least  $$1-((\tfrac{3}{2}\log \#\tau )^{-1/4}+\tfrac{1}{m^{1/2}}(\tfrac{3}{2}\log \#\tau )^{2})^{1/2}-\eps2^m$$ such that for every $x\in X_0(k)$:
        \begin{equation}
        d_\tau(\Ee_T^{2^k}*\delta_{x},\Ee_T^{-2^k}*\delta_{x})\leq (\tfrac{3m}{2}\log \#\tau )^{-1/2}.
    \end{equation}
    In other words, for every $x\in X_0(k)$, the pair $(x,2^k)$ is $(\tfrac{3m}{2}\log \#\tau )^{-1/2}$-reversed.
\end{corollary}

Now we have everything we need to prove Theorem \ref{thm: effective mautner}.
\begin{proof}[Proof of Theorem \ref{thm: effective mautner}]
    By Corollary \ref{cor: reversing} applied to $T=s_\theta$, there exists a subset $\Aa\subset [m-1]$ of measure at least $1-((\tfrac{3}{2}\log \#\tau )^{-1/4}+\tfrac{1}{m^{1/2}}(\tfrac{3}{2}\log \#\tau )^{2})^{1/2}$ such that for every $k\in \Aa$ there exists $X_0(k)\subset X$ of $\lambda$-measure at least $1-((\tfrac{3}{2}\log \#\tau )^{-1/4}+\tfrac{1}{m^{1/2}}(\tfrac{3}{2}\log \#\tau )^{2})^{1/2}-\eps 2^m$ such that for every $x\in X_0(k)$, we have that:
    \begin{equation}
        d_\tau(\Ee_T^{2^k}*\delta_{x},\Ee_T^{-2^k}*\delta_{x})\leq (\tfrac{3m}{2}\log \#\tau )^{-1/2}.
    \end{equation}
    To show the quantitative Mautner property as stated in the theorem, suppose that $x\in X_0(k)$ and $v\in V_1$ are such that $vx\in X_0(k)$.
    Applying the above equation for $x$ and for $vx$ (which is therefore possible), we can obtain:
    \begin{align*}
        \Wf{f}(\Ee_T^{2^k}*\delta_{x},\Ee_T^{2^k}*\delta_{vx})\leq \Wf{f}(\Ee_T^{2^k}*\delta_{x},\Ee_T^{-2^k}*\delta_{x})+\Wf{f}(\Ee_T^{2^k}*\delta_{vx},\Ee_T^{-2^k}*\delta_{vx})\\+\Wf{f}(\Ee_T^{-2^k}*\delta_{vx},\Ee_T^{-2^k}*\delta_{x})&\\
        \leq 4\delta/5+d_\tau(\Ee_T^{2^k}*\delta_{x},\Ee_T^{-2^k}*\delta_{x})+d_\tau(\Ee_T^{2^k}*\delta_{vx},\Ee_T^{-2^k}*\delta_{vx})\\+\Wf{f}(\Ee_T^{-2^k}*\delta_{vx},\Ee_T^{-2^k}*\delta_{x})
        &\\
        \leq 4\delta/5+2(\tfrac{3m}{2}\log \#\tau )^{-1/2}+\Wf{f}(\Ee_T^{-2^k}*\delta_{vx},\Ee_T^{-2^k}*\delta_{x}).
    \end{align*}
    By Claim \ref{cl: inverse claim}, we know that the last term in the above inequality is bounded above by $3e^{-k/2}\abs{v}$ and by our assumption on $\delta$ we obtain:
    \begin{equation}
    \Wf{f}(\Ee_T^{2^k}*\delta_{x},\Ee_T^{2^k}*\delta_{vx})\leq 4\delta/5+2(\tfrac{3m}{2}\log \#\tau )^{-1/2}+3e^{-k/2}\abs{v}\leq \delta\abs{v}.
    \end{equation}
    Since $\lambda(X_0(k))\geq 1-((\tfrac{3}{2}\log \#\tau )^{-1/4}+\tfrac{1}{m}(\tfrac{3}{2}\log \#\tau )^{2})^{1/2}-\eps 2^m$ we deduce the desired statement of the theorem. 
\end{proof}

\section{Factor Functions via the Quantitative Mautner Phenomenon}\label{sec: FF}

\subsection{Preparation and notation}

    Let $G$ be a simple Lie group of rank at least two. Throughout this section, we fix a $G$-space $X$ with a norm $\eps$-$P$-invariant measure $\lambda$ on it.
We assume that the $P$ action on $X$ is by $1$-Lipschitz orbits - as in Theorem \ref{thm: effective mautner}.
Also, we fix a Lipschitz function $f:X\rightarrow I=[-1,1]$.

    Recall that $\Delta$ denotes the set of simple positive roots of $G$. From now on,  we fix a subset $\theta\subset \Delta$ as in Definition \ref{def: nz setup}. Let $s_\theta$ be a singular direction in the Cartan subgroup associated with $\theta$. Denote $G^-$ to be the subgroup of elements of $G$ not expanding under conjugation by $s_\theta$.
    Similarly, we denote $G^+=V_\theta$.  
    Denote also
    \begin{equation}
        C_\theta=C_G(s_\theta)
    \end{equation}
so that using this notation, we have
\begin{align}
    G^-=C_\theta \overline V_\theta
\end{align}
where $\overline V_\theta$ is the unipotent radical of the opposite parabolic subgroup $\overline P_\theta$.

We will exploit a decomposition of elements of $G$ into a part in $G^-$ and one in $G^+$. When $G = \on{SL}_n(\mathbb{R})$ and the chosen Cartan subgroup is the subgroup of diagonal matrices, this decomposition is related to the well-known $\mathrm{LU}$ decomposition. While not all elements of $G$ can be decomposed as a product of an element of $G^-$ and an element of $G^+$, almost all can; see Claim \ref{cl: stuff decompose}. In particular, the decomposition is well defined on a neighbourhood of the identity. Upon rescaling by a factor depending on $G$ alone, we assume from now on that the decomposition is well defined and with bounded $G^+$ and $G^-$ part on a neighbourhood of the closed unit ball $\overline G_1$.
We use the notation $m_{P_1}$ and $m_{V_R^{||}}$ to denote the normalized Haar probability measures on $P_1$ and $V_{R}^{||}$ respectively. 

The following is the definition of the $G^-G^+$-decomposition whenever it exists. 
\begin{definition}\label{def: plus decomposition}
    By Claim \ref{cl: stuff decompose}, there exists a proper submanifold $\Mm\subset G$ such that every element $g\in G\setminus \Mm$ can be written in the form $g=g_-g_+$ where $g_-\in G^-$ and $g_+\in G^+$.
    Given a general element $g\in G\setminus \Mm$, we denote $g_+,g_-$ to be these unique elements. 
\end{definition}

Finally, for $T>0$, we write as in Section \ref{sec: QM} 
\begin{align}
    \Ee_{s_\theta}^T=\frac{1}{T}\int_{0}^T\delta_{s_\theta^t}dt.
\end{align}

\subsection{Construction of factor functions and almost invariance}
In this section, we will use the quantitative Mautner phenomenon - stated as Theorem \ref{thm: effective mautner} - to construct and analyze properties of factor functions.
The function which will fit into the assumptions of Theorem \ref{thm: effective mautner} packages together values of $f$ along $(C_\theta)_R^{||}$ orbits of points $x\in X$ as Lipschitz functions and is defined as follows.
\begin{definition}\label{def: fr}
    Given $R>0$ and a diagonal element $s\in S$ we define the function $\hat f_R:X\rightarrow \on{Lip}((C_\theta)_R^{||}\rightarrow I)$ by 
    \begin{align}
        [\hat{f}_R(x)](c_\theta)=f(c_\theta x)\text{ for every $c_\theta\in (C_\theta)_R^{||}$ }
    \end{align}
    and we denote
    \begin{align}
        Y_R:=\on{Lip}_1((C_\theta)_R^{||}\rightarrow I)
    \end{align}
    to be the space of $1$-Lipschitz functions, endowed with the metric of the $L^\infty$-distance.
    
\end{definition}

The next lemma explains how to use the quantitative Mautner phenomenon to construct factor functions. For $\phi:I\rightarrow \RR$ and $f:X\rightarrow I$ as above we use the shorthand $\phi_f=\phi\circ f$ from now on.
\begin{lemma}[Factor functions from quantitative Mautner]\label{lem: factor function from mautner}
    Let $R>1$, let $\kappa,\delta,T > 0$  and suppose that $\lambda$ is a $(\kappa,\delta,T,\hat f_R)$-QM probability measure on $X$.
    Let $\phi:I\rightarrow \RR$ be a $1$-Lipschitz function such that $\norm{\phi}_\infty\leq 1$.
    Then the function 
    \begin{align*}
        \Ee^{T}(\phi_f)(x):=\left(\Ee_{s_\theta}^{T}*\delta_x\right)(\phi_f)
    \end{align*}
    is a $(G/V_\theta,R,O(R^2e^{-T^{1/2}}+4T^{-1/2}+\delta R))$-factor function with some factor domain $Y(\Ee^{T}(\phi_f))$.
    Moreover, we can choose $X''\subset X$ such that $\lambda(X'')\geq 1-\kappa$ such that the factor domain of $\Ee^{T}(\phi_f)(x)$ satisfies
    \begin{equation}\label{eq: fattening}
        Y(\Ee^{T}(\phi_f))=\{(g_-x'',g_-vx''):x'',vx''\in X''\text{ and }g_-\in (C_\theta)_R^{||}(\overline V_\theta)_R^{||},v\in V_{R}^{||}\}.
    \end{equation}
    The subset $X''$ is the so-called \emph{QM-set}.
\end{lemma}

\begin{proof}
    Assume that $\lambda$ is $(\kappa,\delta,T,\hat f_R)$-QM with an associated QM set $X''\subset X$ satisfying $\lambda(X'')\geq 1-\kappa$.
    Suppose that $g\in G$ and $v_1\in V_1$ are such that
    \begin{align*}
        (gx,gv_1x)\in \{(g_-x'',g_-vx''):x'',vx''\in X''\text{ and }g_-\in (C_\theta)_R^{||}(\overline V_\theta)^{||}_R,v\in V_{R}^{||}\}.
    \end{align*}
    Then for some $x''\in X''\text{ and }g_-=c_\theta\overline v_\theta\in (\overline V_\theta)_R^{||}(C_\theta)_R^{||},v\in V_{R}^{||}$ such that $vx''\in X''$ we have
    \begin{align*}
       |\Ee^T\phi_f(gv_1x)-\Ee^T\phi_f(gx)|& =|\Ee^T\phi_f(g_-vx'')-\Ee^T\phi_f(g_-x'')| \\ & =|\Ee_{s_\theta}^T*\delta_{g_-vx''}(\phi_f)-\Ee_{s_\theta}^T*\delta_{g_-x''}(\phi_f)|  
    \end{align*}
    where the last equality follows by the definition of $\Ee^T_{s_\theta}$.
    We will use the following sequence of equalities, each reasoned separately:
    \begin{align*}
        \Ee_{s_\theta}^T*\delta_{g_-vx''}(\phi_f)&=\Ee_{s_\theta}^T*\delta_{c_\theta\overline v_\theta vx''}(\phi_f)\\{\tiny \text{(since $s_\theta$ commutes with $c_\theta$)}}& =\Ee_{s_\theta}^T*\delta_{ \overline v_\theta vx''}(c_\theta^{-1}\cdot \phi_f)
        \\ {\tiny \text{(by definition of $\Ee_{s_\theta}^T$)} }&=\frac{1}{T}\int_0^T \delta_{s_\theta^t \overline v_\theta vx''}(c_\theta^{-1}.\phi_f)dt
        \\{\tiny\text{(neglecting the subinterval $[0,T^{1/2}]$)}}& =\frac{1}{T-T^{1/2}}\int_{T^{1/2}}^T \delta_{s_\theta^t \overline v_\theta vx''}(c_\theta^{-1}.\phi_f)dt+O(\norm{\phi_f}_\infty T^{-1/2})
        \\ {\tiny \text{(multiplying and dividing by $s_\theta^t$)}} &=\frac{1}{T-T^{1/2}}\int_{T^{1/2}}^T \delta_{s_\theta^t \overline v_\theta s_\theta^{-t}s_\theta^t vx''}(c_\theta^{-1}.\phi_f)dt+O(\norm{\phi_f}_\infty T^{-1/2})
        \\ {\tiny \text{$c_\theta^{-1}\phi_f$ is $\abs{c_\theta}$-Lipschitz} }&=\frac{1}{T-T^{1/2}}\int_{T^{1/2}}^T \delta_{s_\theta^t vx''}(c_\theta^{-1}.\phi_f)+O(|c_\theta||s_\theta^t \overline v_\theta s_\theta^{-t}|)dt+O(\norm{\phi_f}_\infty T^{-1/2})
        \\{\tiny \text{($s_\theta$ contracts $\overline v_\theta$)} }&=\frac{1}{T-T^{1/2}}\int_{T^{1/2}}^T \delta_{s_\theta^t vx''}(c_\theta^{-1}.\phi_f)dt+O(R|s_\theta^{T^{1/2}} \overline v_\theta s_\theta^{-T^{1/2}}|+\norm{\phi_f}_\infty T^{-1/2})
        \\{\tiny \text{(returning the neglected subinterval $[0,T^{1/2}]$)}}&=\frac{1}{T}\int_{0}^T \delta_{s_\theta^t vx''}(c_\theta^{-1}.\phi_f)dt+O(R|s_\theta^{T^{1/2}} \overline v_\theta s_\theta^{-T^{1/2}}|+2\norm{\phi_f}_\infty T^{-1/2})
        \\{\tiny \text{(by the definition of $\Ee_{s_\theta}^T$)}}&=\Ee_{s_\theta}^T*\delta_{vx''}(c_\theta^{-1}.\phi_f)+O(R|s_\theta^{T^{1/2}}\overline v_\theta s_\theta^{-T^{1/2}}|+2\norm{\phi_f}_\infty T^{-1/2})
        \\{\tiny \text{(*)}}&=\Ee_{s_\theta}^T*\delta_{x''}(c_\theta^{-1}.\phi_f)+O(R|s_\theta^{T^{1/2}}\overline v_\theta s_\theta^{-T^{1/2}}|+\delta\abs{v}+2\norm{\phi_f}_\infty T^{-1/2})
        \\{\tiny \text{ (**)} }&=\Ee_{s_\theta}^T*\delta_{g_-x''}(\phi_f)+O(\delta\abs{R}+2(2\norm{\phi_f}_\infty T^{-1/2}+R|s_\theta^{T^{1/2}}\overline v_\theta s_\theta^{-T^{1/2}}|)).
    \end{align*}

\textbf{Explanation of $(*)$:}
To justify the equality $(*)$ we use Definition~\ref{def: qm phenomenon2}. 
Recall that $Y_R=\on{Lip}((C_\theta)^{\|}_{R}\to I)$ equipped with the supremum metric, and
define the function
\[
\xi_\theta:Y_R\to\mathbb R, \qquad \xi_\theta(h):=\phi\bigl(h(c_\theta)\bigr),
\qquad h\in Y_R.
\]
Since $\phi$ is $1$-Lipschitz and evaluation at a point is $1$-Lipschitz on
$Y_R$, the map
$\xi_\theta$ is also $1$-Lipschitz.
Moreover, for every $x\in X$ and $c_\theta\in (C_\theta)_R^{||}$,
\[
\xi_\theta(\hat f_R(x))
=
\phi\!\left(\hat f_R(x)(c_\theta)\right)
=
\phi\!\left(f(c_\theta x)\right)
=
(c_\theta^{-1}\phi_f)(x).
\]
Thus the observable $c_\theta^{-1}\phi_f$ is obtained from $\hat f_R$ by
post-composition with the $1$-Lipschitz function $\xi_\theta$.
Our next step is to use the $\hat f_R$ quantitative Mautner property to conclude the proof of passage $(*)$.
We do so by post-composing $\hat f_R$ with $\xi_\theta$ and plugging into the definition of QM, namely Definition \ref{def: qm phenomenon2}.

Since Definition~\ref{def: qm phenomenon2} allows post-composition of
$\hat f_R$ with arbitrary $1$-Lipschitz test functions, we may indeed apply it to
$\xi_\theta$. It therefore follows from the assumption that $\lambda$ is $(\kappa,\delta,T,\hat f_R)$-QM that 
\[
\abs{\Ee_{s_\theta}^{T}*\delta_{vx}(c_\theta^{-1}\phi_f)-\Ee_{s_\theta}^{T}*\delta_x(c_\theta^{-1}\phi_f)}\leq \delta \abs{v}
\]
which is precisely the passage $(*)$.

\textbf{Explanation of $(**)$:}
This step is a repeated application of all the reasoning leading to $(*)$, only with the choice $v=e$.

\textbf{Conclusion of the proof:}
    Since $\overline v_\theta\in (\overline V_\theta)_R^{||}$ and since $s_\theta$ contracts $\overline V_\theta$ we have $|s_\theta^{T^{1/2}}\overline v_\theta s_\theta^{-T^{1/2}}|\leq Re^{-T^{1/2}}$ and so the term inside $O(\cdot)$ can be replaced by $\delta\abs{v}+4\norm{\phi_f}_\infty T^{-1/2}+2R^2e^{-T^{1/2}}$.

    Subtracting the right from the left of this computation gives 
    \begin{align}
        \abs{\Ee^T(\phi_f)(g_-vx'')-\Ee^T(\phi_f)(g_-x'')}=O(\delta\abs{v}+4\norm{\phi_f}_\infty T^{-1/2}+2R^2e^{-T^{1/2}})
    \end{align}
    which proves that if a pair $(gvx,gx)\in \{(g_-x'',g_-vx''):x'',vx''\in X''\text{ and }g_-\in (C_\theta)_R^{||}(\overline V_\theta)_R^{||},v\in V_{R}^{||}\}$ then
    \begin{align}
        \abs{\Ee^T(\phi_f)(gvx'')-\Ee^T(\phi_f)(gx)}=O(\delta\abs{v}+4\norm{\phi_f}_\infty T^{-1/2}+2R^2e^{-T^{1/2}})
    \end{align}
    so that $Y(\Ee^T(\phi_f))$ is indeed what the claim asserts it to be.

\end{proof}

The following is an immediate corollary of the fact that $\Ee^T\phi_f$ is a factor function, and of our understanding of its factor domain.

\begin{corollary}\label{cor: bombach}
    In the notation of the  Lemma \ref{lem: factor function from mautner}, for $g=g_-g_+\in G_1 \subset G \setminus \Mm$ with $g_-\in G_1^-,g_+\in G_1^+$, if $x\in X''$ and  $g_+x\in X''$, then 
    \begin{equation}
        \abs{\Ee^{T}\phi_f(gx)-\Ee^{T}\phi_f(g_-x)}\leq O(R^2e^{-T^{1/2}}+4T^{-1/2}+\delta R).
    \end{equation}
\end{corollary}
\begin{proof}
    Since $R>1$, $x,g_+x\in X''$ and $g_-\in G^-_1,g_+\in G^+_1$, we have by the last part of Lemma \ref{lem: factor function from mautner}, $(g_-g_+x,g_-x)\in Y(\Ee^{T}\phi_f)$.
    Since $\Ee^{T}\phi_f$ is a factor function for $Y(\Ee^{T}\phi_f)$, and by Definition \ref{def: factor function} for $g=g_-,x=x, v=g_+$ we are allowed to deduce that
    \begin{equation}
        \abs{\Ee^{T}\phi_f(g_-g_+x)-\Ee^{T}\phi_f(g_-x)}\leq O(R^2e^{-T^{1/2}}+4T^{-1/2}+\delta R)
    \end{equation}
    as desired. 
\end{proof}

The next definition is crucial for this section, because for every $x\in X$ it represents the set of elements $g\in G_1$ for which $\Ee^T\phi_f(gx)$ is close to $\Ee^T\phi_f(x)$.
This extra invariance property is used extensively in the final proof of Theorem \ref{thm: qnz}.

To state it, recall Definition \ref{def: g1 set}, which was the following:

\begin{definition}\label{def: factor tuple again}
    Let $x \in X$. We say that:
    \begin{enumerate}
        \item for $g \in G$ and $v \in V_1$, $(x,g,v)$ is a \emph{factor triple} if $(gx,gvx)\in Y(\Ee_T( \phi_f))$;
        \item given $\overline g=(e,g_1,\dots,g_n)\in G^{\times (n+1)}$ and $\overline v=(e,v_1,\dots, v_n)\in V_1^{\times(n+1)}$, we say that $(x,\overline g,\overline v)$ is a \emph{factor tuple} if for every $0\leq k\leq n$ we have $$(g_1v_1\cdots g_kx,g_1v_1\cdots g_kv_{k+1}x)\in Y(\Ee_T(\phi_f)).$$
    \end{enumerate}
\end{definition}

Recall that factor triples $(x,g,v)$ are defined in such a way that the invariance property \eqref{eq: ff invariance} holds
$$|\Ee_T(\phi_f)(gx) - \Ee_T(\phi_f)(gvx) | = O(R^2e^{-T^{1/2}}+4T^{-1/2}+\delta R) .$$
And factor tuples are tuples such that this invariance can be compounded along the word (Lemma \ref{lem: extra invariance}).

The following crucial definition records elements $g\in G_1$ which, as words, can be spelled using factor tuples.

\begin{definition}\label{def: g1 set again}
    Suppose $V,S$ generate $G$ in $n$ steps. Write $Y:=Y(\Ee_T( \phi_f))$ and choose $x\in X$, we define:
    \begin{equation}
        G_1(Y,x)=\{g:g\in G_1\text{ and }(x,\overline s,\overline v)\text{ is a factor tuple with } g = s_1v_1 \cdots s_nv_n,s_i\in S, v_i\in V_1\}.
    \end{equation}
    We write furthermore
    $$\hat G_1(Y,x)=G_1(Y,x)\cap \{g \in G_1: g=g_-g_+,g_+x\in X'',|g_+|,|g_-|\leq 1\}.$$
\end{definition}

The importance of the set $\hat G_1(Y(\Ee^T(\phi_f)),x)$ is very clearly extra invariance of the factor function $\Ee^T\phi_f$ it encodes. We spell this out precisely in the next lemma.

\begin{corollary}[Extra invariance along $G^-$]\label{cor: factor function + fast generation}
    Suppose that for some $T>0$ the measure $\lambda$ is $(\kappa,\delta,T,\hat f_{R^{An}})$-QM. Write $Y_{R^{An}}(f):=Y(\Ee^T(\phi_f))$ and $X'' \subset X$ the two sets provided by Lemma \ref{lem: factor function from mautner}.
    
     Let $x\in X''$ and suppose that $G_+,(\stab_G(x)\cap G_{R}^{||})$ generate $G$ in $n$ steps.
    Then for every $g\in \hat G_1(Y_{R^{An}}(f),x)$:
    \begin{equation}
        \abs{\Ee^{T}\phi_f(g_-x)-\Ee^{T}\phi_f(x)}\leq (n+1)R^{8n}e^{-T^{1/2}}+4(n+1)T^{-1/2}+(n+1)\delta R^{4An}.
    \end{equation}
\end{corollary}

Here, and throughout this section $A > 0$ is a sufficiently large constant depending on $G$ alone that might change line to line. 
\begin{proof}
    Using Lemma \ref{lem: factor function from mautner} for $R'=R^{An}$ (namely we take the $R$ in the lemma to be this $R'$) and for the $T$ from the statement of the corollary, we know that the function $\Ee^{T}\phi_f$ is a $(G/G_+,R',R'^2e^{-T^{1/2}}+4T^{-1/2}+\delta R')$-factor function, and in particular since $R'=R^{An}\geq R^{2n}$ by Remark \ref{rem: factor for smaller r}, is a $(G/G_+,R^{2n},R'^2e^{-T^{1/2}}+4T^{-1/2}+\delta R')$-factor function.

    Since $G_+$ and $\stab_G(x)\cap G_R^{||}$ generate $G$ in $n$ steps we can apply Lemma \ref{lem: extra invariance} to $\Ee^T\phi_f$ and deduce that for every $g\in G_1(Y_{R^{An}}(f),x)$,
	\begin{equation}
		\abs{\Ee^{T}\phi_f(gx)-\Ee^{T}\phi_f(x)}\leq n(R'^2e^{-T^{1/2}}+4T^{-1/2}+\delta R').
	\end{equation}
    Moreover, since $x\in X''$, $g_+x\in X''$ and $g\in G_1$ (as we assume $g\in \hat G_1(Y_{R^{An}}(f),x)$), the conditions of Corollary \ref{cor: bombach} hold, and we know that:
    \begin{equation}
        \abs{\Ee^{T}\phi_f(gx)-\Ee^{T}\phi_f(g_-x)}\leq R'^2e^{-T^{1/2}}+4T^{-1/2}+\delta R'.
    \end{equation}
    Combining the above two equations and substituting $R'=R^{An}$ we get the desired inequality. 
\end{proof}

In the final steps of the proof, we will choose $T$, $R$ and $\delta$ so that $$(n+1)R^{8n}e^{-T^{1/2}}+4(n+1)T^{-1/2}+(n+1)\delta R^{An}  \lll 1.$$
So elements $g \in \hat G_1(Y_{R^{An}}(f),x)$ provide the extra invariance of $\Ee^{T}\phi_f$ at $x$ along $G^-$ we are looking for.

\subsection{Analysis of the invariance set}
In this section we analyze the sets $\hat G_1(Y_{R^{An}}(f),x)$, and show that they are rich in the sense of containing fine metric nets inside $G_1$.
    The proof of Lemma \ref{lem: lots of good points} - where we find this fine metric nets - is highly technical and involved. 
    To simplify the reading, we have divided its proof into steps, each in a different sub-subsection.

    \textbf{Standing assumptions and notation for this subsection}:
    \begin{enumerate}[label=\arabic*.]
        \item The measure $\lambda$ is $(\kappa,\delta,T,\hat f_{R^{An}})$-QM.
        \item The set $X''$ is the QM-set of $\lambda$ for the above parameters.
        \item For every $x\in X$, we denote $V_x=\{v\in V^{||}_{R^{An}}: vx\in X''\}$ and $E_x=V_{R^{An}}^{||}\setminus V_x$.
        \item We assume that $(\kappa+\eps)^{1/2}R^n\in (0,1)$.
    \end{enumerate}

This organization is postponed till after the statement of the lemma. 
\begin{lemma}\label{lem: lots of good points}
    There exists a subset $X'''\subset X''$ such that $\lambda(X''')\geq 1-R^{-n/2}-5(\kappa+\eps)^{1/2}R^n$ and for every $x\in X'''$ such that $G_+,\stab_G(x)\cap G^{||}_{R}$ generate $G$ in $n$ steps, the set $\hat G_1(Y_{R^{An}}(f),x)$ contains a $(2R^{-\frac {n}{2\on{dim}(G)}}+(\kappa+\eps)^{1/2}R^{An})$-net in $G_1$.
\end{lemma}

\begin{remark}[Idea of the proof (I)]
The purpose of Lemma \ref{lem: lots of good points} is to show that the set $\hat{G}_1(Y_{R^{An}}(f),x)$ is
large -- in the sense of containing a fine net in $G_1$ -- whenever $G_+=V_\theta$
together with $\mathrm{stab}_G(x)\cap G^{\|}_R$ generates $G$ in $n$ steps.
Combined with Corollary \ref{cor: factor function + fast generation}, this allows us to promote the approximate invariance
of the averaged function $\Ee^T(\phi_f)$, from $\widehat{G}_1(Y_{R^{An}}(f),x)$ to a
quantitatively dense subset of the whole unit ball $G_1$. The three steps of the argument will be previewed in installments using the series of remarks "Idea of the proof", parts (II)-(IV) which appear in the following. 
\end{remark}

\subsubsection{What exactly is $\hat G_1(Y_{R^{An}}(f),x)$?}\label{subsection: what is hat}

\begin{remark}[Idea of the proof (II)]
In this subsection we record an explicit criterion for membership in
$\widehat{G}_1(Y_{R^{An}}(f),x)$. The generation assumption lets us write each $g\in G_1$
as a word with    alternating elements of $\mathrm{stab}_G(x)\cap G^{\|}_R$ and of
$V=G_+$ as letters; in Lemma \ref{lem: being in g hat} we record simple conditions on the prefixes of this word that imply $g \in \widehat{G}_1(Y_{R^{An}}(f),x)$. In short, the conditions are threefold: existence of a $G_-G_+$ decomposition for each prefix, an upper bound on the size of the $G_-$ parts, and a condition on the $G_+$ parts carrying $x$ to the QM-set $X''$.

We will see in subsequent sections that these conditions are individually satisfied by many such words, and so is their conjunction. 
\end{remark}

Since the words we will consider are spelled using alternating letters of two different types, we will need to distinguish between the prefixes ending at each type of letter. The following definition does that. 

\begin{definition}\label{def: maps}
For every $g\in G_1$ for which there are $\gamma_1,\dots,\gamma_n\in \stab_G(x)\cap G_R^{||}$ such that $g\in\gamma_1V_1\gamma_2V_1\cdots \gamma_nV_1.$
  Define, for each $k=1,\dots,n$, the two prefix maps :
  \begin{equation}
      \psi_k, \rho_k:\underbrace{V_1\times \cdots \times V_1}_{n}\rightarrow G
  \end{equation}
  by, whenever $0 \leq k\leq n-1$, 
  \begin{align*}
    \psi_k(v_1,\dots,v_n)=\gamma_1v_1\cdots \gamma_kv_k\gamma_{k+1}  \quad \text{"the $k$-th $\gamma$-prefix"}
  \end{align*}
  and for $1 \leq k \leq n$
  \begin{align*}
      \rho_k(v_1,\dots,v_n)=\gamma_{1}v_1\cdots \gamma_kv_k\quad \text{"the $k$-th $v$-prefix".}
  \end{align*}
  Moreover, for every tuple $(v_1,\dots,v_n)$ for which they are defined, define $$\psi_k^+(v_1,\dots,v_n):=\psi_k(v_1,\dots,v_n)_+ \text{ and } \rho_k^+(v_1,\dots,v_n):=\rho_k(v_1,\dots,v_n)_+.$$
  Similarly, define the maps $\psi_k^-,\rho_k^-$.
\end{definition}

Now we will show that words that satisfy certain conditions on their prefixes naturally land in $\hat G_1(Y_{R^{An}}(f),x)$. Such words that are "well-spelled"  therefore provide extra invariance by Corollary \ref{cor: factor function + fast generation}. The advantage of these conditions is that they can be checked in practice.  

\begin{lemma}[Criterion on prefixes]\label{lem: being in g hat}  
    Let $x\in X$ and suppose that for some $\gamma_1,\dots,\gamma_n\in \stab_G(x)\cap G_R^{||}$ and $g\in G_1$ we can write $g=\gamma_1v_1\cdots \gamma_n v_n$ for some $v_i\in V_1$.
   
    Assume now that the prefixes of the word $g = \gamma_1 v_1 \ldots \gamma_nv_n$ satisfy the following conditions on its prefixes $\forall k \in \{1, \ldots, n-1\}$ we have :
    \begin{equation}\label{Eq: gamma-pref cond}\text{ ($\gamma$-prefix): } (\gamma_1v_1 \cdots \gamma_{k+1})_+ \cdot x \in X'' \text{ and } (\gamma_1v_1 \cdots \gamma_{k+1})_- \in (C_\theta)_{R^{An}}^{||}(\overline V_\theta)_{R^{An}}^{||};
    \end{equation}
    and $\forall k \in \{1,\ldots, n\}$:
     \begin{equation}\label{Eq: v-pref cond}\text{($v$-prefix):\qquad } (\gamma_1v_1 \cdots v_k)_+ \cdot x \in X'' \text{ and } (\gamma_1v_1 \cdots v_k)_- \in (C_\theta)_{R^{An}}^{||}(\overline V_\theta)_{R^{An}}^{||}.
    \end{equation}
    Then $g\in \hat G_1(Y_{R^{An}}(f),x)$.   
\end{lemma}

A word satisfying the conditions of Lemma \ref{lem: being in g hat} will be loosely called \emph{well spelled}.

\begin{proof}
  First, by Lemma \ref{lem: factor function from mautner} and since $Y_{R^{An}}(f)=Y(\Ee^T(\phi_f))$ by definition, we know that
    \begin{equation}\label{eq: y def}
        Y_{R^{An}}(f)=\{(g_-x'',g_-vx''):x'',vx''\in X''\text{ and }g_-\in (C_\theta)_{R^{An}}^{||}(\overline V_\theta)_{R^{An}}^{||},v\in V_{R^{An}}^{||}\}.
    \end{equation}
  Second, we assume that $\stab_G(x)\cap G_R^{||},V$ generate $G$ in $n$ steps, which means that $g$ can be written as $g=\rho_{n}(v_1,\dots,v_n)$ for some $v_1, \ldots, v_n \in V_1$.

    Recall that such that an element $g = \rho_n(v_1, \ldots, v_n) \in G_1(Y_{R^{An}}(f),x)$ if and only if for all $k$:
    \begin{align}\label{eq: andian}
        (\psi_{k}(v_1,\dots,v_n)x,\rho_{k+1}(v_1,\dots,v_n)x)\in Y_{R^{An}}(f).
    \end{align}
    By the definition of $Y_{R^{An}}(f)$ appearing in \eqref{eq: y def}, this holds if and only if we can find some $x''\in X'',v\in V^{||}_{R^{An}}$ and $q\in (C_\theta)_{R^{An}}^{||}(\overline V_\theta)_{R^{An}}^{||}$ such that $vx''\in X''$ and:
    \begin{equation}\label{Eq: restate inv set}
        (\psi_{k}(v_1,\dots,v_n)x,\rho_{k+1}(v_1,\dots,v_n)x)=(qx'',qvx'').
    \end{equation}
    Now $\rho_{k+1}(v_1, \ldots, v_n)=\psi_{k}(v_1, \ldots, v_n)v_{k+1}$ and  $v_{k+1}\in G^+$. By the uniqueness of the $G^-G^+$ decomposition when it exists (assumed in the definition of $\psi_k^\pm$), this yields $\rho_{k+1}^-=\psi_k^-\in G^-$ and
    \begin{equation}
        \rho_{k+1}^+(v_1,\dots,v_n)=\psi_{k}^+(v_1,\dots,v_n)v_{k+1}=v_{k+1}^{\psi_{k}^+(v_1,\dots,v_n)}\psi_{k}^+(v_1,\dots,v_n)
    \end{equation}
    where $g^h:= hgh^{-1}$. Note that $v_{k+1}^{\psi_k^+(v_1,\dots,v_n)}\in V = G^+$.
    We claim that if we prove
    \begin{align}\label{eq: condition}
        &\rho_{k+1}^+(v_1,\dots,v_n)x,\psi_k^+(v_1,\dots,v_n)x\in X'',\\
        &\label{eq: condition 2}\rho_{k+1}^-=\psi_k^-\in (C_\theta)_{R^{An}}^{||}(\overline V_\theta)_{R^{An}}^{||},\\
        &\label{eq: condition 3} |v_{k+1}^{\psi_k^+(v_1,\dots,v_n)}|\leq R^{An}
    \end{align}
    for every $k=1,\dots,n-1$, Equation \eqref{Eq: restate inv set} will show that $g\in G_1(Y_{R^{An}}(f),x)$. Indeed, for each $k$ one can then assign $q := \psi_k^-(v_1, \ldots, v_n) = \rho_k^-(v_1, \ldots, v_n)$, $v := v_{k+1}^{\psi_k^+(v_1,\dots,v_n)}$ and $x'' := \psi_k^+(v_1,\dots,v_n)x$ in \eqref{Eq: restate inv set}.
    
    Let us start by noticing that \eqref{eq: condition} and \eqref{eq: condition 2} are explicitly the assumptions of Lemma \ref{lem: being in g hat}. Equation \eqref{eq: condition 3} follows from the following  facts: since $k\leq n$, $|\psi_k(v_1, \ldots, v_n)| \leq R^n$, and \eqref{eq: condition 2} gives $|\psi_k^-(v_1, \ldots, v_n)| \leq R^{An}$. So submultiplicativity gives $|\psi_k^+(v_1,\dots,v_n)|\leq R^{An}$. Since $v_{k+1} \in V_1$, submultiplicativity again gives $|v_{k+1}^{\psi_k^+(v_1,\dots,v_n)}| \leq R^{An}$. This concludes. 
    
    Hence, $g \in G_1(Y,x)$. We further want to have $g\in \hat G_1(Y,x)$. Namely, we want to prove that $g_+x\in X''$  (Definition \ref{def: g1 set again}).
    But $g_+ = \rho_n(v_1, \ldots, v_n)$, so this is simply \eqref{Eq: v-pref cond} at $k = n$.

    In conclusion, $g\in \hat G_1(Y_{R^{An}}(f),x)$ as desired. 
\end{proof}

\subsubsection{A local $P$-step into $\hat G_1(Y_{R^{An}}(f),x)$}\label{sub: local p step}

\begin{remark}[Idea of the proof (III)]
The heart of the subsection is the perturbation argument of Claim \ref{cl: p1}. For a fixed word $g=\rho_n(v_1,\dots,v_n)$  we adjust the $V$-letters one at a time to make sure that the emerging word is well spelled i.e. satisfies the conditions of Lemma \ref{lem: being in g hat}:
a differential estimate \eqref{eq: derivative pushforward}, together with the smallness of $E_{x}:=\{v \in V_{R^{n}} : vx \notin X''\}$, lets us
replace each $v_j$ by a nearby $v'_j$ whose plus-part avoids $E_{x}$, while the
stability of $(D,c)$-tuples (Claims \ref{cl: conjugating to a good tuple not appendix} and \ref{cl: good tuples are stable not appendix}) keeps the minus-parts
controlled. This yields a well spelled word satisfying \eqref{Eq: gamma-pref cond} and \eqref{Eq: v-pref cond}, and hence, an
element $g'\in\widehat{G}_1(Y_{R^{An}},x)$ which remains close to $g$.

Unfortunately, this strategy might not work for all base points $x$ and all words $g$. To remedy that, we further perturb the basepoint $x$. Namely, we show that for most $p \in P_1$ the strategy works for the basepoint $px$ and the conjugated word $pgp^{-1}$ (with the letters also conjugated by $p$). 
\end{remark}

The goal of this part is to prove the following claim.
As explained in the above remark, the goal of the claim is to wiggle words into a well spelled words and apply Lemma \ref{lem: being in g hat}. More precisely:

 \begin{claim}\label{cl: p1}
        There exists a subset $X'''\subset X''$ such that $\lambda(X''')\geq 1-(\kappa+\eps)^{1/2}R^{n/2}$, if $\eps$ and $\kappa$ are small enough, for every $x\in X'''$ and every $g\in G_1$ given by $g=\rho_n(v_1,\dots,v_n)=\gamma_1v_1\cdots \gamma_nv_n$ where $\gamma_i\in \stab_G(x)\cap G_{R}^{||}$ and $v_i\in V_1$ there is a subset $P_1(x,g)\subset P_1$ such that
        \begin{enumerate}
            \item $m_{P_1}(P_{1}(x,g))\geq 1- cR^n$;
            \item For every $p\in P_1(x,g)$ such that $px\in X'''$ there exists $g'_p\in \hat G_1(Y_{c^{-1}R^{2n}}(f),px)$ such that $d(pgp^{-1},g_p')\leq c^{-1}(\kappa+\eps)^{1/2}R^{2n}$.
        \end{enumerate}
    \end{claim}

Before we begin the proof and since we are going to perturb the basepoint, we need the following claim which is used to define a set $X'''$ whose points return often to the QM-set $X''$ under the action of $P$.
    \begin{claim}\label{cl: good orbits}
        We can find $X'''\subset X''$ such that $\lambda(X''')\geq 1-2(\kappa+\eps)^{1/2}R^{n/2}$
        and for every $x\in X'''$
        \begin{equation}\label{eq: bound on bad set}
            m_{V_{R^n}^{||}}(\{v \in V_{R^n}^{||} : vx \notin X''\})\leq (\kappa+\eps)^{1/2}R^{n}.
        \end{equation}    
    \end{claim}
    \begin{proof}
        We know that $\lambda(X'')\geq 1-\kappa$ and that since $\lambda$ is $\eps$-norm-$P$-invariant (and $V\subset P$), we have for every $v\in V_{R^n}^{||}$ that $\norm{v.\lambda-\lambda}_{TV}\leq \eps R^n$.
        This implies, by averaging the resulting inequalities for every $v\in V_{R^n}^{||}$, that $\norm{m_{V_{R^n}^{||}}*\lambda-\lambda}_{TV}\leq\eps R^n$.
        In particular, computing the above inequality on the set $X''$ and using $\lambda(X'')\geq 1-\kappa$ gives $$m_{V_{R^n}^{||}}*\lambda(X'')\geq 1-\kappa-\eps R^n.$$
        By the definition of convolution and Fubini, it means that
        \begin{align*}
            \int m_{V_{R^n}^{||}}(v:vx\in X'')d\lambda(x)  =m_{V_{R^n}^{||}}\times \lambda((v,x)\in V_{R^n}^{||}\times X:vx\in X'')\\
            \geq 1-\kappa-\eps R^n.
        \end{align*}
        By the Markov inequality this implies that
        \begin{align}
            \lambda\left(x:m_{V_{R^n}^{||}}(vx\in X'')\geq 1-(\kappa+\eps R^n)^{1/2}\right)\geq 1-(\kappa+\eps R^n)^{1/2}
        \end{align}
        and we define $X'''$ to be this set, namely
        \begin{align*}
            X'''=\{x:m_{V_{R^n}^{||}}(vx\in X'')\geq 1-(\kappa+\eps R^n)^{1/2}\}.
        \end{align*}
        Since $\kappa+\eps R^n\in (0,1)$, 
        \begin{align*}
            \lambda(X''')\geq 1-(\kappa+\eps R^n)^{1/2}\geq 1-(\kappa+\eps)^{1/2}R^n
        \end{align*}
        and similarly for every $x\in X'''$
        \begin{align}\label{eq: orbit intersection}
            m_{V_{R^n}^{||}}(vx\in X'')\geq 1-(\kappa+\eps )^{1/2}R^n)\geq 1-(\kappa+\eps)^{1/2}R^n.
        \end{align}
        Intersecting $X'''$ with $X''$, which has measure at least $1-\kappa$, we may assume that $X'''\subset X''$ and in addition to Equation \eqref{eq: orbit intersection} we also have $\lambda(X''')\geq 1-(\kappa+\eps)^{1/2}R^n-\kappa \geq 1-2(\kappa+\eps)^{1/2}R^n$ as desired. 
    \end{proof}

    Our next step is invoking, from Appendix \ref{A matrix decomposition}, the notion of $(D,c)$-tuples and relevant results. They are an important technical tool for the control over the $G_-G_+$ decomposition. 
    
    Indeed, they provide control over how the $G_+$ part of successive prefixes behaves under small perturbation. This control is precisely what will be needed to successively "wiggle" the letters coming from $V$. 

    Let us first make precise what kind of control we are looking for at each step.

\begin{nota}
Let $\vV$ denote the Lie algebra of the Lie group $V$.
For every $g_1, g_2\in G$ and $v_0\in V_1$, denote by $D_{g_1, g_2}(v_0):\vV\to \vV$ the differential at $v_0$ of the map 
$$v\mapsto (g_1g_2)_+^{-1}(g_1vg_2)_+.$$

Moreover we define, for every $c>0$ the sets

\begin{align*}
    \Dd_c^+= \{(g_1,g_2)\in G^2: \|D_{g_1, g_2}^{-1}(v_0)\|_{\rm op}^{-1} \geq c\text{ for all $v_0\in V_1$},|(g_1)_+|\leq c^{-1}|g_1|\}.
\end{align*}
and 
\begin{align*}
    \Pp^+_c(g_1,g_2)&= \{p \in P_1 : (pg_1g^{-1},pg_2p^{-1}) \in \Dd_c^+\}\\&=\{p\in P_1: 
    \forall v\in V_1, \|D_{pg_1p^{-1}, pg_2p^{-1}}^{-1}(v)\|_{\rm op}^{-1} \geq c \\& \qquad \qquad \qquad \qquad \text{ and }|(pg_1p^{-1})_+|\leq c^{-1}|pg_1p^{-1}|\}
\end{align*}

\end{nota}

Now the $(D,c)$-tuples are those tuples such that successive products obtained from reading left to write satisfy, at each step, the control condition we just introduced.

\begin{definition}\label{def: dc paper}
    Given $\overline g=(g_1,\dots,g_n)\in G^n$ we say that $\overline g$ is a $D,c$-tuple if for every $i=0,\dots,n-1$ we have $((g_1g_2\cdots g_i),g_{i+1})\in \Dd_c^+$.
    The set of $(D,c)$-tuples is denoted $(D,c)^n$.
\end{definition}

In the appendix, we prove two results about the size of $(D,c)$-tuples. 

We first prove that most tuples are $(D,c)$-tuples, up to conjugating by most elements of $P$; this explains the need to move the base point from $x$ to $px$. 

\begin{claim}[Claim \ref{cl: conjugating to a good tuple}]\label{cl: conjugating to a good tuple not appendix}
    Given any $g_1,\dots,g_n\in G_R^{||}$ and $c>0$, we can find a subset $\Pp^+$ of $P_1$ measure at least $1-R^{q_Gn}c^{q_G'}$ such that for every $p\in \Pp^+$ we have that $p\overline gp^{-1}$ is a $(D,c)$-tuple where $\overline g=(g_1,\dots,g_n)$.
\end{claim}

Second, we prove a stability result for $(D,c)$-tuples: neighboring tuples remain $(D,c/2)$-tuples in a neighbourhood that is quantitatively not too small.

\begin{claim}[Claim \ref{cl: good tuples are stable}]\label{cl: good tuples are stable not appendix}
    For every $c>0$ and for every $(g_1,g_2)\in \Dd_c^+\cap (G_R^{||})^2$ there exists $q_G>0$ dependent only on $G$ such that for all $g\in G_{c/R^{q_{G}}}$ we have $(gg_1,g_2)\in \Dd_{c/2}^+$.
    As a result, the set of $D,c$-tuples is stable in the following sense
    \begin{align*}
        \left(G_{c/R^{nq_{G}}}\times \cdots \times G_{c/R^{nq_{G}}} \cdot (D,c)^n\right)\cap (G_R^{||})^n\subset (D,c/2)^{n}.
    \end{align*}
\end{claim}

Now we are ready to prove Claim \ref{cl: p1}.

\begin{proof}[Proof of Claim \ref{cl: p1}]
The proof has two main steps.

    \textbf{Step 1: turning the tuple $(p\gamma_1p^{-1},pv_1p^{-1},\dots, p\gamma_np^{-1},pv_np^{-1})$ into a $(D,c)$-tuple:}
    Applying Claim \ref{cl: conjugating to a good tuple not appendix} to the tuple $(\gamma_1,v_1,\dots,\gamma_n,v_n)$, we obtain that for every $p$ in the subset $\Pp_+$ defined in the claim, which has measure at least $1-R^{q_Gn}c^{q_G'}$, the tuple $(p\gamma_1p^{-1},pv_1p^{-1},\dots, p\gamma_np^{-1},pv_np^{-1})$ is a $(D,c)$-tuple.
    For simplicity of notation, we replace $R$ and $c$ by $R^{q_G},c^{q_G'}$ respectively.
    Indeed, the conditions of the claim hold for $g_1=\gamma_1,g_2=v_1,\dots,g_{2n-1}=\gamma_{n},g_{2n}=v_n$ which all belong to $G_R^{||}$ (the $v_i$'s even belong to $G_1$).
    We thus define $P_1(g,x)=\Pp_+$. Note that $\Pp_+$ indeed varies with $x$ because $\gamma_1, \ldots, \gamma_n \in \stab_G(x)$ and with $g$ because $g=\rho_n(v_1,\dots,v_n)$.

    \textbf{Step 2: wiggling the $v_i$s so that $\rho$ lands into $\hat G_1(Y_{R^{An}}(f),px)$:}
In the following sub-step, we construct $v_1'(p)$, a small wiggle of $v_1(p)$, for which the prefix conditions of Lemma \ref{lem: being in g hat} are satisfied for $k=1$.
This will give us the induction base and the strategy for the induction step, used to verify these prefix conditions for every $k=1,\dots,n$, as is done in the next step. 
    
\textbf{Step 2.1: basis for the inductive definition of the nearby element:}
By the previous step, for every $p\in P_{1}(g,x)$ the tuple  
    \begin{align*}
        (\gamma_1(p),v_1(p),\dots,\gamma_n(p),v_n(p)):=(p\gamma_1p^{-1},pv_1p^{-1},\dots, p\gamma_np^{-1},pv_np^{-1})
    \end{align*}
    is a $(D,c)$-tuple.

Define the map $\zeta_1$ by
    \begin{align*}
        V_1\ni v\mapsto(\gamma_1(p) v_1(p)v\gamma_2(p))_+\in V^{||}_{R^n}
    \end{align*}
    and the map $\eta_1$ by
    \begin{align*}
        V_1\ni v\mapsto(\gamma_1(p) v_1(p)v)_+\in V^{||}_{R^n}.
    \end{align*}
    Since by the definition of $(D,c)$-tuples we know that $\norm{D^{-1}_{\gamma_1(p)v_1(p),\gamma_2(p)}(\cdot)}^{-1} \geq c$ we can obtain 
    \begin{align}\label{eq: derivative pushforward}
        (\zeta_1)_*m_{V_1}\leq \norm{D^{-1}_{\gamma_1(p)v_1(p),\gamma_2(p)}}m_{V_{R^n}^{||}}\leq c^{-1}m_{V_{R^n}^{||}}
    \end{align}
    and for the map $\eta_1$, we simply obtain by conjugation with $\gamma_+v_1(p)$ that for some $c_G$ depending only on $G$:
    \begin{align}\label{eq: derivative pushforward 2}
        (\eta_1)_*m_{V_1}\leq c_Gm_{V_{R^n}^{||}}.
    \end{align}
    By the stability of $(D,c)$-tuples from Claim \ref{cl: good tuples are stable not appendix}, for every $v\in V_{c/R^{nq_G}}$ the tuple $(\gamma_1(p),vv_1(p),\gamma_2(p),\dots,\gamma_n(p),v_n(p)$ is a $(D,c/2)$-tuple.
    This means in particular that:
    \begin{align}
        \label{eq: size bounds 2}|(\gamma_1(p)vv_1(p)\gamma_2(p))_+|\leq c^{-1}|\gamma_1(p)vv_1(p)\gamma_2(p)|\qquad \text{for every }v\in V_{c/R^{nq_G}}
    \end{align}
    and 
    \begin{align}
        \label{eq: size bounds 3}|(\gamma_1(p)vv_1(p))_+|\leq c^{-1}|\gamma_1(p)vv_1(p)|\qquad \text{for every }v\in V_{c/R^{nq_G}}
    \end{align}
    As we assume (in part $(b)$ of the claim we are proving) that $px\in X'''$, by Claim \ref{cl: good orbits} we know that $m_{V_{R^n}^{||}}(\{v \in V_{R^n}^{||} : vpx \notin X''\})\leq (\kappa+\eps)^{1/2}R^{n}$.
    Write $E_{px}:=\{v \in V_{R^n}^{||} : vpx \notin X''\}$. Thus, by Equation \eqref{eq: derivative pushforward} we obtain 
    \begin{align}
        m_{V_1}(\zeta_1^{-1}(E_{px}))&\leq c^{-1}m_{V_{R^n}^{||}}(E_{px})\\ &\leq  c^{-1}(\kappa+\eps)^{1/2}R^{n}:=\delta''.
    \end{align}
    Similarly, by Equation \eqref{eq: derivative pushforward 2} we obtain
    \begin{align}
        m_{V_1}(\eta_1^{-1}(E_{px}))\leq c_Gm_{V_{R^n}^{||}}(E_{px})\leq \delta''.
    \end{align}

    Define the set $V_1(1)=\left(\eta_1^{-1}(E_{px})\cup \zeta_1^{-1}(E_{px})\right)^c$.
    We therefore deduce that $V_1(1)^c$ has measure at most $2\delta''$ and therefore has a $\delta''^{c_G}$-dense complement in $V_1$ for some positive constant $c_G$ depending on $G$ alone.

    Note that membership in $V_1(1)$ guarantees that 
    \begin{align}
        \eta_1(v)px=(\gamma_1(p)v_1(p)v)_+px,\zeta_1(v)px=(\gamma_1(p)v_1(p)v\gamma_2(p))_+px\in X''.
    \end{align}
    Put differently, for every $v_1'\in V_1(1)$ we have 
    \begin{align}\label{eq: prefixes in x''}
        \psi_1(v_1',v_2,\dots,v_n)_+px=(\gamma_1(p)v_1'\gamma_2(p))_+ px\in X''\\
        \rho_1(v_1',v_2\dots,v_n)_+px=(\gamma_1v_1')_+px=vpx\in X''.
    \end{align}
    If $\eps,\kappa$ are small enough (which we assume) we automatically get $\delta''^{c_G}\leq c/R^{nq_{G}}$ so we can find (by the $\delta''^{c_G}$-density of $V_1(1)$) some $v_1'(p)\in V_1(1)$ such that $d(v_1(p),v_1'(p))\leq \delta''^{c_G}\leq c/R^{nq_{G}}$.

    Therefore by Equation \eqref{eq: size bounds 2}:
    \begin{align}\label{eq: minus gamma prefix estimate}
        |(\gamma_1(p)v_1'(p)\gamma_2(p))_-|\leq |\gamma_1(p)v_1'(p)\gamma_2(p))||\gamma_1(p)v_1'(p)\gamma_2(p))_+^{-1}|
        \\
        \leq c^{-1}|\gamma_1(p)v_1'(p)\gamma_2(p)|^2\leq c^{-1}R^{2\cdot 3}\leq c^{-1}R^{2n}
    \end{align}
    and by Equation \eqref{eq: size bounds 3}:
    \begin{align}\label{eq: minus v prefix estimate}
        |(\gamma_1(p)v_1'(p))_-|\leq |\gamma_1(p)v_1'(p)||(\gamma_1(p)v_1'(p))_+^{-1}|\leq R^2c^{-1}R^2\leq c^{-1}R^{2n}.
    \end{align}
    Equations \eqref{eq: minus gamma prefix estimate} and \eqref{eq: minus v prefix estimate} together with \eqref{eq: prefixes in x''} are precisely the criteria on prefixes required in Lemma \ref{lem: being in g hat} for $k=1$.
    
    \textbf{Step 2.2: the inductive step}
    We will construct and prove by induction on $j=1,\dots,n$, a sequence of elements $v_j'(p)\in V_1$ such that for every $j=1,\dots,n$:
    \begin{enumerate}[label=\arabic*.]
        \item $d(v_j'(p),v_j(p))\leq \delta''$ (proximity);
        \item $\psi_{j+1}(v_1'(p),\dots,v_j'(p),v_{j+1}',\dots, v_n)_+px, \rho_{j+1}(v_1'(p),\dots,v_j'(p),v_{j+1}',\dots, v_n)_+px\in X''$ (plus parts of prefix land in the good set $X''$);
        \item $\psi_{j+1}(v_1'(p),\dots,v_j'(p),v_{j+1}',\dots, v_n)_-\leq c^{-1}R^{2n}, \quad \rho_{j+1}(v_1'(p),\dots,v_j'(p),v_{j+1}',\dots, v_n)_-\leq c^{-1}R^{2n}$ (control over the minus parts);
        \item The tuple $(\gamma_1(p),v_1'(p),\dots,\gamma_j(p),v_j'(p),\gamma_{j+1}(p),v_{j+1}(p),\dots,v_n(p))$ is a $(D,c/2)$-tuple (remaining a $(D,c/2)$-tuple). 
    \end{enumerate}

    Recall that by the previous step, $d(v_1'(p),v_1(p))\leq \delta''$, the tuple $$(\gamma_1(p),v_1'(p),\gamma_2(p),\dots,\gamma_n(p),v_n(p))$$ is a $(D,c/2)$ tuple, and conditions 2,3 of the induction hold for $j=1$ by Equations \eqref{eq: prefixes in x''}, \eqref{eq: minus gamma prefix estimate} and \eqref{eq: minus v prefix estimate} respectively.
    Therefore, the induction base holds. 
    
    Inductively repeating the argument we carried above for the base case $j=1$, we define the map $\zeta_j$ by
    \begin{align*}
        V_1\ni v\mapsto (\gamma_1(p)v_1'(p)\cdots \gamma_{j-1}(p)v_{j-1}'(p)\gamma_j(p)v\gamma_{j+1}(p))_+
    \end{align*}
    and the map $\eta_j$ by
    \begin{align*}
        V_1\ni v\mapsto (\gamma_1(p)v_1'(p)\cdots \gamma_{j-1}(p)v_{j-1}'(p)\gamma_j(p)v)_+.
    \end{align*}
    We obtain the set $V_1(j+1)=\left(\eta_{j}^{-1}(E_{px})\cup \zeta_j^{-1}(E_{px})\right)^c$, whose complement in $V_1$ has measure at most $2\delta''$ and for every $v_{j+1}'\in V_1(j+1)$ we have (analogously to Equation \eqref{eq: prefixes in x''})
    \begin{align}\label{eq: prefixes in x'' step j}
        \psi_{j+1}(v_1'(p),\dots,v_j'(p),v_{j+1}',\dots, v_n)_+px\in X'',\\
        \rho_{j+1}(v_1'(p),\dots,v_j'(p),v_{j+1}',\dots, v_n)_+px\in X''
    \end{align}
    verifying condition 2 of the induction hypothesis, and (analogously to Equations \eqref{eq: minus gamma prefix estimate} and \eqref{eq: minus v prefix estimate}):
    \begin{align}\label{eq: minus gamma prefix estimate step j}
        \psi_{j+1}(v_1'(p),\dots,v_j'(p),v_{j+1}',\dots, v_n)_-\leq c^{-1}R^{2n}, \quad \rho_{j+1}(v_1'(p),\dots,v_j'(p),v_{j+1}',\dots, v_n)_-\leq c^{-1}R^{2n}
    \end{align}
    verifying condition 3 of the induction hypothesis. 
    Defining $v_j'(p)=vv_j(p)$ for some $v\in V_{\delta''}$ for which $v_j'(p)\in V_1(j+1)$ we obtain that still, $$(\gamma_1(p),v_1'(p),\dots,\gamma_j(p),v_j'(p),\dots, \gamma_n(p),v_n(p))$$ is a $(D,c/2)$-tuple by Claim \ref{cl: good tuples are stable not appendix}, verifying condition 4 of the induction hypothesis (notice that it is indeed still $c/2$ and not $c/4$ as this is what the claim asserts).
    Now, the conditions of Lemma \ref{lem: being in g hat} are satisfied for $j+1$ and the resulting new tuple $(\gamma_1(p),v_1'(p),\dots,\gamma_j(p),v_{j+1}'(p),\gamma_{j+1}(p),v_{j+1}(p),\dots,v_n(p))$ is still a $(D,c/2)$ tuple, and also $d(v_j'(p),v_j(p))\leq \delta''$, verifying condition 1 of the induction, so we proved the induction step.
    \textbf{Step 2.3: verifying proximity}
    
     Recall that we denoted
    \begin{align*}
    (\gamma_1(p),v_1(p),\dots,\gamma_n(p),v_n(p)):=(p\gamma_1p^{-1},pv_1p^{-1},\dots, p\gamma_np^{-1},pv_np^{-1})
    \end{align*}
    and now define 
    \begin{align*}
        g_p':=\gamma_1(p)v_1'(p)\cdots \gamma_n(p)v_n'(p)\text{ and }g_p=pgp^{-1}=\gamma_1(p)v_1(p)\cdots \gamma_n(p)v_n(p).
    \end{align*}
    We proved in the previous step that the tuple $(\gamma_1(p),v_1'(p),\dots,\gamma_n(p),v_n'(p))$ satisfies the conditions of Lemma \ref{lem: being in g hat} with respect to the word $g'_p = \gamma_1(p)v_1'(p)\cdots \gamma_n(p)v_n'(p)$ and the base point $px$.
    
    By Lemma \ref{lem: being in g hat} we thus obtain that 
    \begin{align*}
        g_p'=\gamma_1(p)v_1'(p)\cdots \gamma_n(p)v_n'(p)\in \hat G_1(Y_{R^{An}}(f),px).
    \end{align*}

    Recall that for every $i=1,\dots,n$ the constructed $v_i'(p)$ are such that $d(v_i'(p),v_i(p))\leq \delta''$, and since $\gamma_i(p)\in G_R^{||}$ by assumption, we further obtain
    \begin{align}
        d(g_p',g_p)=d(\gamma_1(p)v_1'(p)\cdots \gamma_n(p)v_n'(p),\gamma_1(p)v_1(p)\cdots \gamma_n(p)v_n(p))\leq R^n\delta''\\
        =c^{-1}(\kappa+\eps)^{1/2}R^{2n}
    \end{align}
    as desired. 
\end{proof}

\subsubsection{Global step into $\hat G_1(Y_{R^{An}}(f),x)$ and conclusion of the proof}\label{sub: global}

\begin{remark}
    This subsection globalizes this local step. For $K>0$, applying Claim~\ref{cl: p1} to the points of a
$K^{-\dim(G)^{-1}}$-net $\{g_k\}_{k=1}^{K}\subset G_1$, Claim~\ref{cl: really almost done} produces, for most
$p$, a net contained in $\widehat{G}_1(Y_{R^{An}}(f),px)$. The proof of Lemma~\ref{lem: lots of good points} shows that such basepoints of the form $px$ form a large subset of
$X$, exploiting in particular the $\epsilon$-$P$-almost invariance of $\lambda$. Optimizing the parameters gives the final net radius and the stated lower bound on the
measure of the good set.
\end{remark}
The first claim we need applies Claim \ref{cl: p1} several times to construct a net in $\hat G_1(Y_{R^{An}}(f),x)$.
\begin{claim}\label{cl: really almost done}
    Let $c>0$ and let $K>0$ be natural and suppose $\eps$ and $\kappa$ are small enough.
    For every $x\in X'''$ such that $\stab_G(x)\cap G_R^{||},V$ generate $G$ in $n$ steps, there exists a subset $P_K(x)\subset P_1$ such that
    \begin{enumerate}
        \item $m_{P_1}(P_K(x))\geq 1-KcR^n$;
        \item For every $p\in P_K(x)$ such that $px\in X'''$, the set $\hat G_1(Y_{c^{-1}R^{2n}}(f),px)$ contains a $2K^{-\dim(G)^{-1}}+c^{-1}(\kappa+\eps)^{1/2}R^{2n}$-net in $G_1$.
    \end{enumerate}
\end{claim}
\begin{proof}
    Pick $\{g_i\}_{i=1}^K\subset G_1$ a $K^{-\on{dim}(G)^{-1}}$-net in $G_1$. By fast generation at $x$, we can apply Claim \ref{cl: p1} to $x$ and $g_i$ and obtain sets $P_{1}(x,g_i)$.  Define $P_K(x)=\bigcap_{i=1}^KP_{1}(x,g_i)$.
    By the estimate in part $(a)$ of Claim \ref{cl: p1} applied to every $i=1,\dots,K$ we obtain $m_{P_1}(P_K(x))\geq 1-KcR^n$.

    Again by Claim \ref{cl: p1}, for every $p\in P_K(x)$ such that $px\in X'''$ and for every $i=1,\dots,K$, we can find $g_{p,i}'\in \hat G_1(Y_{c^{-1}R^{2n}}(f),px)$ such that $d(g_{p,i}',pg_ip^{-1})\leq c^{-1}(\kappa+\eps)^{1/2}R^{2n}$.
    Since $\{g_i\}$ is a $K^{-\on{dim}(G)^{-1}}$ in $G_1$, we know that $\{pg_ip^{-1}\}$ contains a $2K^{-\on{dim}(G)^{-1}}$ in $G_1$.
    Since $d(g_{p,i}',pg_ip^{-1})\leq c^{-1}(\kappa+\eps)^{1/2}R^{2n}$ we obtain a $2K^{-\on{dim}(G)^{-1}}+c^{-1}(\kappa+\eps)^{1/2}R^{2n}$-net in $G_1$ consisting of elements from $\hat G_1(Y_{R^{An}}(f),px)$ as desired. 
\end{proof}
Now we have everything we need to conclude the proof of Lemma \ref{lem: lots of good points}.
\begin{proof}[Proof of Lemma \ref{lem: lots of good points}]
    Recall that $\lambda$ is $\eps-P$-norm almost invariant and that $m_{P_1}(P_K(x))\geq 1-KcR^n,\lambda(X''')\geq 1-(\kappa+\eps)^{1/2}R^{n}$.
    Therefore by Fubini we can estimate
    \begin{align}\label{eq: new triple prime}
        \lambda\left(x\in X''':m_{P_1}(P_K(x)x\cap X''')\geq 1-(\kappa+\eps)^{1/2}R^n-KcR^n\right)\geq 1-3(\kappa+\eps)^{1/2}R^n.
    \end{align}
    We therefore redefine $X'''$ to be the elements in the old $X'''$ satisfying $m_{P_1}(P_1(x)\cap X''')\geq 1-(\kappa+\eps)^{1/2}R^n-KcR^n$, and will use it henceforth together with the newly established estimate
    \begin{align}\label{eq: new triple prime'}
        \lambda(X''')\geq 1-3(\kappa+\eps)^{1/2}R^n
    \end{align}
    In addition, for every $x\in X'''$ we re-define $P_K(x)$ to be $\{p\in P_K(x):px\in  X''' \}$ so that by the above estimate, we also have $m_{P_1}(P_k(x))\geq 1-(\kappa+\eps)^{1/2}R^n-KcR^n$ for every $x\in X'''$.

    Define the set 
    \begin{align*}
        \Dd=\{x\in X:\stab_G(x)\cap G_R^{||},V\text{ generate $G$ in $n$ steps and }\\
        \hat G_1(Y,x)\text{ contains a $2K^{-\dim(G)^{-1}}+c^{-1}(\kappa+\eps)^{1/2}R^{2n}$-net in $G_1$}\}.
    \end{align*}
    By the $\eps$-almost invariance of $\lambda$, the definition of convolution and Equation \eqref{eq: new triple prime'} we can estimate
    \begin{align*}
        \lambda(\Dd)\geq m_{P_1}*\lambda (\Dd)-\eps=\int_{X}m_{P_1}(p:px\in \Dd)d\lambda(x)-\eps\\
        \geq \int_{X'''}m_{P_1}(p:px\in \Dd)d\lambda(x)-\eps-3(\kappa+\eps)^{1/2}R^n.
    \end{align*}
    We proved in Claim \ref{cl: really almost done} that if $x\in X'''$ and $\stab_G(x)\cap G_R^{||},V$ generate $G$ in $n$ steps, $p\in P_K(x)$ and $px\in X'''$ then $\hat G_1(Y,px)$ contains a $2K^{-\dim(G)^{-1}}+c^{-1}(\kappa+\eps)^{1/2}R^{2n}$-net in $G_1$ and so $px\in \Dd$.
    Moreover, since $x\in X'''$, by our new definition of $X'''$ and of $P_K(x)$ we know that $px\in X'''$ for every $p\in P_K(x)$.
    Therefore we can continue the estimation above as follows
    \begin{align*}
        \lambda(\Dd)\geq \int_{X'''\cap \{x\in X:\stab_G(x)\cap G_R^{||},V\text{ generate $G$ in $n$ steps}\}}m_{P_1}(P_K(x))d\lambda(x)-\eps-3(\kappa+\eps)^{1/2}R^n\\
        \geq \lambda(\{x\in X:\stab_G(x)\cap G_R^{||},V\text{ generate $G$ in $n$ steps}\})-KcR^n-\eps-4(\kappa+\eps)^{1/2}R^n.
    \end{align*}
    Before we conclude the proof, we choose the parameters $K,c,\kappa'$ appropriately.
    We choose $c=R^{-2n}$ and $K=R^{n/2}$ so that
    \begin{align}
        2K^{-\dim(G)^{-1}}+c^{-1}(\kappa+\eps)^{1/2}R^{2n}\leq 2R^{-\frac{n}{2}\on{dim}(G)^{-1}}+(\kappa+\eps)^{1/2}R^{An}
    \end{align}
    and 
    \begin{align}
        KcR^n+\eps+4(\kappa+\eps)^{1/2}R^n\leq R^{-n/2}+5(\kappa+\eps)^{1/2}R^n.
    \end{align}
    Therefore the subset $\Dd$ has measure at least $$\lambda(\{x\in X:\stab_G(x)\cap G_R^{||},V\text{ generate $G$ in $n$ steps}\})-R^{-n/2}-5(\kappa+\eps)^{1/2}R^n$$ and 
    for every $x$ in it, the set $\hat G_1(Y,x)$ contains a $2R^{-\frac n2\on{dim}(G)^{-1}}+(\kappa+\eps)^{1/2}R^{An}$-net as desired.

    The set $\mathcal{D} \cup \{x\in X:\stab_G(x)\cap G_R^{||},V\text{ generate $G$ in $n$ steps}\}^c$ is the set $X'''$ from Lemma \ref{lem: lots of good points} we were looking for.
\end{proof}

The following proposition is a direct result of Corollary \ref{cor: factor function + fast generation}, Lemma \ref{lem: lots of good points} and the fact that nets in $G_1$ project to nets in $G_1^-$ with a scalar multiple of the radius. 
\begin{proposition}\label{prop: final prop}
    Suppose $X$ is a $G$-space and that $\lambda$ is a norm $\eps$-$P$-invariant measure on $X$ which is $(\kappa,\delta,T,\hat f_{R^{An}})$-QM for some $R>0$ large.
    Suppose in addition that $x\in X'''$ (defined in Lemma \ref{lem: lots of good points}, recall that $\lambda(X''')\geq 1-R^{-n/2}-5(\kappa+\eps)^{1/2}R^n$) and $V_\theta,\stab_G(x)\cap G^{||}_{R}$ generate $G$ in $n$ steps.
    Then for every $g_-\in G_1^-$ there is some $g'_-\in G^-_1$ such that $d(g_-,g'_-)\leq 2R^{-\frac n2\on{dim}(G)^{-1}}+(\kappa+\eps)^{1/2}R^{An}$ and:
    \begin{equation}
        \abs{\Ee^{T}\phi_f(g'_-x)-\Ee^{T}\phi_f(x)}\leq 8n^3R^{8n}e^{-T^{1/2}}+8nT^{-1/2}+2n\delta R^{An}
    \end{equation}    
    In particular, if $\phi$ is $1$-Lipschitz, the function $\Ee^T\phi_f$ is $1$-Lipschitz for the orbit map of $G^-$ and therefore we also have for all $g_- \in G^-_1$ 
    \begin{equation}
        \abs{\Ee^{T}\phi_f(g_-x)-\Ee^{T}\phi_f(x)}\leq  2R^{-\frac n2\on{dim}(G)^{-1}}+(\kappa+\eps)^{1/2}R^{An}+8n^3R^{8n}e^{-T^{1/2}}+8nT^{-1/2}+2n\delta R^{An}.
    \end{equation}    
\end{proposition}

\newpage

\section{Almost invariant measures and almost projective factors}\label{sec: final proof}
In this section we prove our main theorem \ref{thm: qnz}.
The proof begins with the following lemma, helping us to pass (in terms of preserving integrals) from a test function to what will later be a factor function.
\begin{lemma}\label{lem: final}
    Suppose $X$ is a $G$-space and that $\lambda$ is a norm $\eps$-$P$-invariant measure on $X$ and let $\overline u_\theta\in \overline U_\theta$.
    Then for every $T>0$ large, and every $1$-Lipschitz function $\phi$, we have:
    \begin{equation}
        \int_X \phi_f(\overline u_\theta x)d\lambda(x)=\int_X \Ee_{s_\theta}^{T}\phi_f(\overline u_\theta x)d\lambda(x)+O(\abs{\overline u_\theta}T\eps).
    \end{equation}
\end{lemma}

\begin{proof}
    Since $s_\theta$ commutes with $\overline U_\theta$ (by definition), if we define the function $\phi_{f,\overline u_\theta}(x)=\phi_f(\overline u_\theta x)$ then:
    \begin{equation}
        \Ee_{s_\theta}^{T}\phi_f(\overline u_\theta x)=\Ee_{s_\theta}^{T}\phi_{f,\overline u_\theta}(x).
    \end{equation}
    Since the measure $\lambda$ is $\eps$-$P$-almost invariant, and since the function $\phi_{f,\overline u_\theta}$ is $\abs{\overline u_\theta}$-Lipschitz, we know that:
    \begin{align*}
        \int_X\Ee_{s_\theta}^{T}\phi_{f,\overline u_\theta}(x)d\lambda(x) & = \int_X \frac1T \left(\int_0^T \phi_{f,\overline u_\theta}(s_\theta^{t} x)dt\right)d\lambda(x) \\ & = \frac1T \int_0^T \left(\int_X \phi_{f,\overline u_\theta}(s_\theta^{t}x)d\lambda(x) \right) dt   \\ & =\int_X \phi_{f,\overline u_\theta}(x)d\lambda(x)+O(\abs{\overline u_\theta}T\eps).
    \end{align*}
    By definition of $\phi_{f,\overline u_\theta}$ and since $\overline u_\theta$ commutes with $s_\theta$, we deduce:
    \begin{align*}
        \int_X \phi_f(\overline u_\theta x)d\lambda(x)=\int_X \phi_{f,\overline u_\theta}(x)d\lambda(x)= \int_X\Ee_{s_\theta}^{T}\phi_{f,\overline u_\theta}(x)d\lambda(x)+O(\abs{\overline u_\theta}T\eps)
        &\\=\int_X\Ee_{s_\theta}^{T}\phi_f(\overline u_\theta x)d\lambda(x)+O(\abs{\overline u_\theta}T\eps)
    \end{align*}
    as desired.
\end{proof}

To recall the definition of an almost factor, we recall the one sided Hausdorff distance.
\begin{definition}[One sided Hausdorff metric]\label{def: dh+1}
    Given a metric space $(X,d)$ and sets $A,B\subset X$, we define the one sided Hausdorff distance from $A$ to $B$ by:
    \begin{equation}
        d_H^+(A,B)=\sup_{a\in A}d(a,B).
    \end{equation}
\end{definition}
Using this definition, we can recall what an almost factor is.
\begin{definition}[Almost factor]\label{def: almost factor2}
	Let $X$ be a $G$-space, let $W\subset X$ be measurable, $\delta>0$ and let $\Qq$ denote the space of maximal parabolic subgroups.

    We say that $W$ has a projective $(R,\delta)$-factor if there is a measurable map $\pi:W\rightarrow \Qq$ such that for every $x\in W$ we have
    \begin{align}
        d_H^+(\stab_G(x)\cap G^{||}_{R},\pi(x))\leq \delta.
    \end{align}
    When $W$ consists of one point $x$, we simply say that the action has a projective factor at $x$.
\end{definition}
The following theorem is a step towards the full proof, where we assume that the measure $\lambda$ being quantitative Mautner with general parameters. 
Later in the final proof, we compute these parameters and drop this assumption.
\begin{theorem}[Almost factor versus additional invariance]\label{thm: dic thm}
    Suppose $X$ is a $G$-space and that $\lambda$ is a norm $\eps$-$P$-invariant measure on $X$.
    Let $f:X\rightarrow \RR$ be a $1$-bounded $1$-Lipschitz function and let $R>0$ large.
    Recall the definition of $\hat f_R$ given in Definition \ref{def: fr}.
    Also, let $0<\eta<1$ be some number. 
    
    Let $n\in \NN$ and suppose $\lambda$ is $(\kappa,\delta,T,\hat f_{R^{An}})$-QM. 
    Then there exists $X''\subset X$ such that $\lambda(X'')\geq 1-\kappa$ and one of the following holds:
    \begin{enumerate}
        \item There exists a subset $X'\subset X$ such that $\lambda(X')\geq \eta$ such that $X'$ is a projective $(R,CR^{c'}/n^c)$-factor;
        \item For every $\theta$ and for every $1$-Lipschitz function $\phi:f(X)\rightarrow \RR$ and every $\overline u_\theta\in (\overline U_\theta)_1$, we 
        have: 
        \begin{align*}
            \int_X\phi_f(\overline u_\theta x)d\lambda(x)=\int_X\phi_f(x)d\lambda(x)+O(T\eps+R^{-\frac n2\on{dim}(G)^{-1}}\\
            +(\kappa+\eps)^{1/2}R^{An}+8n^3R^{8n}e^{-T^{1/2}}+8nT^{-1/2}+2n\delta R^{An}+\eta).
        \end{align*}
    \end{enumerate}
\end{theorem}

\begin{proof}
    Fix $\theta$ and take $X''$ to be the set from Lemma \ref{lem: factor function from mautner} so that indeed $\lambda(X'')\geq 1-\kappa$.
    
    Suppose that the measure of the set of points $x\in X''$ for which $\stab_G(x)\cap G_R^{||}$ and $V_\theta$ do not generate $G$ in $n$-steps is at most $\eta$.
    Otherwise, using Lemma \ref{lem: fast generation} for $S=\stab_G(x)\cap G_{R}^{||}$ we know that there exists $V_x\geq V_\theta$, a maximal proper subgroup of $G$, such that $d(S,V_x)\leq CR^{c'}/n^c$.
    By Theorem \ref{thm: upgrade to parabolic} we can assume that $V_x$ is parabolic for every $x$.
    Let $X'$ to be the set where the above property holds, so that $\lambda(X')\geq \eta$. 
    For every $x\in X'$, define $W_x=X'\cap G_1x$ and note that, potentially replacing $C$ by $2C$ (which we still denote $C$), we obtain that $G/V_x$ is an $(R,C/n^c)$-almost factor on the set $W_x$.
    Writing $X'=\cup_xW_x$ and recalling that each $V_x$ is parabolic, we obtain that $X'$ has a projective $(R,CR^{c'}/n^c)$-factor, since 
    \begin{align*}
        d(S,V_x)=d_H^+(\stab_G(x)\cap G_R^{||},V_x).
    \end{align*}
    We remark that the map $x\mapsto V_x$ can be shown to be measurable.

    So, the measure of the set of points $x\in X''$ such that $V_\theta,\stab_G(x)\cap G_R^{||}$ generate $G$ in $n$-steps is at least $1-\eta$.
    By Proposition \ref{prop: final prop} we can find $Z\subset X''$ of measure at least $1-3(\kappa+\eps)^{1/2}R^n-\eta$ (the set $Z$ is taken as the intersection of $X'''$ from Proposition \ref{prop: final prop} with the set of points for which $\on{stab}_G(x)\cap G_R^{||}$ and $V_\theta$ generate $G$ in $n$ steps) such that for every $x\in Z$ and every $\overline u_\theta \in (\overline U_\theta)_1$ we have:
    \begin{equation}\label{eq: 1600}
        \abs{\Ee_{s_\theta}^{T}\phi_f(\overline u_\theta x)-\Ee_{s_\theta}^{T}\phi_f(x)}\leq R^{-\frac n2\on{dim}(G)^{-1}}+(\kappa+\eps)^{1/2}R^{An}+8n^3R^{An}e^{-T^{1/2}}+8nT^{-1/2}+2n\delta\\
    \end{equation}    
    
    By the $\eps$-$P$-norm almost invariance of the measure $\lambda$, since $\lambda(X'')\geq 1-\kappa$ and by Lemma \ref{lem: final} and by Equation \eqref{eq: 1600}:
    \begin{align*}
        \int_X\phi_f(\overline u_\theta x)d\lambda(x)  & =\int_X \Ee_{s_\theta}^{T}\phi_f(\overline u_\theta x)d\lambda(x)+O(\abs{\overline u_\theta}T\eps)\\ & =\int_{X''} \Ee_{s_\theta}^{T}\phi_f(\overline u_\theta x)d\lambda(x)+O(\abs{\overline u_\theta}T\eps+\kappa) \\
            &=\int_{Z} \Ee_{s_\theta}^{T}\phi_f( \overline u_\theta x)d\lambda(x)+O(\abs{\overline u_\theta}T\eps+\kappa+3(\kappa+\eps)^{1/2}R^n+\eta)\\ & =\int_{X} \Ee_{s_\theta}^{T}\phi_f(x)d\lambda(x)  +O(\abs{\overline u_\theta}T\eps+
            4(\kappa+\eps)^{1/2}R^n+\eta+R^{-\frac n2\on{dim}(G)^{-1}}
            \\& \qquad  \qquad \qquad \qquad +(\kappa+\eps)^{1/2}R^{An}+8n^3R^{8n}e^{-T^{1/2}}+8nT^{-1/2}+2n\delta R^{An})\\
            & =\int_{X} \phi_f(x)d\lambda(x)+O(T\eps+R^{-\frac n2\on{dim}(G)^{-1}} \\ &\qquad\qquad\qquad\qquad +
            (\kappa+\eps)^{1/2}R^{An}+8n^3R^{8n}e^{-T^{1/2}}+8nT^{-1/2}+2n\delta R^{An}+\eta)
        \end{align*}
        as desired.
        Since $\theta$ was arbitrary, we obtain option $(b)$.
\end{proof}

We recall Definition \ref{def: fr}.
\begin{definition}\label{def: fr2}
    Given a Lipschitz function $f:X\rightarrow I$ where $I\subset \RR$ is a compact interval, $R>0$ and a diagonal element $s\in S$ we define a function $\hat f_R:X\rightarrow \on{Lip}((C_\theta)_R^{||}\rightarrow I)$ by 
    \begin{align}
        [\hat{f}_R(x)](c_\theta)=f(c_\theta x)\text{ for every $c_\theta\in (C_\theta)_R^{||}$ }
    \end{align}
    and denote
    \begin{align}
        Y_R:=\on{Lip}_1((C_\theta)_R^{||}\rightarrow I).
    \end{align}
    For $\phi:I\rightarrow \RR$ and $f:X\rightarrow I$ we denote $\phi_f=\phi\circ f$.
\end{definition}
\begin{claim}\label{cl: covering number bound}
    For every $\delta>0$, the covering number $N_\delta(Y_R)$ of $Y_R$ by balls of radius $\delta$ is at most $O(e^{(R\delta^{-1})^{d}})$ where $d=\on{dim}(G)$.
\end{claim}

We are now ready to prove Theorem \ref{thm: qnz}.

\begin{proof}[Proof of Theorem \ref{thm: qnz}]
    Let $X$ be a $G$-space and let $\nu$ be an $\eps$-almost stationary measure.
    Our first step is to reduce to the discussion of norm $O_G(\eps^{1/8})$ almost $P$-invariant measures.

    To do so, we write, using the $\eps$-almost stationarity and denoting $N=\ceil{\eps^{-1/2}}$
    \begin{equation}\label{eq: cezaro decomp}
        \nu=\int_X\left(\frac{1}{N}\sum_{i=1}^{N-1}\mu^{*i}*\delta_x\right)d\nu(x)+O(\eps^{1/2}).
    \end{equation}
    We can apply Theorem \ref{af:decomposition} to decompose the measures $\nu_N=\frac{1}{N}\sum_{i=1}^N\mu^{*i}$ as
    \begin{align}
        \nu_N=\int_K k_*\lambda_{N,k}dm_K(k)
    \end{align}
    where $\lambda_{N,k}$ is norm $C_0N^{-1/4}$-almost invariant under $P$ for some $C_0=C_0(G)$.
    Convolving the above equation with $\delta_x$ and denoting $\nu_{N,x}=\frac{1}{N}\sum_{i=1}^{N}\mu^{*i}*\delta_x$ we may write
    \begin{equation}\label{eq: local furst decomp}
        \nu_{N,x}=\int_Kk_*\lambda_{N,k,x}dm_K(k)
    \end{equation}
    where every $k\in K$ the measure $\lambda_{N,k,x}$ is $\eps':=N^{-1/4}=C_0\eps^{1/8}$ norm almost $P$-invariant. 
    It will therefore suffice to analyze each of the measures $\lambda_{N,k,x}$ and either prove that for most points with respect to most $x,k$ the measures $\lambda_{N,k,x}$ are weak $f$-almost $G$-invariant, or there exists an almost factor (this will be made explicit in the end of the proof using a double application of Markov inequality).

    For the time being we denote $\lambda=\lambda_{N,k,x}$ for some $k\in K,x\in X$.

    Put
\[
 d:=\dim G,\qquad L:=\log(1/\eps),\qquad
 b:=\frac1{1000d},\qquad \beta:=\frac{b}{4d}.
\]
Write $c_{\mathrm{fg}},c'$ for the exponents in
Lemma~\ref{lem: fast generation}, and choose
$0<c\leq\min\{c_{\mathrm{fg}},\beta\}$.
Let $R_0\geq2$ be sufficiently large for the preceding propositions,
set $A_G:=40d\log R_0$, and suppose
$\zeta\geq A_G/\log L$. Throughout, $\eps$ is sufficiently small
depending only on $G$, and $\eps'\leq C_0\eps^{1/8}$.

Set $R:=L^{\zeta/(40d)}$, so that $R\geq R_0$.
For $\zeta\geq1/50$, any fixed maximal proper parabolic $Q$ gives
\[
 d_H^+(\operatorname{stab}_G(x)\cap G_R^{\|},Q)
 \leq \sup_{g\in G_R^{\|}}d_G(g,e)
 \ll_G 1+\log R
 \ll_G R^{c'}\zeta^c.
\]
Thus the projective alternative holds on all of $X$ after
enlarging $C$. We may therefore assume $\zeta\leq1/50$, and choose
\[
 n:=\left\lfloor\frac1{25\zeta}\right\rfloor,\qquad
 \delta:=L^{-8b},\qquad m:=\lceil L^{1/2}\rceil.
\]
These choices give
\[
 \frac1{50}\leq n\zeta\leq\frac1{25},\qquad
 L^{b/2}\leq R^n\leq L^b,\qquad n\ll_G\log L.
\]

We verify the hypotheses of Theorem~\ref{thm: effective mautner}.
For each $\theta$, use the bounded space
\[
 Y_{\theta,r}:=
 \operatorname{Lip}_1((C_\theta)_r^{\|},[-1,1])
\]
with the supremum metric. The orbit-map normalization shows that
$\hat f_r(x)\in Y_{\theta,r}$. Moreover, for $g\in P_1$,
\[
 \|\hat f_r(gx)-\hat f_r(x)\|_\infty
 \leq\sup_{a\in(C_\theta)_r^{\|}}d_G(ag,a)
 =d_G(g,e).
\]
Hence $\hat f_r$ is $P,1$-Lipschitz, as required.

By Claim~\ref{cl: covering number bound}, with
$H_\theta:=\log N_{\delta/10}(Y_{\theta,R^{An}})$,
\[
 1\leq H_\theta
 \ll_G (R^{An}/\delta)^d\log(2/\delta)
 \ll_G L^{12bd}\log L
 \ll_G L^{1/50}.
\]
The lower bound follows from the constant functions in
$Y_{\theta,R^{An}}$. Define
\[
 q_\theta:=
 \left[
   \left(\frac{3mH_\theta}{2}\right)^{-1/4}
   +m^{-1/2}\left(\frac{3H_\theta}{2}\right)^2
 \right]^{1/2},
 \qquad
 \kappa:=\max_\theta q_\theta+2^m\eps'.
\]
Then
\[
 q_\theta\ll_G(mH_\theta)^{-1/8}+m^{-1/4}H_\theta
 \ll_G L^{-1/16},
 \qquad
 2^m\eps'\ll_G e^{-L/16},
 \qquad
 \kappa\ll_G L^{-1/16}.
\]
Also, for every integer $k\geq m/2$,
\[
 (mH_\theta)^{-1/2}+e^{-k/2}
 \leq L^{-1/4}+e^{-m/4}<\delta/10.
\]
Thus all the numerical hypotheses of
Theorem~\ref{thm: effective mautner} hold, including
$p_\theta:=1-q_\theta>1/2$.

For each $\theta$, let $\mathcal A_\theta\subset\{1,\ldots,m-1\}$
be the good set supplied by that theorem.
Since there are only finitely many $\theta$ and
$\sum_\theta q_\theta=o(1)$, their intersection has proportion
greater than $1/2$. Choose
\[
 \ell\in\bigcap_\theta\mathcal A_\theta\cap[m/2,m],
 \qquad T:=2^\ell.
\]
Then $2^{m/2}\leq T\leq2^m$, and $\lambda$ is
$(\kappa,\delta,T,\hat f_{R^{An}})$-QM for every $\theta$.

The error terms in Theorem~\ref{thm: dic thm} satisfy
\begin{align*}
 T\eps'&\ll_G e^{-L/16},&
 R^{-n/(2d)}&\leq L^{-\beta},\\
 (\kappa+\eps')^{1/2}R^{An}
   &\ll_G L^{-1/32+4b},&
 n\delta R^{An}&\ll_G(\log L)L^{-4b},\\
 nT^{-1/2}
   &\ll_G(\log L)2^{-L^{1/2}/4},\\
 n^3R^{8n}e^{-\sqrt T}
   &\ll_G(\log L)^3L^{8b}
       \exp\bigl(-2^{L^{1/2}/4}\bigr).
\end{align*}
Each is $O_G(L^{-\beta})$.

Finally,
\[
 n^{-c_{\mathrm{fg}}}
 \leq(50\zeta)^{c_{\mathrm{fg}}}\ll_G\zeta^c.
\]
Consequently, Theorem~\ref{thm: dic thm} gives either a subset
$X'\subset X$ of $\lambda$-measure at least $\eta$ with a projective
$(R_\zeta,\delta_\zeta)$-factor, where
\[
 R_\zeta=L^{\zeta/(40d)},\qquad
 \delta_\zeta=C L^{c'\zeta/(40d)}\zeta^c,
\]
or, for every $\theta$, every
$\overline u_\theta\in(\overline U_\theta)_1$, and every
$1$-bounded $1$-Lipschitz $\phi:f(X)\to\mathbb R$,
\[
 \int_X\phi_f(\overline u_\theta x)\,d\lambda(x)
 =
 \int_X\phi_f(x)\,d\lambda(x)+O_G(L^{-c}+\eta).
\]

Enlarge $C$ so that every factor produced above has tolerance at
most $\delta_\zeta/2$, and put
\[
 W_\zeta:=
 \left\{x\in X:
   \inf_{Q\in\mathcal Q}
   d_H^+(\operatorname{stab}_G(x)\cap G_{R_\zeta}^{\|},Q)
   <\delta_\zeta
 \right\}.
\]
The defining distance is Borel in $x$ and continuous in $Q$.
A countable dense family in $\mathcal Q$ therefore gives a
measurable choice of parabolic on $W_\zeta$, so $W_\zeta$ has an
$(R_\zeta,\delta_\zeta)$-projective factor.
Our choices of metric and norm also give $kW_\zeta=W_\zeta$
for every $k\in K$.

If $\nu(W_\zeta)\geq\eta$ for every admissible $\zeta$, the second
alternative of the theorem follows. Otherwise, fix an admissible
$\zeta$ with $\nu(W_\zeta)<\eta$, and write $W:=W_\zeta$.

We first justify passage from the root subgroups to $G$.
The preceding argument applies with any threshold $t\in(0,1)$
in place of $\eta$. If $\lambda(W)<t$, its factor alternative
is impossible. Moreover, it can be applied to every translated
observable
\[
 f_h(x):=f(hx),\qquad h\in G.
\]
Indeed, equip $X$ with $d_h(x,y):=d_X(hx,hy)$. Then $f_h$ is
$1$-bounded and $1$-Lipschitz, and the orbit maps remain
$1$-Lipschitz, since
\[
 d_h(ax,bx)\leq d_G(ha,hb)=d_G(a,b).
\]
Thus the preceding Mautner argument and its resulting root-subgroup
estimates apply uniformly to every $f_h$. Their good Mautner times
may depend on $h$; no simultaneous choice is needed.

Set
\[
 S:=P_1\cup\bigcup_\theta(\overline U_\theta)_1.
\]
There is $M_G<\infty$ such that $G_1\subset S^{M_G}$:
the corresponding Lie algebras generate $\mathfrak g$, so a finite
product of these identity neighbourhoods contains an identity
neighbourhood in $G$, and compactness of $G_1$ gives the assertion.
Telescoping along these words, using the translated-observable
estimates and the norm $P$-almost invariance, gives
\[
 \lambda(W)<t
 \quad\Longrightarrow\quad
 \sup_{\substack{h\in G\\g\in G_1}}
 \Wf{f_h}(g_*\lambda,\lambda)
 \ll_G L^{-c}+t.
\]
Consequently, for every $k\in K$ and $g\in G_1$,
\[
 \Wf{f}(g_*k_*\lambda,k_*\lambda)
 =
 \Wf{f_k}((k^{-1}gk)_*\lambda,\lambda)
 \ll_G L^{-c}+t,
\]
provided $\lambda(W)<t$.

We now average. Write $\rho:=\nu\otimes m_K$ and
\[
 \bar\nu:=
 \int_{X\times K}k_*\lambda_{N,k,x}\,d\rho(x,k)
 =\int_X\nu_{N,x}\,d\nu(x).
\]
By \eqref{eq: cezaro decomp},
\[
 \|\bar\nu-\nu\|_{TV}\leq C_G\eps^{1/2}.
\]
Since $W$ is $K$-invariant,
\[
 \int_{X\times K}\lambda_{N,k,x}(W)\,d\rho(x,k)
 =\bar\nu(W)\leq \eta+C_G\eps^{1/2}=:A.
\]
We may assume $A<1$, since otherwise the first alternative is
immediate after enlarging $C$. Put $t:=A^{1/4}$.
A single application of Markov's inequality gives
\[
 \rho\bigl\{(x,k):\lambda_{N,k,x}(W)\geq t\bigr\}
 \leq A/t=A^{3/4}.
\]
On the complementary set, the preceding invariance estimate
applies; on the exceptional set, use $\Wf{f}\leq2$.
Convexity therefore gives, for every $g\in G_1$,
\[
 \Wf{f}(g_*\bar\nu,\bar\nu)
 \ll_G L^{-c}+A^{1/4}+A^{3/4}
 \ll_G L^{-c}+A^{1/4}.
\]
Finally,
\begin{align*}
 \Wf{f}(g_*\nu,\nu)
 &\leq
 2\|\nu-\bar\nu\|_{TV}
 +\Wf{f}(g_*\bar\nu,\bar\nu)\\
 &\ll_G L^{-c}+\eta^{1/4}+\eps^{1/8}
 \ll_G L^{-c}+\eta^{1/4}.
\end{align*}
This is the first alternative, and completes the proof.
\end{proof}

\section{A decomposition into almost invariant and almost projective parts}\label{sec: decomposition qnz}

In this section, we apply iteratively our main theorem to obtain a decomposition of almost stationary measures into an almost invariant and an almost stationary part. To do so, we apply the main theorem in multiple rounds, as a tradeoff each round weakens the bounds available. 

This decomposition will be key in the companion paper \cite{QFG}. 

\begin{definition}\label{def:projective}
A probability measure is $(R,\delta)$-projective if it is supported
on a measurable set with an $(R,\delta)$-projective factor.
Here a factor on $W$ means a measurable choice of proper parabolic
subgroups $x\mapsto Q_x$ such that
\[
 d_H^+(\Stab_G(x)\cap G_R^{\|},Q_x)\leq\delta\qquad(x\in W).
\]
Throughout, ``supported on'' means that the complement has measure zero.
\end{definition}

Write $\log^{(j)}$ for the $j$-fold iterated logarithm, with
$\log^{(0)}t:=t$, and fix
\[
 \log^*t:=\min\{j\geq0:\log^{(j)}t\leq e\}.
\]

\begin{theorem}[Ergodic almost Nevo--Zimmer]\label{thm:ergodic-qnz}
Let $\nu,X,\eps$ satisfy the assumptions of Theorem~\ref{thm: qnz},
let $\Ff\subseteq\Ff_1(X)$, let $k\geq1$, and fix
$\vartheta\in(0,1]$. Set $\eps_0:=\eps$.
For $k\geq2$, also put
\[
 L_j:=\log^{(3j)}(1/\eps)\quad(0\leq j\leq k),\qquad
 U_k:=\log^{(3k-1)}(1/\eps),
\]
and $\eps_j:=L_j^{-1/2}$ for $1\leq j\leq k-1$.
There are $a,b,C_1,T>0$, depending only on $G$, such that the following
holds, with the additional assumption
$L_k\geq T(1+\log k+|\log\vartheta|)$ when
$k\geq2$.
Define $q_0:=1$ and, for $0\leq i\leq k-2$,
\[
 q_{i+1}:=q_i-\min\{k^{-1},\vartheta q_i/2\}.
\]
There are an integer $0\leq j\leq k-1$ and a decomposition
\begin{equation}\label{eq:decomposition}
 \nu=\lambda\nu_{\mathrm{proj}}+q\nu_{\mathrm{inv}},\qquad
 q:=q_j,\qquad\lambda:=1-q,
\end{equation}
into probability measures, with zero terms omitted, for which
$\nu_{\mathrm{inv}}$ is $\eps_j$-almost stationary. If $\lambda>0$, then $\nu_{\mathrm{proj}}$ is
$(R_k,\delta_k)$-projective, where one may take
\begin{equation}\label{eq:product-rate}
 R_k:=U_k^a,\qquad
 \delta_k:=C_1 L_k^cU_k^{-a-c/2},\qquad
 R_k\delta_k=C_1\frac{L_k^c}{U_k^{c/2}}.
\end{equation}
If $j<k-1$, then $\nu_{\mathrm{inv}}$ is
\[
 \left(f,C\left(|\log\eps_j|^{-c}
                  +\min\left\{\vartheta,\frac{2}{kq}\right\}^{1/4}\right)\right)
 \text{-almost invariant}\qquad(f\in\Ff).
\]
If $j=k-1$, then for every $\eta\in[0,1]$ either $\nu_{\mathrm{inv}}$ is
$(f,C(|\log\eps_j|^{-c}+\eta^{1/4}))$-almost invariant for every
$f\in\Ff$, or, for every $\zeta>A_G/\log|\log\eps_j|$, there is
a set of $\nu_{\mathrm{inv}}$-measure at least $\eta$ with an
$(R_\zeta,\delta_\zeta)$-projective factor, where
\[
 R_\zeta:=|\log\eps_j|^{\zeta/(40\dim G)},\qquad
 \delta_\zeta:=C|\log\eps_j|^{c'\zeta/(40\dim G)}\zeta^c.
\]
Moreover,
\[
 q_j\geq(1-\vartheta/2)^j,\qquad
 q_{k-1}\leq\min\{1,2/(\vartheta k)\}.
\]
For every $k\geq2$ and every $f\in\Ff$, the finite measure
$q\nu_{\mathrm{inv}}$ is
\begin{equation}\label{eq:weighted-invariance}
 (f,C_1(L_k^{-b}+(\vartheta k)^{-1/4}))\text{-almost invariant},
\end{equation}
and its normalization $\nu_{\mathrm{inv}}$ is
\begin{equation}\label{eq:normalized-invariance}
 (f,C_1(L_k^{-b}+(\vartheta kq)^{-1/4}))\text{-almost invariant}.
\end{equation}
\end{theorem}

For $k=1$ and $\Ff=\{f\}$ this is exactly
Theorem~\ref{thm: qnz}, with both parameters $\eta$ and $\zeta$ retained. A leaner statement can be deduced from it 

\subsection{Averaging and extraction of projective measures}

For a probability measure $\rho$ and $m\geq1$, define
\[
 A_\mu^m\rho:=\frac1m\sum_{i=1}^m\mu^{*i}*\rho.
\]
Telescoping gives
\begin{equation}\label{eq:averaging}
 \|\mu*A_\mu^m\rho-A_\mu^m\rho\|_{\TV}\leq\frac2m,
 \qquad
 \|A_\mu^m\rho-\rho\|_{\TV}\leq m\eps
\end{equation}
when $\rho$ is $\eps$-almost stationary.

\begin{lemma}[Heat-kernel concentration]\label{lem:concentration}
There are $v_G,C_H,c_H>0$ such that
\begin{equation}\label{eq:concentration}
 \mu^{*i}(G_{v_Gt}^{\,c})\leq C_He^{-c_Ht}
 \qquad(t\geq i\geq1).
\end{equation}
In particular the same bound holds for the average over $1\leq i\leq m$
when $t\geq m$.
\end{lemma}

\begin{proof}
The Gaussian tail estimate for the heat kernel, as in
\cite[Appendix~B]{QFG}, bounds the tail outside $G_s$ at
time $t_Gi$ by
$C\exp(-c(s-C_0i)^2/i)$ for $s\geq C_0'i$.
Choose $v_G>\max\{2C_0,C_0'\}$. For $t\geq i$, the exponent at
$s=v_Gt$ is at most $-c't^2/i\leq-c't$.
The fixed comparison between the Cartan and group distances is included
in $v_G$ and $C_H$. Averaging preserves the bound.
\end{proof}

Conjugation changes the projective-factor parameters by at most
$(R,\delta)\mapsto(e^{-D_Gr}R,e^{D_Gr}\delta)$ for $g\in G_r$;
see \cite[Claim~5.15]{QFG} and the corresponding norm-ball
comparison. Fix $D\geq1$ large enough to include the factor $v_G$ in
Lemma~\ref{lem:concentration}. Put
\[
 B_m\rho:=\frac1m\sum_{i=1}^m
       (\mu^{*i}|_{G_{v_Gm}})*\rho,\qquad t_m:=C_He^{-c_Hm}.
\]
Then $B_m\rho\leq A_\mu^m\rho$ and
$\|A_\mu^m\rho-B_m\rho\|_{\TV}\leq t_m$.
If $\rho$ is $(R,\delta)$-projective, the normalization of $B_m\rho$
is $(e^{-Dm}R,e^{Dm}\delta)$-projective. Indeed, transport the factor
under each translation in $G_{v_Gm}$ and choose a parabolic at each
point in their union. Measurable selection applies to the resulting
relation with the compact space of parabolics; completing the measures
and discarding null sets gives a measurable factor. The same argument
shows that projectivity passes to submeasures and to finite sums, with
the smallest radius and largest error.

\begin{lemma}[Exact peeling]\label{lem:peeling}
Let $\rho$ be $\eps$-almost stationary and suppose that $W$ has an
$(R,\delta)$-projective factor, with $p:=\rho(W)\geq\eta>0$.
If $t_m\leq1/8$ and $m\eps\leq\eta/8$, then, for every
$0<\theta\leq\eta/2$, there is an exact decomposition
\[
 \rho=\theta\sigma+(1-\theta)\rho'
\]
such that $\sigma$ is $(e^{-Dm}R,e^{Dm}\delta)$-projective and
\begin{align}
 \|\mu*\sigma-\sigma\|_{\TV}
   &\leq C\left(m^{-1}+t_m+\frac{m\eps}{\eta}\right),
       \label{eq:piece-stationarity}\\
 \|\mu*\rho'-\rho'\|_{\TV}
   &\leq\frac{\eps+C\theta(m^{-1}+t_m+m\eps/\eta)}{1-\theta}.
       \label{eq:residual-stationarity}
\end{align}
\end{lemma}

\begin{proof}
Set $\sigma_0:=\rho|_W/p$ and
$\beta:=\rho\wedge(pB_m\sigma_0)$, and write $r:=\beta(X)$.
Since $pB_m\sigma_0\leq A_\mu^m\rho$, \eqref{eq:averaging} gives
\[
 r\geq p(1-t_m)-m\eps\geq3p/4.
\]
Define $\sigma:=\beta/r$ and
$\rho':=(\rho-\theta\sigma)/(1-\theta)$.
These are probability measures since $\theta\leq p/2<r$.
Moreover $\beta\leq pA_\mu^m\sigma_0$ and
$p-r\leq pt_m+m\eps$, so
\[
 \|\sigma-A_\mu^m\sigma_0\|_{\TV}
 \leq2(p-r)/p\leq2t_m+2m\eps/p.
\]
Together with \eqref{eq:averaging}, this proves
\eqref{eq:piece-stationarity}. Projectivity follows from
$\beta\leq pB_m\sigma_0$. Finally,
\[
 (1-\theta)(\mu*\rho'-\rho')
   =\mu*\rho-\rho-\theta(\mu*\sigma-\sigma),
\]
which proves \eqref{eq:residual-stationarity}.
\end{proof}

\begin{proof}[Proof of Theorem~\ref{thm:ergodic-qnz}]
For $k=1$, take $j:=0$, $q:=1$ and $\nu_{\mathrm{inv}}:=\nu$, and apply
Theorem~\ref{thm: qnz}; if the invariant alternative
fails for one member of $\Ff$, that member supplies the asserted factor
for every admissible $\zeta$.

Suppose $k\geq2$. We construct
\[
 \nu=\pi_j+q_j\rho_j,
 \qquad0\leq j\leq k-1,
\]
starting with $\pi_0:=0$ and $\rho_0:=\nu$. Let
$\eps_0:=\eps$, and use the stationarity bounds
$\eps_j:=L_j^{-1/2}$ for $j\geq1$.
Their validity will be checked below. For $j\leq k-2$, put
\begin{gather*}
 \eta_j:=\min\left\{\vartheta,\frac{2}{kq_j}\right\},\qquad
 u_j:=\log|\log\eps_j|,\\
 \alpha:=\frac{c}{2(1+c')},\qquad
 \zeta_j:=40(\dim G)\alpha\frac{\log u_j}{u_j}.
\end{gather*}
Apply Theorem~\ref{thm: qnz} to $\rho_j$ with parameter $\eta_j$.
If the invariant alternative holds for every $f\in\Ff$, stop.
Otherwise choose a function for which it fails and take $\zeta=\zeta_j$.
The assumption on $L_k$ ensures $40(\dim G)\alpha\log u_j>A_G$,
and hence $\zeta_j>A_G/u_j$.
Substitution gives a set of $\rho_j$-measure at least $\eta_j$ with
\begin{equation}\label{eq:qnz-stage}
 R_j^{(0)}=u_j^\alpha,
 \qquad
 R_j^{(0)}\delta_j^{(0)}
       \leq C\frac{(\log u_j)^c}{u_j^{c/2}}.
\end{equation}
Fix $0<\omega<1/(4D)$, and set
\[
 m_{j+1}:=\floor{\omega\log R_j^{(0)}},\qquad
 \theta_j:=\frac{\eta_j}{2}.
\]
Lemma~\ref{lem:peeling} gives
$\rho_j=\theta_j\sigma_{j+1}+(1-\theta_j)\rho_{j+1}$.
Define $\pi_{j+1}:=\pi_j+q_j\theta_j\sigma_{j+1}$.
Since $q_j\theta_j=\min\{k^{-1},\vartheta q_j/2\}$, the remaining
mass is $q_{j+1}$ as in the statement.

Here are uniform bounds for the induction. For $j=0$,
$\log u_0=L_1$, and for $j\geq1$,
\[
 u_j=\log\bigl(\tfrac12\log L_j\bigr),\qquad
 \log u_j=L_{j+1}+o(1).
\]
Consequently $m_{j+1}$ lies between two fixed positive multiples of
$L_{j+1}$. Also $\theta_j\leq1/2$ and
$\theta_j/\eta_j=1/2$. Thus \eqref{eq:residual-stationarity} implies
\[
 \|\mu*\rho_{j+1}-\rho_{j+1}\|_{\TV}
 \leq C(m_{j+1}^{-1}+m_{j+1}\eps_j+t_{m_{j+1}})
 \leq L_{j+1}^{-1/2}.
\]
To justify the last inequality and the extraction hypotheses, use
$L_j=\exp^{(3)}(L_{j+1})$ and $L_{j+1}\geq L_k$.
For $T$ sufficiently large,
$L_k\geq T(1+\log k+|\log\vartheta|)$ ensures, at every
stage, $m_{j+1}\eps_j\leq\min\{\vartheta/8,1/(4k)\}$,
$t_{m_{j+1}}\leq1/8$, and
$C(m_{j+1}^{-1}+m_{j+1}\eps_j+t_{m_{j+1}})
\leq L_{j+1}^{-1/2}$. This proves the induction with constants
independent of the number of stages.

Put $v:=\alpha(1-D\omega)>0$. Each extracted measure has a factor
with radius at least $u_j^v$ and error at most
$C(\log u_j)^c u_j^{-v-c/2}$. For $j\leq k-2$ the above formulas
give $u_j\geq U_k$. The functions $u\mapsto u^v$ and
$u\mapsto(\log u)^cu^{-v-c/2}$ are respectively increasing and
decreasing once $T$ is large. Choose $a:=v/2$. Every extracted measure
therefore has the common parameters in \eqref{eq:product-rate}.
Their normalized sum has the same parameters.

If the construction stops at stage $j\leq k-2$, its remaining
probability measure is
\[
 \bigl(f,C(|\log\eps_j|^{-c}+\eta_j^{1/4})\bigr)
 \text{-almost invariant}\qquad(f\in\Ff).
\]
For some $b>0$ depending only on $G$,
$|\log\eps_j|^{-c}\leq C L_k^{-b}$ throughout these stages.
Multiplication by $q_j$ proves both stated invariance estimates.
If no such stop occurs, retain $\rho_{k-1}$.
The recurrence gives $q_j\geq(1-\vartheta/2)^j$. Until $q_j$ first
falls below $2/(\vartheta k)$, its decrease at each step is $1/k$;
afterwards it is decreasing. Hence
$q_{k-1}\leq\min\{1,2/(\vartheta k)\}$.
Theorem~\ref{thm: qnz}, applied to this $\eps_{k-1}$-almost stationary
measure, gives the asserted terminal alternative for every $\eta$.
For probability measures $W_1^f\leq2$; together with this bound on
$q_{k-1}$, this proves
\eqref{eq:weighted-invariance} and
\eqref{eq:normalized-invariance} in this case as well.
Normalize $\pi_j$ if it is nonzero and set $\nu_{\mathrm{inv}}:=\rho_j$.
\end{proof}

\appendix

\section{A matrix decomposition}\label{A matrix decomposition}
\begin{nota}
    We recall that $P$ denotes the minimal parabolic subgroup in $G$ and that $\overline V$ denotes the unipotent radical of the opposite parabolic subgroup.
\end{nota}
\begin{definition}
    Given $g\in G$, we say that $g$ is \emph{composite} if $g$ can be written uniquely as $g=\overline vp$ for some $\overline v\in \overline V$ and $p\in P$.
    
\end{definition}
\begin{remark}
    Let $\theta \subset \Delta$, let $V_\theta$ be the subgroup expanded by $s_\theta$ and let $G^{-}_\theta=\overline{P}_\theta=L_\theta \overline{V}_\theta$ be the corresponding non-expanded parabolic subgroup. If $g\in \overline{V}P,$ then $g$ admits a unique decomposition $g=g^{-}_\theta g^{+}_\theta$ where $g^{+}_\theta \in V_\theta$ and $g^{-}_\theta \in G^{-}_\theta$.
    Indeed, $P=(P\cap L_\theta) V_\theta$ and $\overline{V} (P\cap L_\theta)\subset \overline{P}_\theta$.
\end{remark}

We will not prove the following classical fact.
\begin{claim}[Almost everything is composite]\label{cl: stuff decompose}
    There exists a proper submanifold $\Mm\subset G$ such that every $g\in G\setminus \Mm$ is composite.
    Moreover, when $G=\on{SL}_n(\RR)$ the manifold $\Mm$ is given explicitly by
    \begin{equation}
        \Mm=\set{g\in \on{SL}_n(\RR):\det(g_{ii})= 0\text{ for some }i=1,\dots,d-1}
    \end{equation}
    where we denote $g_{ii}\in \on{M}_{i}(\RR)$ to be the $i$'th leading principal minor in $g$.
\end{claim}

Let $G$ be a simple Lie group.
Let $A<G$ be the Cartan torus with Lie algebra $\fa$ and closed positive Weyl chamber $\fa^+$.
Let $s \in \fa^+$ (it can be on the boundary), 
let $V$ the expanding subgroup, 
$\bar V$ the contracting subgroup with respect to the $s$-action, and 
and $G^{s}$ be the centralizer of $s$. 
Let $Q = \bar VG^s$ be the nonexpanding parabolic subgroup with respect to the $s$-action. 
For every $g\in G$ that can be written as $g = g_{-0}g_+$, for $g_{-0} \in Q$ and $g_+\in V$ are uniquely determined as a function of $g$.
Let $P\subseteq G$ be a the subgroup noncontracting by the action of $\fa^+$. 
Let $m_{P,1}$ be the noralised restriction of Haar measure on $P$ to a ball $P_1^{||}$ of radius $1$ in $P$ around the identity. 
Throughout, $|\cdot|$ is the norm fixed in Section~\ref{sec: nota}
and $q_G,n,C$ denote constants depending only on $G$, which may grow
from line to line. Recall $m_{P,1}$ is normalised Haar measure on the
unit ball $P_1:=P_1^{\|}$.

Everything below rests on the following lemma. It is stated for a
family of polynomials indexed by $g\in G$, which is exactly the shape in
which it will be used: the point is that the constant is polynomial in
$|g|$, uniformly over the family.

\begin{lemma}\label{lem: uniform sublevel}
Let $\{f_g\}_{g\in G}$ be a family of polynomial functions on $P_1$ of
degree at most $d$, whose coefficients are polynomial functions of the
matrix entries of $g$, and suppose $f_g\not\equiv0$ on $P_1$ for every
$g\in G$. Then there are $C,q_G>0$ and $n\geq1$ such that
\[
        m_{P,1}\bigl(\{p\in P_1:|f_g(p)|\leq t\}\bigr)
        \leq
        C\bigl(t\,|g|^{q_G}\bigr)^{1/n}
        \qquad\text{for all }t>0,\ g\in G .
\]
\end{lemma}

\begin{proof}
Write $\|f\|_{P_1}:=\sup_{P_1}|f|$. By the sublevel-set inequality for
polynomials of bounded degree on a compact set with nonempty interior
\cite[Thm.~8]{CarberyWright} --- equivalently, by the fact that such
polynomials are $(C,1/d)$-good in the sense of
\cite[\S3]{KleinbockMargulis} --- there are $C$ and $n$, depending only
on $d$ and $P_1$, with
\begin{equation}\label{eq: remez}
        m_{P,1}\bigl(\{p\in P_1:|f(p)|\leq\varepsilon\|f\|_{P_1}\}\bigr)
        \leq C\varepsilon^{1/n},
        \qquad \varepsilon>0 .
\end{equation}
It therefore suffices to bound $\|f_g\|_{P_1}$ from below by
$c|g|^{-q_G}$. Put
\[
        \Phi(R):=\inf_{|g|\leq R}\|f_g\|_{P_1}.
\]
The set $\{|g|\leq R\}$ is compact and $g\mapsto\|f_g\|_{P_1}$ is
continuous and, by hypothesis, strictly positive; hence $\Phi(R)>0$ for
every $R$. Moreover $\Phi$ is semialgebraic, by Tarski--Seidenberg
applied to the formula defining it. A positive semialgebraic function of
one variable admits a Puiseux expansion at infinity, so
$\Phi(R)\geq cR^{-q_G}$ for $R\geq1$ \cite[\S2.6]{BCR}. Applying
\eqref{eq: remez} with $\varepsilon=t/\|f_g\|_{P_1}\leq t|g|^{q_G}/c$
finishes the proof.
\end{proof}

\begin{remark}
The normalisation by $\|f_g\|_{P_1}$ in \eqref{eq: remez} is not
cosmetic: rescaling $f_g$ by a constant changes no zero set, so no bound
of the form $m(\{|f_g|\leq t\})\leq Ct^{1/n}$ can hold with $C$
independent of the size of $f_g$. This is the only place where the
polynomial dependence on $|g|$ is produced.
\end{remark}

We will also need: 
\begin{lemma}\label{lem: phi factors}
Let $g_1,g_2\in G$ with $g_1\in QV$ and put $v_1:=(g_1)_+$. Then for
every $v\in V$ for which either side is defined,
\[
        (g_1vg_2)_+=\bigl((v_1v)g_2\bigr)_+ .
\]
Consequently
\[
        v\longmapsto (g_1g_2)_+^{-1}(g_1vg_2)_+
\]
is the composition of the left translation $v\mapsto v_1v$ of $V$, the
map $\rho_{g_2}:u\mapsto (ug_2)_+$, and a further left translation.
Each factor is a diffeomorphism between open subsets of $V$; in
particular $D_{g_1,g_2}(v_0)$ is invertible at every $v_0$ at which it
is defined.
\end{lemma}

\begin{proof}
Write $g_1=q_1v_1$ with $q_1\in Q$. Then $g_1vg_2=q_1(v_1vg_2)$, and
$Q$-factors do not affect the $V$-coordinate, which gives the identity.
The map $\rho_{g_2}$ is the restriction to the big cell of the right
$g_2$-action on $Q\backslash G$; its inverse is $u\mapsto(ug_2^{-1})_+$,
since $ug_2=qu'$ implies $u'g_2^{-1}=q^{-1}u$.
\end{proof}

We now move on to the statements we will call upon in \S \ref{sec: FF}.

\begin{claim}\label{cl: lu measure bound}
For every $g\in G$,
\[
m_{P,1}\left(\left\{p\in P_1:(pgp^{-1})_+ \text{ exists}\right\}\right)=1.
\]
Moreover, there exist $n\geq 1$ such that for every $x>0$,
\[
m_{P,1}\left(\left\{p\in P_1: |(pgp^{-1})_+|<x\right\}\right)
\geq 1-|g|^{q_G}x^{-1/n}.
\]
\end{claim}
\begin{proof}
Write $\Omega:=\overline VP$ for the open Bruhat cell; by the Remark
following Claim \ref{cl: stuff decompose}, $\Omega\subset QV$, so
$g\in\Omega$ implies that $g_+$ is defined. Hence
\[
        \Omega_g:=\{p\in P_1:(pgp^{-1})_+\text{ is defined}\}
        \supset
        \{p\in P_1:pgp^{-1}\in\Omega\},
\]
and the right-hand side is the trace on $P_1$ of a Zariski-open subset
of $P$, since $p\mapsto pgp^{-1}$ is algebraic.

That open set is nonempty. Indeed $\Omega$ and $\Omega g^{-1}$ are both
Zariski-open and dense in $G$, so we may pick
$x\in\Omega\cap\Omega g^{-1}$ and write $x=\overline v_1p_1$,
$xg=\overline v_2p_2$. Then
\[
        p_1gp_1^{-1}
        =\overline v_1^{-1}\overline v_2\,p_2p_1^{-1}\in\Omega .
\]
A nonempty Zariski-open subset of $P$ has full Haar measure, so
$m_{P,1}(\Omega_g)=1$.

For the quantitative statement, the projection $\Omega\to V$,
$g\mapsto g_+$, is a morphism, so on $\Omega_g$ we may write
\[
        F_g(p):=(pgp^{-1})_+=\frac{A_g(p)}{B_g(p)},
\]
with $A_g,B_g$ of degree bounded in terms of $G$ and with coefficients
polynomial in the entries of $g$; and $B_g\not\equiv0$ on $P_1$ by the
previous paragraph. Since $P_1$ is a fixed compact set,
$\sup_{P_1}|A_g|\leq C|g|^{q_G}$. Hence
\[
        |F_g(p)|\geq x
        \quad\Longrightarrow\quad
        |B_g(p)|\leq C|g|^{q_G}/x,
\]
and Lemma \ref{lem: uniform sublevel} applied to $B_g$ with
$t=C|g|^{q_G}/x$ gives
\[
        m_{P,1}\bigl(\{p\in P_1:|F_g(p)|\geq x\}\bigr)
        \leq
        C\bigl(|g|^{q_G}x^{-1}\bigr)^{1/n}
        \leq
        |g|^{q_G}x^{-1/n}
\]
after enlarging $q_G$.
\end{proof}

\begin{nota}
For every $g_1, g_2\in G$ and $v_0\in V_1$, denote by $D_{g_1, g_2}(v_0):\vV\to \vV$ the differential at $v_0$ of the map 
$$v\mapsto (g_1g_2)_+^{-1}(g_1vg_2)_+.$$

Moreover we define, for every $c>0$ the sets
\begin{align*}
    \Pp^+_c(g_1,g_2)=\{p\in P_1: 
    \|D_{pg_1p^{-1}, pg_2p^{-1}}^{-1}(v_0)\|_{\rm op}^{-1} \geq c\text{ for every $v_0\in V_1$ and }|(pg_1p^{-1})_+|\leq c^{-1}|pg_1p^{-1}|\}
\end{align*}
and 
\begin{align*}
    \Dd_c^+=\{(g_1,g_2)\in G^2: \|D_{g_1, g_2}^{-1}(v_0)\|_{\rm op}^{-1} \geq c\text{ for every $v_0\in V_1$},|(g_1)_+|\leq c^{-1}|g_1|\}
\end{align*}
so that 
\begin{align*}
    \Pp^+_c(g_1,g_2)=\{p\in P_1: (pg_1p^{-1},pg_2p^{-1})\in \Dd_c^+\}.
\end{align*}
\begin{remark}
    To simplify the proof, and since the proof remains the same in this case, we assume that the $v_0$ from the definition of $D_{g_1,g_2}(v_0)$ is always equal to $e$ (the trivial element), and denote $D_{g_1,g_2}:=D_{g_1,g_2}(e)$.
\end{remark}
\end{nota}
\begin{claim}\label{cl: lu bounded below differential}
    For every $g_1, g_2\in G$, 
    \[m_{P, 1}(\{p\in P: D_{pg_1p^{-1}, pg_2p^{-1}}\text{ is invertible}\}) = 1,\]
    and for some $n \ge 1$ and $C > 0$ and every $x > 0$ we have
    \[m_{P, 1}(\Pp_{(|g_1||g_2|)^{-q_G}/x}^+(g_1,g_2)) \ge 1-Cx^{-1/n}.\]
\end{claim}

\begin{proof}
By Claim \ref{cl: lu measure bound} applied to $g_1$ and to $g_1g_2$,
for $m_{P,1}$-almost every $p\in P_1$ both $pg_1p^{-1}$ and
$pg_1g_2p^{-1}$ lie in $QV$. For such $p$, Lemma \ref{lem: phi factors}
shows that $D_{pg_1p^{-1},pg_2p^{-1}}$ is invertible. This proves the
first assertion, with no further argument.

For the quantitative assertion, set
$\Psi(p):=D_{pg_1p^{-1},pg_2p^{-1}}$. By algebraicity of $g\mapsto g_+$
on $QV$, the entries of $\Psi$ are rational in $p$ with numerator and
denominator of bounded degree and coefficients polynomial in the entries
of $g_1,g_2$; hence
\[
        \det\Psi(p)=\frac{A(p)}{B(p)},
        \qquad
        \sup_{P_1}|A|,\ \sup_{P_1}|B|\ \leq\ C(|g_1||g_2|)^{q_G},
\]
with $A\not\equiv0$ by the previous paragraph. Applying Lemma
\ref{lem: uniform sublevel} to $A$ and using Cramer's rule together with
the upper bound on the entries of $\Psi$, we obtain, after enlarging
$q_G$,
\[
        m_{P,1}\Bigl(\bigl\{p\in P_1:
        \|\Psi(p)^{-1}\|_{\rm op}^{-1}
        <(|g_1||g_2|)^{-q_G}/x\bigr\}\Bigr)
        \leq Cx^{-1/n}.
\]

It remains to control the second condition in the definition of
$\Pp_c^+$. Since $P_1$ is a fixed compact set,
$|pg_1p^{-1}|\geq c_0|g_1|$ for every $p\in P_1$; so it suffices to
bound $|(pg_1p^{-1})_+|$ by $c^{-1}c_0|g_1|$, and Claim
\ref{cl: lu measure bound} applied to $g_1$ with the parameter
$c^{-1}c_0|g_1|$ bounds the measure of the exceptional set by
$|g_1|^{q_G}(c^{-1}c_0|g_1|)^{-1/n}$. With
$c=(|g_1||g_2|)^{-q_G}/x$ and $q_G$ enlarged once more, this is at most
$Cx^{-1/n}$. Combining the two estimates gives
\[
        m_{P,1}\bigl(\Pp^+_{(|g_1||g_2|)^{-q_G}/x}(g_1,g_2)\bigr)
        \geq 1-Cx^{-1/n}.
\]
\end{proof}

\begin{definition}\label{def: dc}
    Given $\overline g=(g_1,\dots,g_n)\in G^n$ we say that $\overline g$ is a $D,c$-tuple if for every $i=1,\dots,n-1$ we have $(g_1\cdots g_i,g_{i+1})\in \Dd_c^+$.
    The set of $(D,c)$-tuples is denoted $(D,c)^n$.
\end{definition}

The proof of the following claim is very similar to the proofs of Claims \ref{cl: lu bounded below differential} and \ref{cl: lu measure bound} and uses the same methods. 
We thus omit its proof.
\begin{claim}[left stability of $\Dd_c^+$]\label{cl: good tuples are stable}
    For every $c>0$ and for every $(g_1,g_2)\in \Dd_c^+\cap (G_R^{||})^2$ there exists $q_G>0$ dependent only on $G$ such that for all $g\in G_{c/R^{q_{G}}}$ we have $(gg_1,g_2)\in \Dd_{c/2}^+$.
    As a result, the set of $D,c$-tuples is stable in the following sense
    \begin{align*}
        \left(G_{c/R^{nq_{G}}}\times \cdots \times G_{c/R^{nq_{G}}} \cdot (D,c)^n\right)\cap (G_R^{||})^n\subset (D,c/2)^{n}.
    \end{align*}
\end{claim}

\begin{claim}\label{cl: conjugating to a good tuple}
    Given any $g_1,\dots,g_\ell\in G_R^{||}$ and $c>0$, we can find a subset $\Pp^+$ of $P_1$ measure at least $1-C R^{3q_G\ell} c^{1/n}$ such that for every $p\in \Pp^+$ we have that $p\overline gp^{-1}$ is a $D,c$-tuple.
\end{claim}
\begin{proof}
By Definition \ref{def: dc}, $p\overline gp^{-1}\in(D,c)^\ell$ if and
only if $(pap^{-1},pbp^{-1})\in\Dd_c^+$ for each of the $\ell-1$ pairs
\[
        (a,b)\in\Ff(\overline g)
        :=\bigl\{(g_1\cdots g_i,\,g_{i+1}):1\leq i\leq\ell-1\bigr\},
\]
and each such pair lies in $(G_{R^\ell}^{\|})^2$. Fix one pair and
apply Claim \ref{cl: lu bounded below differential} with the parameter
$x=(|a||b|)^{-q_G}/c$:
\[
        m_{P,1}\bigl(\{p\in P_1:(pap^{-1},pbp^{-1})\notin\Dd_c^+\}\bigr)
        \leq Cx^{-1/n}
        =C\,c^{1/n}(|a||b|)^{q_G/n}
        \leq C\,c^{1/n}R^{2q_G\ell/n}.
\]
Summing over the $\ell-1$ pairs and using $\ell\leq R^{q_G\ell}$ for
$R\geq2$,
\[
        m_{P,1}\bigl(\{p\in P_1:p\overline gp^{-1}\notin(D,c)^\ell\}\bigr)
        \leq C R^{3q_G\ell}c^{1/n},
\]
which is the assertion.
\end{proof}

\section{Markov and Fubini computations}

\begin{nota}
    For this section, we fix $X$ to be a $G$-space and let $T$ be an operator on $X$.
    Also, we fix a finite partition $\tau$ of $X$ and denote its size by $\# \tau$.
    We denote $\Ee_T^n$ to be the measure on $\inn{T}$ given by $\tfrac{1}{n}\sum_{i=0}^{n-1}\delta_{T^i}$ and moreover $\Ee_T^{\mp}=\frac{1}{2n}\sum_{i=1-n}^{n-1}\delta_{T^i}$.
\end{nota}

\begin{definition}
    Denote for every $x\in X$ and $m\in \NN$:
    \begin{equation}
        S_m(x)=\sum_{k=0}^{m-1}d_\tau(\Ee_T^{2^k}*\delta_{x},\Ee_T^{-2^k}*\delta_{x}).
    \end{equation}
\end{definition}

\begin{corollary}\label{cor: joint cor}
    The following inequality holds for every $x_0\in X$:
    \begin{equation}
        \tfrac{1}{2^m}\abs{\{n=0,\dots,2^m-1:S_{m}(T^nx)> (\tfrac{3}{2}\log \#\tau )^{3/2}\}}\leq (\tfrac{3}{2}\log \#\tau )^{-1/2}.
    \end{equation}
    Moreover, if $S_m(T^nx)\leq (\tfrac{3}{2}\log \#\tau )^{3/2}$, then:
    \begin{equation}
        \tfrac{1}{m}\abs{\{k=0,\dots,m-1:d_\tau(\Ee_T^{2^k}*\delta_{T^nx},\Ee_T^{-2^k}*\delta_{T^nx})\geq (\tfrac{3m}{2}\log \#\tau )^{-1/2}\}}\leq \tfrac{1}{m^{1/2}}(\tfrac{3}{2}\log \#\tau)^{2}.
    \end{equation}
\end{corollary}
\begin{proof}
    Otherwise, there is a proportion of at least $(\tfrac{3}{2}\log \#\tau )^{-1/2}$ of $n$'s between $0$ and $2^m-1$ for which $S_m(T^nx)> (\tfrac{3}{2}\log \#\tau )^{3/2}$.
    This implies that:
    \begin{equation}
        \tfrac{1}{2^m}\sum_{n=0}^{2^m-1}S_m(T^nx)> (\tfrac{3}{2}\log \#\tau )^{3/2}\cdot (\tfrac{3}{2}\log \#\tau )^{-1/2}=\tfrac{3}{2}\log \#\tau.
    \end{equation}
    However, since $S_m(T^nx)$ is exactly (by definition) the term in the sum appearing in the previous claim, this would contradict it. 

    To prove the 'moreover' part, note that if there are at least $m^{1/2}(\tfrac{3}{2}\log \#\tau )^{2}$ numbers $k=0,\dots,m-1$ for which 
    \begin{equation}
        d_\tau(\Ee_T^{2^k}*\delta_{T^nx},\Ee_T^{-2^k}*\delta_{T^nx})\geq (\tfrac{3m}{2}\log \#\tau )^{-1/2},
    \end{equation}
    then this implies that 
    \begin{equation}
        S_m(T^nx)=\sum_{k=0}^{m-1}d_\tau(\Ee_T^{2^k}*\delta_{T^nx},\Ee_T^{-2^k}*\delta_{T^nx})> (\tfrac{3}{2}\log \#\tau )^{3/2}
    \end{equation}
    which is once again in contradiction with the assumption $S_m(T^nx)\leq (\tfrac{3}{2}\log \#\tau )^{3/2}$.
\end{proof}

\begin{claim}\label{cl: reversing using fubini+markov}
    If $\lambda$ is some probability measure on $X$, then there exists a number $n=0,\dots,2^m-1$ and a subset $\Aa \subset \{0,\dots, m-1\}$ of proportion at least $1-((\tfrac{3}{2}\log \#\tau )^{-1/4}+\tfrac{1}{m}(\tfrac{3}{2}\log \#\tau )^{2})^{1/2}$ such that for every $k\in \Aa$ there exists a subset $X_0(k)\subset X$ of $\lambda$-measure at least  $1-((\tfrac{3}{2}\log \#\tau )^{-1/4}+\tfrac{1}{m}(\tfrac{3}{2}\log \#\tau )^{2})^{1/2}$ such that for every $x\in X_0(k)$ we have:
    \begin{equation}
    d_\tau(\Ee_T^{2^k}*\delta_{T^nx},\Ee_T^{-2^k}*\delta_{T^nx})\leq (\tfrac{3m}{2}\log \#\tau )^{-1/2}.
    \end{equation}
\end{claim}

\begin{proof}
    We will use Corollary \ref{cor: joint cor} and Fubini Theorem.
    Denote for every natural number $\ell$ the measure $m_{[\ell]}$ as the uniform probability measure on $\{0,\dots,\ell-1\}$.
    Then by the first part of the corollary and by Fubini we have:
    \begin{equation}
        m_{[2^m]}\times \lambda(\{(n,x)\in [2^m-1]\times X:S_m(T^nx)>(\tfrac{3}{2}\log \#\tau )^{3/2}\})\leq (\tfrac{3}{2}\log \#\tau )^{-1/2}.
    \end{equation}
    Again by Fubini, this implies:
    \begin{equation}
        \int_{[2^m-1]}\lambda(\{S_m(T^nx)>(\tfrac{3}{2}\log \#\tau )^{3/2}\})dm_{[2^m]}(n)\leq (\tfrac{3}{2}\log \#\tau )^{-1/2}.
    \end{equation}
    By Markov inequality we deduce that:
    \begin{equation}
        m_{[2^m]}(\{\lambda(\{S_m(T^nx)>(\tfrac{3}{2}\log \#\tau )^{3/2}\})>(\tfrac{3}{2}\log \#\tau )^{-1/4}\})\leq (\tfrac{3}{2}\log \#\tau )^{-1/4}
    \end{equation}
    which says, in other words, that there exists a subset $\Aa'\subset [2^m-1]$ of measure at least $(\tfrac{3}{2}\log \#\tau )^{-1/4}$ such that for each $n\in \Aa'$ we have $\lambda(\{S_m(T^nx)>(\tfrac{3}{2}\log \#\tau )^{3/2}\})\leq (\tfrac{3}{2}\log \#\tau )^{-1/4}$.
    In particular, take some $n\in \Aa'$.
    Denote $X_0=\{x\in X:S_m(T^nx)\leq (\tfrac{3}{2}\log \#\tau )^{3/2}\}$.
    Then we have proved that $\lambda(X_0)\geq 1-(\tfrac{3}{2}\log \#\tau )^{-1/4}$.
    By the 'moreover' part of Corollary \ref{cor: joint cor}, since for every $x\in X_0$ we have $S_m(T^nx)\leq (\tfrac{3}{2}\log \#\tau )^{3/2}$, we know that: 
    \begin{equation}
        m_{[m]}(\{k\in [m-1]:d_\tau(\Ee_T^{2^k}*\delta_{T^nx},\Ee_T^{-2^k}*\delta_{T^nx})\geq (\tfrac{3m}{2}\log \#\tau )^{-1/2}\})\leq \tfrac{1}{m^{1/2}}(\tfrac{3}{2}\log \#\tau )^{2}.
    \end{equation}
    Since $\lambda(X_0)\geq 1-(\tfrac{3}{2}\log \#\tau )^{-1/4}$ we can deduce:
    \begin{equation}
        m_{[m]}\times \lambda (\{(k,x):d_\tau(\Ee_T^{2^k}*\delta_{T^nx},\Ee_T^{-2^k}*\delta_{T^nx})\geq (\tfrac{3m}{2}\log \#\tau )^{-1/2}\})\leq (\tfrac{3}{2}\log \#\tau )^{-1/4}+\tfrac{1}{m^{1/2}}(\tfrac{3}{2}\log \#\tau )^{2}.
    \end{equation}
    By Fubini, this shows that 
    \begin{align*}
        \int_{[m-1]}\lambda(\{x\in X: d_\tau(\Ee_T^{2^k}*\delta_{T^nx},\Ee_T^{-2^k}*\delta_{T^nx})\geq (\tfrac{3m}{2}\log \#\tau )^{-1/2}\})dm_{[m]}(k)&\\ \leq (\tfrac{3}{2}\log \#\tau )^{-1/4}+\tfrac{1}{m^{1/2}}(\tfrac{3}{2}\log \#\tau )^{2}.
    \end{align*}
    By Markov again, we obtain:
    \begin{align*}
        m_{[m]}(\{k:\lambda(\{x\in X: d_\tau(\Ee_T^{2^k}*\delta_{T^nx},\Ee_T^{-2^k}*\delta_{T^nx})\geq (\tfrac{3m}{2}\log \#\tau )^{-1/2}\})&\\\geq ((\tfrac{3}{2}\log \#\tau )^{-1/4}+\tfrac{1}{m^{1/2}}(\tfrac{3}{2}\log \#\tau )^{2})^{1/2}\})\leq ((\tfrac{3}{2}\log \#\tau )^{-1/4}+\tfrac{1}{m^{1/2}}(\tfrac{3}{2}\log \#\tau )^{2})^{1/2}
    \end{align*}
    and therefore there exists a subset $\Aa\subset [m-1]$ of measure at least $1-((\tfrac{3}{2}\log \#\tau )^{-1/4}+\tfrac{1}{m^{1/2}}(\tfrac{3}{2}\log \#\tau )^{2})^{1/2}$ such that for every $k\in \Aa$, we have:
    \begin{equation}
        \lambda(\{x\in X: d_\tau(\Ee_T^{2^k}*\delta_{T^nx},\Ee_T^{-2^k}*\delta_{T^nx})\geq (\tfrac{3m}{2}\log \#\tau )^{-1/2}\})\leq ((\tfrac{3}{2}\log \#\tau )^{-1/4}+\tfrac{1}{m^{1/2}}(\tfrac{3}{2}\log \#\tau )^{2})^{1/2}
    \end{equation}
    namely for every $k\in \Aa$ there exists $X_0(k)\subset X$ of measure at least $1-((\tfrac{3}{2}\log \#\tau )^{-1/4}+\tfrac{1}{m^{1/2}}(\tfrac{3}{2}\log \#\tau )^{2})^{1/2}$ such that for every $x\in X_0(k)$ we have:
    \begin{equation}
        d_\tau(\Ee_T^{2^k}*\delta_{T^nx},\Ee_T^{-2^k}*\delta_{T^nx})\leq (\tfrac{3m}{2}\log \#\tau )^{-1/2}
    \end{equation}
    as desired.
\end{proof}

\section{Upgrade to parabolic}

Recall that $G$ is simple with finite centre, that $s\in\fa^+\setminus\{0\}$
is a singular direction, and that $V$ is the expanding horospherical
subgroup for $s$; equivalently $Q^+:=G^sV$ is the parabolic subgroup not
contracted by $s$ and $V=R_u(Q^+)$. Only the last property is used
below, so we state the result for the unipotent radical of an arbitrary
proper parabolic. Since $Q^+\neq G$ we have $V\neq\{e\}$.

We regard $G$ as the group of real points of a linear algebraic group,
which is legitimate because $Z(G)$ is finite.

\begin{theorem}\label{thm: upgrade to parabolic}
Let $Q^+<G$ be a proper parabolic subgroup and $V=R_u(Q^+)$. Then every
proper closed subgroup $H<G$ containing $V$ is contained in a maximal
parabolic subgroup of $G$.
\end{theorem}

The two classical inputs are the following. First, if $L<G$ is a
reductive algebraic subgroup, then some Cartan involution of $G$
preserves $L$ \cite{Mostow55}. Second, if $U<G$ is a nontrivial
unipotent subgroup, then there is a proper parabolic subgroup $P<G$ with
$U\subset R_u(P)$ and $N_G(U)\subset P$ \cite[Cor.~3.2]{BorelTits71};
see also \cite[Rem.~17.16]{MalleTesterman}.

\begin{lemma}\label{lem: opposite unipotents generate}
Let $Q^+<G$ be a proper parabolic subgroup, $U^+=R_u(Q^+)$, and let
$\vartheta$ be any Cartan involution of $G$. Then $\vartheta(Q^+)$ is the
parabolic subgroup opposite to $Q^+$, and $U^+$ together with
$U^-:=\vartheta(U^+)$ generates $G$.
\end{lemma}

\begin{proof}
Let $\theta$ also denote the involution of $\fg$, let $K=G^\vartheta$,
and fix a minimal parabolic $P_0$ built from a $\vartheta$-stable Cartan
subspace, so that $\vartheta(\mathfrak{q}_0)=\mathfrak{q}_0^-$ for every standard
parabolic subalgebra $\mathfrak{q}_0\supset\fp_0$. Write
$\mathfrak{q}^+=\Ad(g)\mathfrak{q}_0$ for a suitable standard $\mathfrak{q}_0$ and $g\in G$, and
$g=kp$ with $k\in K$, $p\in P_0\subset Q_0$ by the Iwasawa
decomposition. Then $\mathfrak{q}^+=\Ad(k)\mathfrak{q}_0$, and since $\vartheta$ fixes $k$,
\[
        \vartheta(\mathfrak{q}^+)=\Ad(k)\,\vartheta(\mathfrak{q}_0)=\Ad(k)\mathfrak{q}_0^-,
\]
which is opposite to $\Ad(k)\mathfrak{q}_0=\mathfrak{q}^+$. This proves the first
assertion, and gives the decomposition
$\fg=\mathfrak{u}^-\oplus\mathfrak{l}\oplus\mathfrak{u}^+$, where $\mathfrak{l}$ is the common
$\vartheta$-stable Levi factor.

Let $\mathfrak{h}$ be the subalgebra generated by $\mathfrak{u}^+$ and $\mathfrak{u}^-$. It is
stable under $\mathrm{ad}\mathfrak{u}^{\pm}$ because it is a subalgebra containing
$\mathfrak{u}^\pm$. It is stable under $\mathrm{ad}\mathfrak{l}$ because $\mathfrak{l}$ normalises $\mathfrak{u}^+$
and $\mathfrak{u}^-$, hence normalises the subalgebra they generate, by the
Jacobi identity and induction on bracket length. So $[\fg,\mathfrak{h}]\subset\mathfrak{h}$
and $\mathfrak{h}$ is a nonzero ideal of the simple algebra $\fg$, i.e.
$\mathfrak{h}=\fg$. The subgroup generated by $U^+$ and $U^-$ is therefore a
connected subgroup with Lie algebra $\fg$, hence open, hence equal to
$G$.
\end{proof}

\begin{proof}[Proof of Theorem \ref{thm: upgrade to parabolic}]
Put $\mathfrak{h}:=\mathrm{Lie}(H)$. Since $V$ is connected and $V\subset H$, we have
$\mathfrak{v}:=\mathrm{Lie}(V)\subset\mathfrak{h}$, so $\mathfrak{h}\neq0$; and $\mathfrak{h}\neq\fg$ because $H$ is
proper and $G$ connected. As $\fg$ is simple, $\mathfrak{h}$ is not an ideal, so
\[
        N:=N_G(\mathfrak{h})
\]
is a \emph{proper algebraic} subgroup of $G$. It contains $H$, because
$H$ normalises $H^\circ$, and hence contains $V$.

Suppose $N$ were reductive. By \cite{Mostow55} there is a Cartan
involution $\vartheta$ of $G$ with $\vartheta(N)=N$, so $N$ contains both
$V$ and $\vartheta(V)$; by Lemma \ref{lem: opposite unipotents generate}
these generate $G$, contradicting properness of $N$.

Hence $R_u(N)\neq\{e\}$. Since $N$ normalises $R_u(N)$, the Borel--Tits
theorem \cite[Cor.~3.2]{BorelTits71} provides a proper parabolic
subgroup $P<G$ with
\[
        H\subset N\subset N_G\bigl(R_u(N)\bigr)\subset P .
\]
Finally every proper parabolic subgroup is contained in a maximal one.
\end{proof}

\bibliographystyle{plain}
\bibliography{BibErg}{}

\end{document}